\documentclass[reqno,11pt]{amsart}

\usepackage[top=2.0cm,bottom=2.0cm,left=3cm,right=3cm]{geometry}
\usepackage{amsthm,amsmath,amssymb,dsfont}
\usepackage{mathrsfs,amsfonts,functan,extarrows,mathtools}
\usepackage[colorlinks]{hyperref}

\usepackage{marginnote}
\usepackage{xcolor}
\usepackage{stmaryrd}
\usepackage{esint}
\usepackage{graphicx}
\usepackage{bm}

\newtheorem{theorem}{Theorem}[section]
\newtheorem{proposition}{Proposition}[section]

\newtheorem{corollary}{Corollary}[section]
\newtheorem{remark}{Remark}[section]
\newtheorem{lemma}{Lemma}[section]

\numberwithin{equation}{section}
\allowdisplaybreaks

\def\d{\mathrm{d}}

\def\R{\mathbb{R}}
\def\T{\mathbb{R}}

\def\eps{\varepsilon}
\def\div{\mathrm{div}}

\newcounter{wronumber}

\begin{document}
\title[From VML to incompressible NSFM with Ohm's law]
			{From Vlasov-Maxwell-Landau system with Coulomb potential to two-fluid incompressible Navier-Stokes-Fourier-Maxwell system with Ohm's law}

\author[Yuanjie Lei]{Yuanjie Lei}
\address[Yuanjie Lei]
{\newline School of Mathematics and Statistics, Huazhong University of Sciences and Technology, Wuhan, 430074, P. R. China}
\email{leiyuanjie@hust.edu.cn}

\author[Ning Jiang]{Ning Jiang}
\address[Ning Jiang]{\newline School of Mathematics and Statistics, Wuhan University, Wuhan, 430072, P. R. China}
\email{njiang@whu.edu.cn}

\thanks{\today}

\maketitle

\begin{abstract}
   In the diffusive regime, we obtain the uniform in Knudsen number estimate on the two-species Vlasov-Maxwell-Landau system with Coulomb potential around the global equilibrium. As a consequence, we justify the limit to the two-fluid incompressible Navier-Stokes-Fourier-Maxwell (NSFM) system with Ohm's law. This is the first result on the global in time uniform in Knudsen number estimate on the Boltzmann and Landau type equations with Coulomb potential coupled Vlasov-Maxwell equations.  \\


\end{abstract}





\section{Introduction.}\label{Sec:Introduction}

\subsection{Vlasov-Maxwell-Landau system}
The two-species Vlasov-Maxwell-Landau system (in brief, VML) is a fundamental system in kinetic theory.  Precisely,  VML system consists the following equations:
\begin{equation}\label{VMB-0}
    \left\{
    \begin{array}{c}
      \partial_t F^+ + v\cdot \nabla_x F^+ + \tfrac{q^+}{m^+}(E + v \times B)\cdot\nabla_v F^+ = Q(F^+, F^+) + Q(F^+, F^-)\,,\\[2mm]
            \partial_t F^- + v\cdot \nabla_x F^- - \tfrac{q^-}{m^-}(E + v \times B)\cdot\nabla_v F^- = Q(F^-, F^-) + Q(F^-, F^+)\,,\\[2mm]
     {\mu}_0\eps_0\partial_t E - \nabla_x \times B = -{\mu}_0\int_{\mathbb{R}^3}(q^+F^+- q^- F^-)v\,\mathrm{d}v\,,\\[2mm]
      \partial_t B + \nabla_{\!x} \times E = 0\,,\\[2mm]
      \div_x E = \tfrac{1}{\eps_0}\int_{\mathbb{R}^3}(q^+F^+- q^- F^-)\,\mathrm{d}v\,,\quad\!\mbox{and}\quad\!
      \div_x B = 0\,.
        \end{array}
  \right.
\end{equation}

 The collision between charged particles is given by
\begin{equation*}
   \begin{split}
     Q(F_1,F_2)=&\nabla_{ v}\cdot\int_{\mathbf{R}^3}\Phi( v- v')
     \Big\{F_1( v')\nabla_{ v}F_2( v)-\nabla_{ v'}F_1( v')F_2( v)\Big\}d v'\\
     =&\sum_{i,j=1}^3\partial_i\int_{\mathbf{R}^3}\Phi^{ij}( v- v')
     \Big\{F_1( v')\partial_jF_2( v)-\partial_jF_1( v)F_2( v')\Big\}d v',
   \end{split}
\end{equation*}
where $\partial_i=\frac{\partial}{\partial v_i}$ for $i=1,2,3$ and the non-negative $3\times 3$ matrix $\Phi(v)=\left(\Phi^{ij}( v)\right)_{3\times 3}$ is the fundamental Landau (or Fokker-Planck) kernel:
\begin{equation*}
  \Phi^{ij}( v)=\bigg(\delta_{ij}-\frac{ v_i v_j}{| v|^2}\bigg)| v|^{-1},
\end{equation*}
which corresponds to the Coulomb potential for the classical Landau operator which is originally found by Landau (1936).

\subsection{Incompressible NSFM limit of VML}

The goal of this paper is to justify the incompressible Navier-Stokes-Fourier-Maxwell (NSFM) equations from VML under some suitable scaling. Inspired from the corresponding scalings in Vlasov-Maxwell-Boltzmann (VMB) analogue considered by Arsenio and Saint-Raymond \cite{Arsenio-SRM-2016}, we rescale the two-species VML by the Kundsen number $\eps$, which is the ratio of the mean free path and the macroscopic length scale in the Navier-Stokes scaling (also called diffusive scaling):
\begin{equation}\label{VMB-F}
  \left\{
    \begin{array}{l}
      \partial_t F_\eps^\pm + \tfrac{1}{\eps} v \cdot \nabla_x F_\eps^\pm \pm \tfrac{1}{\eps} ( \eps E_\eps + v \times B_\eps ) \cdot \nabla_v F_\eps^\pm = \tfrac{1}{\eps^2} Q (F_\eps^\pm, F_\eps^\pm) + \tfrac{1}{\eps^2} Q (F_\eps^\pm , F_\eps^\mp) \,,\\[2mm]
      \partial_t E_\eps - \nabla_x \times B_\eps = - \tfrac{1}{\eps^2} \int_{\R^3} (F_\eps^+ - F_\eps^-) v \d v \,,\\[2mm]
      \partial_t B_\eps + \nabla_x \times E_\eps = 0 \,,\\[2mm]
      \div_x E_\eps = \tfrac{1}{\eps} \int_{\R^3} ( F_\eps^+ - F_\eps^- ) \d v \,,\\[2mm]
      \div_x B_\eps = 0
    \end{array}
  \right.
\end{equation}
on $\mathbb{R}^3 \times \R^3$, with initial data
\begin{equation}\label{IC-VMB-F}
  \begin{aligned}
    F_\eps^\pm (0, x, v) = F_\eps^{\pm, in} (x,v) \in \R \,, \quad E_\eps (0,x) = E_\eps^{in} (x) \in \R^3 \,, \quad B_\eps (0, x) = B_\eps^{in} (x) \in \R^3 \,.
  \end{aligned}
\end{equation}

It is well-known that the global equilibrium for the two-species VML is $[{M(v)},{M}(v)]$, where the normalized global {\em Maxwellian} is
\begin{equation}\label{GM}
 {M}(v) = (2 \pi)^{-\frac{3}{2}} \exp\left(- \tfrac{|v|^2}{2}\right) \,.
\end{equation}

The two-fluid incompressible NSFM with Ohm's law can be derived from the scaled VML system \eqref{VMB-F} by considering the fluctuation  $F_\eps^\pm (t,x,v) = {M}(v)  + \eps \sqrt{{M(v)}} f_\eps^\pm (t,x,v) $ when the Knudsen number $\eps$ tends to zero. This leads to the perturbed two-species VMB
\begin{equation}\label{VML-F}
  \left\{
    \begin{array}{l}
      \partial_t f_\eps + \tfrac{1}{\eps} \big[ v \cdot  \nabla_x f_\eps + q_0 ( \eps E_\eps + v \times B_\eps ) \cdot \nabla_v f_\eps \big] + \tfrac{1}{\eps^2} \mathscr{L} f_\eps - \tfrac{1}{\eps}  (E_\eps \cdot v) \sqrt{{M}}  q_1\\[2mm]
       \qquad \qquad = \tfrac{1}{2} q_0 (E_\eps \cdot v) f_\eps + \tfrac{1}{\eps} \mathscr{T} (f_\eps, f_\eps) \,,\\[2mm]
      \partial_t E_\eps - \nabla_x \times B_\eps = - \tfrac{1}{\eps} \int_{\R^3} f_\eps \cdot {q_1} v \sqrt{{M}} \d v \,,\\[2mm]
      \partial_t B_\eps + \nabla_x \times E_\eps = 0 \,,\\[2mm]
      \div_x E_\eps = \int_{\R^3} f_\eps \cdot {q_1} \sqrt{{M}} \d v \,, \ \div_x B_\eps = 0 \,,
    \end{array}
  \right.
\end{equation}
where $f_\eps = [ f_\eps^+ , f_\eps^- ]$ represents the vector in $\R^2$ with the components $f_\eps^\pm$, the $2 \times 2$ diagonal matrix $q_0 = \textit{diag} (1, -1)$, the vector $q_1 = [1, -1]$, the linearized collision operator $\mathscr{L}f$ and the nonlinear collision term $\mathscr{T} (f,f)$ are respectively defined by
\[\mathscr{L}f=[\mathscr{L}_+f,\mathscr{L}_-f],\quad\quad\quad\mathscr{T}(f,g)=[\mathscr{T}_+(f,g),\mathscr{T}_-(f,g)]\]
with
\begin{equation*}
\begin{aligned}
\mathscr{L}_\pm f =& -{M}^{-1/2}
\left\{{Q\left( {M},{M}^{1/2}(f_\pm+f_\mp)\right)+ 2Q\left( {M}^{1/2}f_\pm, {M}\right)}\right\},\\[2mm]
\mathscr{T}_{\pm}(f,g) =&{M}^{-1/2}Q\left({M}^{1/2}f_{\pm},{M}^{1/2}g_\pm\right)
+{M}^{-1/2}Q\left({M}^{1/2}f_{\pm},{M}^{1/2}g_\mp\right).
\end{aligned}
\end{equation*}
For the linearized Landau collision operator $\mathscr{L}$, it is well known, cf. \cite{Guo-CMP-02}, that it is non-negative and the null space $\mathcal{N}$ of $\mathscr{L}$ is given by
\begin{equation*}
{\mathcal{ N}}={\textrm{span}}\left\{[1,0] , [0,1], [v_i,v_i] (1\leq i\leq3),[|v|^2,|v|^2]\right\}{M}^{1/2}.
\end{equation*}
If we define ${\bf P}$ as the orthogonal projection from $L^2({\mathbb{R}}^3_ v)\times L^2({\mathbb{R}}^3_ v)$ to $\mathcal{N}$, then for any given function $f(t, x, v )\in L^2({\mathbb{R}}^3_ v)$, one has
\begin{equation*}
  {\bf P}f =\left\{{\rho_+(t, x)[1,0]+\rho_-(t, x)[0,1]+\sum_{i=1}^{3}u_i(t, x) [1,1]v_i+\theta(t, x)[1,1](| v|^2-3)}\right\}{M}^{1/2}
\end{equation*}
with
\begin{equation*}
  \rho_\pm=\int_{{\mathbb{R}}^3}{M}^{1/2}f_\pm \,\d v,\quad
  u_i=\frac12\int_{{\mathbb{R}}^3} v  _i {M}^{1/2}(f_++f_-)\,\d v,\quad
  \theta=\frac{1}{12}\int_{{\mathbb{R}}^3}(| v|^2-3){M}^{1/2}(f_++f_-) \,\d v.
\end{equation*}
Therefore, we have the following macro-micro decomposition with respect to the given global Maxwellian ${M}$:
\begin{equation}\label{macro-micro}
 f(t,x, v)={\bf P}f(t,x, v)+\{{\bf I}-{\bf P}\}f(t, x, v)\,,
\end{equation}
where ${\bf I}$ denotes the identity operator and ${\bf P}f$ and $\{{\bf I}-{\bf P}\}f$ are called the macroscopic and the microscopic component of $f(t,x,v)$, respectively.

In the framework of renormalized solutions, using the moment method, refined relative entropy and compactness methods, in \cite{Arsenio-SRM-2016} the following formal convergence was rigorously justified:
  \begin{equation}\label{NSFM-Limit}
    \begin{aligned}
      f_\eps \rightarrow ( \rho + \tfrac{1}{2} n ) \tfrac{{q_1} + {q_2}}{2} \sqrt{{M}} + ( \rho - \tfrac{1}{2} n ) \tfrac{{q_2} - {q_1}}{2} \sqrt{{M}} + u \cdot v {q_2} \sqrt{{M}} + \theta ( \tfrac{|v|^2}{2} - \tfrac{3}{2} ) \sqrt{{M}}\,,
    \end{aligned}
  \end{equation}
where the vector $q_2 = [1, 1]$ and $(\rho, n, u, \theta, E, B)$ satisfies the following two-fluid incompressible NSFM:
\begin{equation}\label{INSFM-Ohm}
  \left\{
    \begin{array}{l}
      \partial_t u + u \cdot \nabla_x u - {\mu} \Delta_x u + \nabla_x p = \tfrac{1}{2} ( n E + j \times B ) \,, \qquad \div_x \, u = 0 \,, \\ [2mm]
      \partial_t \theta + u \cdot \nabla_x \theta - \kappa \Delta_x \theta = 0 \,, \qquad\qquad\qquad\qquad\qquad\quad\ \, \rho + \theta = 0 \,, \\ [2mm]
      \partial_t E - \nabla_x \times B = - j \,, \qquad\qquad\qquad\qquad\qquad\qquad\ \ \ \, \div_x \, E = n \,, \\ [2mm]
      \partial_t B + \nabla_x \times E = 0 \,, \qquad\qquad\qquad\qquad\qquad\qquad\qquad \div_x \, B = 0 \,, \\ [2mm]
      \qquad \qquad j - nu = \sigma \big( - \tfrac{1}{2} \nabla_x n + E + u \times B \big) \,, \qquad\quad\,\ w = \tfrac{3}{2} n \theta \,,
    \end{array}
  \right.
\end{equation}
where the viscosity ${\mu}$, the heat conductivity $\kappa$ and the electrical conductivity $\sigma$ are given by
\begin{equation}\label{VHE-Coefficients}
  \begin{aligned}
    {\mu} = \tfrac{1}{10} \int_{\R^3} A : \widehat{A} M \d v \,, \quad \kappa = \tfrac{2}{15} \int_{\R^3} B \cdot \widehat{B} M \d v \quad \textrm{and} \quad \sigma = \tfrac{2}{3} \int_{\R^3} \Phi \cdot \widetilde{\Phi} M \d v \,.
  \end{aligned}
\end{equation}
For the derivation of \eqref{VHE-Coefficients}, i.e. the relation of  ${M}$, $\kappa$, $\sigma$ with $A$, $\widehat{A}$, $B$, $\widehat{B}$, $\Phi$ and $\widetilde{\Phi}$, see \cite{Arsenio-SRM-2016}. We will not give detailed formal derivation here. In our proof of Theorem \ref{Main-Thm-1}, we indeed provide how the NSFM (with all transport coefficients) can be derived from VML with Coulomb potential.

\subsection{Hydrodynamic limits of Boltzmann type equations}
One of the most important features of the Boltzmann equations (or more generally, kinetic equations) is their connections to the fluid equations. The so-called fluid regimes of the Boltzmann equations are those of asymptotic dynamics of the scaled Boltzmann equations when the Knudsen number $\eps$ is very small. Justifying these limiting processes rigorously has been an active research field from late 70's. Among many results obtained, the main contributions are the incompressible Navier-Stokes and Euler limits. There are two types of results in this field:

{\bf (Type-I)}, First obtaining the solutions of the scaled Boltzmann equation with {\em uniform} bounds in the Knudsen number $\eps$, then extracting a subsequence converging (at least weakly) to the solutions of the fluid equations as $\eps\rightarrow 0$.

{\bf (Type-II)}, First obtaining the solutions for the limiting fluid equations, then constructing a sequence of special solutions (near the Maxwellians) of the scaled Boltzmann equations for small Knudsen number $\eps$.

The key difference between the results of Type-I and  type-II are: in type-I, the solutions of the fluid equations are {\em not} known a priori, and are completely obtained from taking limits from the Boltzmann equations. In short, it is ``from kinetic to fluid"; In type-II, the solutions of the fluid equations are {\em known} first. In short, it is ``from fluid to kinetic". Usually, type-I results are harder to be achieved since it must obtain enough uniform (with respect to $\eps$) bounds for solutions of the scaled kinetic equations then compactness arguments give the convergence. This approach automatically provides the existence of both the original kinetic equations and the limiting macroscopic equations. On the other hand, type-II approach needs to employ the information of the limiting equations, then prove the solutions of the original kinetic equations with a {\em special} form (usually Hilbert expansion).

We remark that this classification of two approaches also appears in other asymptotical problems. For example, in their work on Kac's program \cite{Mischler-Mouhot}, Mischler and Mouhot called the type-I as ``bottom-up" and type-II as ``top-down". We quote their words here: {\em ...our answer is an ``inverse" answer inthe sense that our methodology is ``top-down" from the limit equation to the many-particle system rather than ``bottom-up" as was expected by Kac.}

The most successful achievement in type-I is the so-called BGL (named after Bardos, Golse and Levermore) program.  From late 80's, Bardos, Golse and Levermore initialized the program to justify Leray's solutions of the incompressible Navier-Stokes equations from DiPerna-Lions' renormalized solutions \cite{BGL-1991-JSP, BGL-CPAM1993}. They proved the first convergence result with five additional technical assumptions. After ten years efforts by Bardos, Golse, Levermore, Lions, Masmoudi and Saint-Raymond, see for example \cite{BGL3, LM3, LM4, GL}, the first complete convergence result without any additional compactness assumption was proved by Golse and Saint-Raymond in \cite{G-SRM-Invent2004} for cutoff Maxwell collision kernel, and in \cite{G-SRM2009} for hard cutoff potentials. Later on, it was extended by Levermore and Masmoudi \cite{LM} to include soft potentials. Recently Ars\'enio got the similar results for non-cutoff case \cite{Arsenio, Arsenio-Masmoudi} based on the existence result of \cite{AV-CPAM2002}. Furthermore, by Jiang, Levermore, Masmoudi and Saint-Raymond, these results were extended to bounded domain where the Boltzmann equation was endowed with the Maxwell reflection boundary condition \cite{Masmoudi-SRM-CPAM2003, JLM-CPDE2010, JM-CPAM2017}, based on the solutions obtained by Mischler \cite{mischler2010asens}.

Another direction in type-I is in the context of classical solutions. The first work in this type is Bardos-Ukai \cite{b-u}. They started from the scaled Boltzmann equation for cut-off hard potentials, and proved the global existence of classical solutions $g_\eps$ uniformly in $0< \eps <1$. The key feature of Bardos-Ukai's work is that they only need the smallness of the initial data, and did not assume the smallness of the Knudsen number $\eps$ in uniform estimate. After having the uniform in $\eps$ solutions $g_\eps$, taking limits can provide a classical solution of the incompressible Navier-Stokes equations with small initial data. Bardos-Ukai's approach heavily depends on the sharp estimates from the spectral analysis on the linearized Boltzmann operator $\mathcal{L}$, and the semigroup method (the semigroup generated by the scaled linear operator $\eps^{-2}\mathcal{L}+\eps^{-1}v\cdot\nabla_{\!x}$). It seems that it is hardly extended to cutoff soft potential, and even harder for the non-cutoff cases, since it is well-known that the operator $\mathcal{L}$ has continuous spectrum in those cases. On the torus, semigroup approach was used by Briant \cite{Briant-JDE-2015} and Briant, Merino-Aceituno and Mouhot \cite{BMM-arXiv-2014} to prove incompressible Navier-Stokes limit by employing the functional analysis breakthrough of Gualdani-Mischler-Mouhot \cite{GMM}. Again, their results are for cut-off kernels with hard potentials. Recently, there is type-I convergence result on the incompressible Navier-Stokes limit of the Boltzmann equation by the first author and his collaborators. In \cite{JXZ-Indiana2018}, the uniform in $\eps$ global existence of the Boltzmann equation with or without cutoff assumption was obtained and the global energy estimates were established.

Most of the type-II results are based on the Hilbert expansion  and obtained in the context of classical solutions. It was started from Nishida and Caflisch's work on the compressible Euler limit \cite{Nishida, Caflisch, KMN}. Their approach was revisitied by Guo, Jang and Jiang, combining with nonlinear energy method to apply to the acoustic limit \cite{GJJ-KRM2009, GJJ-CPAM2010, JJ-DCDS2009}. After then this process was used for the incompressible limits, for examples, \cite{DEL-89} and \cite{Guo-2006-CPAM}. In \cite{DEL-89}, De Masi-Esposito-Lebowitz considered  Navier-Stokes limit in dimension 2. More recently, using the nonlinear energy method, in \cite{Guo-2006-CPAM} Guo justified the Navier-Stokes limit (and beyond, i.e. higher order terms in Hilbert expansion). This was extended in \cite{JX-SIMA2015} to more general initial data which allow the fast acoustic waves.  These results basically say that, given the initial data which is needed in the classical solutions of the Navier-Stokes equations, it can be constructed the solutions of the Boltzmann equation of the form $F_\eps = M +  \sqrt{M}(\eps g_1+ \eps^2 g_2 + \cdots + \eps^n g_n+ \eps^k g^R_\eps)$, where $g_1, g_2, \cdots $ can be determined by the Hilbert expansion, and $g^R_\eps$ is the error term. In particular, the first order fluctuation $g_1 = \rho_1 + \mathrm{u}_1\!\cdot\! v + \theta_1(\frac{|v|^2}{2}-\frac{3}{2})$, where $(\rho_1, \mathrm{u}_1, \theta_1)$ is the solutions to the incompressible Navier-Stokes equations.

Regarding to the Vlasov-Poisson-Boltzmann (VPB) system and VMB, the corresponding fluid limits are more fruitful since the effects of electric and magnetic fields are considered. Analytically, the limits from scaled VPB are similar to the Boltzmann equations because VPB couples with an extra Poisson equation, which has very good enough regularity. This usually does not bring much difficulties. For the limits of VPB, see recent results \cite{GJL-2019, JZ-2020}.

However, for the VMB, the situation is quite different. The corresponding hydrodynamic limits are much harder, even at the formal level, since it is coupled with Maxwell equations which are essentially hyperbolic. In a recent remarkable breakthrough \cite{Arsenio-SRM-2016}, Ars\'enio and Saint-Raymond not only proved the existence of renormalized solutions of VMB,  more importantly, also justified various limits (depending on the scalings) towards incompressible viscous electro-magneto-hydrodynamics. Among these limits, the most singular one is to the two-fluid incompressible Navier-Stokes-Fourier-Maxwell (in brief, NSFM) system with Ohm's law.

The proofs in \cite{Arsenio-SRM-2016} for justifying the weak limit from a sequence of solutions of VMB  to a dissipative solution of incompressible NSFM  are extremely hard. Part of the reasons are, besides many difficulties of the existence of renormalized solutions of VMB itself, the current understanding for the incompressible NSFM with Ohm's law is far from complete. From the point view of mathematical analysis, NSFM have the behavior which is more similar to the much less understood incompressible Euler equations than to the Navier-Stokes equations. That is the reason in \cite{Arsenio-SRM-2016}, they consider the so-called dissipative solutions of NSFM rather than the usual weak solutions. The dissipative solutions were introduced by Lions for 3-dimensional incompressible Euler equations (see section 4.4 of \cite{Lions-1996}).

The studies of incompressible NSFM just started in recent years (for the introduction of physical background, see \cite{Biskamp, Davidson}).  For weak solutions, the existence of global-in-time Leray type weak solutions are completely open, even in 2-dimension.  The first breakthrough comes from Masmoudi \cite{Masmoudi-JMPA2010}, who proved in 2-dimensional case the existence and uniqueness of global strong solutions of incompressible NSFM (in fact, the system he considered in \cite{Masmoudi-JMPA2010} is little different with the NSFM in this paper, but the analytic analysis are basically the same) for the initial data $(v^{in}, E^{in}, B^{in})\in L^2(\mathbb{R}^2)\times (H^s(\mathbb{R}^2))^2$ with $s>0$. It is notable that in \cite{Masmoudi-JMPA2010}, the divergence-free condition of the magnetic field $B$ or the decay property of the linear part coming from Maxwell's equations is {\em not} used.  Ibrahim and Keraani \cite{Ibrahim-Keraani-2011-SIMA} considered the data $(v^{in}, E^{in}, B^{in}) \in \dot B^{1/2}_{2,1}(\mathbb{R}^3)\times (\dot H^{1/2}(\mathbb{R}^3))^2$ for 3-dimensional, and $(v_0, E_0, B_0)\in \dot B^0_{2,1}(\mathbb{R}^2)\times (L^2_{log}(\mathbb{R}^2))^2$ for 2-dimensional case. Later on, German, Ibrahim and Masmoudi \cite{GIM2014} refines the previous results by running a fixed-point argument to obtain mild solutions, but taking the initial velocity field in the natural Navier-Stokes space $H^{1/2}$. In their results the regularity of the initial velocity and electromagnetic fields is lowered. Furthermore, they employed an $L^2L^\infty$-estimate on the velocity field, which significantly simplifies the fixed-point arguments used in \cite{Ibrahim-Keraani-2011-SIMA}. For some other asymptotic problems related, say, the derivation of the MHD from the Navier-Stokes-Maxwell system in the context of weak solutions, see Ars\'enio-Ibrahim-Masmoudi \cite{AIM-ARMA-2015}. Recently, in \cite{JL-CMS-2018} the authors of the current paper proved the global classical solutions of the incompressible NSFM with small intial data, by using the decay properties of both the electric field and  the wave equation with linear damping of the divergence free magnetic field. This key idea was already used in \cite{GIM2014}.

Regarding to the hydrodynamic limits of VMB in the context of classical solutions, the only previous result belongs to Jang \cite{Jang-ARMA2008}. In fact, in \cite{Jang-ARMA2008}, it was taken a very special scaling under which the magnetic effect appeared only at a higher order. As a consequence, it vanished in the limit as the Knudsen number $\eps\rightarrow 0$. So in the limiting equations derived in \cite{Jang-ARMA2008}, there was no equations for the magnetic field at all. We emphasize that in \cite{Jang-ARMA2008}, the author took the Hilbert expansion approach, and the classical solutions to the VMB were based on those of the limiting equations. So the convergence results in \cite{Jang-ARMA2008} belong to the type-II results, as we named in the last subsection.

The main purpose of this paper is to obtain the type-I results for the hydrodynamic limits to Cauchy problem of the Vlasov-Maxwell-Landau system \eqref{VMB-F}-\eqref{IC-VMB-F} with Coulomb potential, when the Knudsen number $\eps$ tends to zero. The key issue is that we do {\em not} employ the Hilbert expansion method, which as mentioned above, has two disadvantages: first, it only gives a special type of the solution of the VML. In other words, the solution with the expansion form of some finite terms. Second, the solution to the limiting equations must be known {\em before} the existence of VML. The approach employed in this paper is to obtain a family of the global in time solutions $F_\eps^\pm (t,x,v)$ to the scaled VML with energy estimate uniform in $0< \eps < 1$. Based on this uniform energy estimate, the moments of the fluctuations of $F_\eps^\pm (t,x,v)$ around the global Maxwellian converge to the solutions to the incompressible Navier-Stokes-Fourier-Maxwell (NSFM) equations. This approach automatically provides a classical solution to the NSFM equations. The first named author of this paper and Luo did this for noncutoff VMB with hard sphere collision kernel in \cite{Jiang-Luo-2022-Ann.PDE}. For the technically much harder general soft potentials cases, both noncutoff and cutoff collision kernels, it was treated in the very recent paper of the authors of the current paper in \cite{JLei-2023}.

\subsection{Notations} We gather here the notations we will use throughout this paper. We first define the following shorthand notation,
\begin{equation*}
  \langle \cdot \rangle = \sqrt{1 + |\cdot|^2} \,.
\end{equation*}
For convention, we index the usual $L^p$ space by the name of the concerned variable. So we have, for $p \in [1, + \infty]$,
\begin{equation*}
  L^p_{[0,T]} = L^p ([0,T]) \,, \ L^p_x = L^p (\T^3) \,, \ L^p_v = L^p (\R^3) \,, \ L^p_{x,v} = L^p (\mathbb{R}^3 \times \mathbb{R}^3) \,.
\end{equation*}
For $p = 2$, we use the notations $ \langle \cdot \,, \cdot \rangle_{L^2_x} $, $ \langle \cdot \,, \cdot \rangle_{L^2_v} $ and $ \langle \cdot \,, \cdot \rangle_{L^2_{x,v}} $ to represent the inner product on the Hilbert spaces $L^2_x$, $L^2_v$ and $L^2_{x,v}$, respectively.

For any multi-indexes $\alpha = (\alpha_1, \alpha_2, \alpha_3)$ and $ \beta = ( \beta_1, \beta_2, \beta_3 )$ in $\mathbb{N}^3$ we denote the $(\alpha,\beta)^{th}$ partial derivative by
\begin{equation*}
  \partial^\alpha = \partial^\alpha_x \partial^\beta_v = \partial^{\alpha_1}_{x_1} \partial^{\alpha_2}_{x_2} \partial^{\alpha_3}_{x_3} \partial^{\beta_1}_{v_1} \partial^{\beta_2}_{v_2} \partial^{\beta_3}_{v_3} \,.
\end{equation*}

As in \cite{Guo-Wang-CPDE-2012}, we introduce the operator $\Lambda^s$ with $s\in\mathbb{R}$ by
\[
\left(\Lambda^sg\right)(t,x,v)=\int_{\mathbb{R}^3}|\xi|^{s}\hat{g}(t,\xi,v)e^{2\pi ix\cdot\xi}d\xi
=\int_{\mathbb{R}^3}|\xi|^{s}\mathcal{F}[g](t,\xi,v)e^{2\pi ix\cdot\xi}d\xi
\]
with $\hat{g}(t,\xi,v)\equiv\mathcal{F}[g](t,\xi,v)$ being the Fourier transform of $g(t,x,v)$ with respect to $x$.  The homogeneous Sobolev space $\dot{H}^s$ is the Banach space consisting of all $g$ satisfying  $\|g\|_{\dot{H}}<+\infty$, where
\[
\|g(t)\|_{\dot{H}^s}\equiv\left\|\left(\Lambda^s g\right)(t,x,v)\right\|_{L^2_{x,v}}=\left\||\xi|^s\hat{g}(t,\xi,v)\right\|_{L^2_{\xi,v}}.
\]

For later use, we define
\begin{equation}\label{norm}
	\begin{split}
		|f|_{\sigma,w}
		\sim&\left|\langle v\rangle^{-\frac{1}{2}}f\right|_{2,w}
		+\left|\langle v\rangle^{-\frac{3}{2}}\nabla_{ v}f\cdot\frac{ v}{| v|}\right|_{2,w}
		+\left|\langle v\rangle^{-\frac{1}{2}}\nabla_{ v}f\times\frac{ v}{| v|}\right|_{2,w}
	\end{split}
\end{equation}
with
$$
|f|_{2,w}=|wf|_2,\quad \|f\|_{\sigma,w}=\left\||f|_{\sigma,w}\right\|,\quad |f|_{\sigma}=|f|_{\sigma,0},\quad \|f\|_{\sigma}=\left\||f|_{\sigma}\right\|.
$$

\subsection{Main results.}
To state for brevity, we introduce the following energy and dissipation rate functional respectively:
\begin{eqnarray}
    \mathcal{E}_N(t) &= & \|f_\eps \|^2_{H^N_{x}} + \| E_\eps \|^2_{H^N_x} + \| B_\eps \|^2_{H^N_x}, \label{Energy-E-D-1}\\
    \mathcal{D}_N(t) &= & \tfrac{1}{\eps^2} \| \{{\bf I-P}\}f_\eps\|^2_{H^N_{x}L^2_\sigma} + \| \nabla_x {\bf P} f_\eps \|^2_{H^{N-1}_x L^2_v} + \| E_\eps \|^2_{H^{N-1}_x} + \| \nabla_x B_\eps \|^2_{H^{N-2}_x}. \label{Energy-E-D-2}
\end{eqnarray}

Due to the weaker dissipation of the Landau operator in the soft potential case rather than the hard cases, in order to deal with the external force term brought by the Lorenian electric-magnetic force, especially including the growth of velocity $v$, we need to introduce the time-velocity weight
\begin{equation}\label{weight}
  w_\ell(\alpha)=e^{\frac {q\langle v\rangle^2}{(1+t)^\vartheta}}\langle v\rangle^{3(\ell-|\alpha|)},\ q\ll 1,\ \ell\geq N
\end{equation}
where the indexes $\alpha$ is related to space $x-$derivatives.
\begin{remark}
For this time-velocity weight $w_\ell(\alpha)$, we need to point that
\begin{itemize}
\item
The algebraic component $\langle v\rangle^{3(\ell-|\alpha|)}$, which depend on spatial derivatives, play an essential role in dealing with
  the magnetic-field term $\frac1\varepsilon v\times B\cdotp \nabla_v f$;
  \item
  The exponential component $e^{\frac {q\langle v\rangle^2}{(1+t)^\vartheta}}$ will generate additional dissipation terms in the energy estimation.
\end{itemize}
\end{remark}
The energy and dissipation rate functional with respect to $w_\ell(\alpha)$ are introduced respectively by
\begin{eqnarray}
\mathcal{E}_{N-1,\ell}(t)
&=&\sum_{|\alpha|\leq N-1}\left\|w_\ell(\alpha)\partial^\alpha \{{\bf I-P}\}f_\eps\right\|^2,\label{E_l-k-1}\label{E-N-1-w}\\
\mathcal{D}_{N-1,\ell}(t)
&=&
\tfrac{1}{\eps^2} \sum_{|\alpha|\leq N-1}\|w_\ell(\alpha)\partial^\alpha \{{\bf I-P}\}f_\eps\|^2_{\sigma}\nonumber\\
&&+ \sum_{|\alpha|\leq N-1}\frac{q\vartheta}{(1+t)^{1+\vartheta}}\|w_\ell(\alpha)\partial^\alpha \{{\bf I-P}\}f_\eps\langle v\rangle\|^2.\label{D-N-1-w}
\end{eqnarray}
To obtain the desired temporal time decays, we introduce the energy and dissipation rate functional with the lowest $k\mbox{-}$order space-derivative
\begin{eqnarray}
\mathcal{E}_{N_0}^{k}(t)&=&\sum_{|\alpha|=k}^{N_0}\left\|\partial^\alpha[f_\varepsilon,E_\varepsilon,B_\varepsilon]\right\|^2,
\label{E-k}\\
\mathcal{D}_{N_0}^{k}(t)&=&
\left\|\nabla^{k}[E_\varepsilon,\rho^\varepsilon_+-\rho^\varepsilon_-]\right\|^2+\sum_{k+1\leq|\alpha|\leq N_0-1}\left\|
\partial^\alpha[{\bf P}f_\varepsilon,E_\varepsilon,B_\varepsilon]\right\|^2\nonumber\\
&&+\sum_{|\alpha|=N_0}\left\|\partial^\alpha {\bf P}f_\varepsilon\right\|^2+\tfrac{1}{\eps^2}\sum_{k\leq |\alpha| \leq N_0}\left\|\partial^\alpha{\bf\{I-P\}}f_\varepsilon\right\|^2_{\sigma},\label{D-k}
\end{eqnarray}
respectively.

The main results of this paper are two theorems listed below. The first theorem is the global existence of the scaled two-species VML system \eqref{VML-F} with uniform energy estimate with respect to the Knudsen number $0 < \eps \leq 1$.

\begin{theorem}\label{Main-Thm-1}
Assume $0<q\ll 1$, $0<\vartheta<1$, $0<\varepsilon<1$, $N_0\geq 3$, $N=2N_0$ and $l\geq N+\frac12$, if the initial data
 \begin{eqnarray}\label{Def-Y_0}
  Y_{f_\varepsilon,E_\varepsilon,B_\varepsilon}(0)&\equiv&\sum_{|\alpha|\leq N}\left\|e^{\frac {q\langle v\rangle^2}{(1+t)^\vartheta}}\langle v\rangle^{3(l-|\alpha|)} \partial^\alpha f_{\varepsilon,0}\right\|\nonumber\\
  &&+\|[E_{\varepsilon,0},B_{\varepsilon,0}]\|_{H^N_x}+\|\Lambda^{-\frac12}[f_{\varepsilon,0},E_{\varepsilon,0},B_{\varepsilon,0}]\|
 \end{eqnarray}
is taken a sufficiently small positive constant, which is independent of $\varepsilon$, then the Cauchy problem admits a global solution, we can also deduce that there exist energy functionals $\mathcal{E}_{N}(t)$, $\mathcal{E}_{N-1,l}(t)$ and the corresponding energy dissipation functionals $\mathcal{D}_{N}(t)$, $\mathcal{D}_{N-1,l}(t)$ which satisfy \eqref{Energy-E-D-1}, \eqref{E-N-1-w}, \eqref{Energy-E-D-2} and \eqref{D-N-1-w} respectively such that
\begin{eqnarray}\label{Main-Thm-1-1}
&&\frac{d}{dt}\left\{(1+t)^{-\frac{1+\epsilon_0}2}\sum_{|\alpha|=N}{\varepsilon^2}\left\|w_l(\alpha)\partial^\alpha f_\varepsilon\right\|^2+\mathcal{E}_{N-1,l}(t)+\mathcal{E}_{N}(t)\right\}\nonumber\\
&&+\mathcal{D}_{N}(t)+\mathcal{D}_{N-1,l}(t)
\lesssim0
\end{eqnarray}
{holds for all $0\leq t\leq T$.}

Meanwhile, we also get the large time behavior in the following result:
\begin{eqnarray}
  \mathcal{E}_{N_0}^{k}(t)\lesssim Y^2_{f_\varepsilon,E_\varepsilon,B_\varepsilon}(0)(1+t)^{-k-\frac12}, \ k=0,1,\cdots, N_0-2,
\end{eqnarray}
where $\mathcal{E}_{N_0}^{k}(t)$ is defined in \eqref{E-k}.
\end{theorem}

The second is on the two-fluid incompressible NSFM limit with Ohm's law as $\eps \rightarrow 0$, taken from the solutions $(f_\eps , E_\eps , B_\eps)$ of system \eqref{VML-F} which are constructed in the first theorem.
\begin{theorem}\label{Main-Thm-2}

Take the assumption as in Theorem \ref{Main-Thm-1}. Assume that the initial data $( f_{\eps,0} , E_{\eps,0}, B_{\eps,0} )$ satisfy
	\begin{enumerate}
		\item $f_{\eps,0} \in H^N_xL^2_{v}$, $E_{\eps,0}$, $B_{\eps,0}\in H^N_x$;
		\item there exist scalar functions $\rho{(0,x)}$, $ \theta{(0,x)}$,  $n{(0,x)}\in H^N_x$ and vector-valued functions $u{(0,x)} $,  $E{(0,x)}$, $B{(0,x)} \in H^N_x$ such that
		\begin{equation}
		  \begin{array}{l}
			f_{\eps,0} \rightarrow f{(0,x,v)}\quad \textrm{strongly in} \ H^N_{x}L^2_v \,, \\
			E_{\eps,0} \rightarrow E{(0,x)} \quad \textrm{strongly in} \ H^N_x \,, \\
			B_{\eps,0} \rightarrow B{(0,x)} \quad \textrm{strongly in} \ H^N_x
		  \end{array}
		\end{equation}
		as $\eps \rightarrow 0$, where $f{(0,x,v)}$ is of the form
		\begin{equation}
		  \begin{aligned}
		    f(0,x,v)= & ( \rho(0,x) + \tfrac{1}{2} n(0,x) ) \tfrac{{q_1} + {q_2}}{2} \sqrt{M} + ( \rho(0,x) - \tfrac{1}{2} n(0,x) ) \tfrac{{q_2} - {q_1}}{2} \sqrt{M} \\
		    & + u(0,x) \cdot v {q_2} \sqrt{M} + \theta(0,x) ( \tfrac{|v|^2}{2} - \tfrac{3}{2} ) {q_2} \sqrt{M} \,.
		  \end{aligned}
		\end{equation}
	\end{enumerate}
  Let $(f_\eps, E_\eps, B_\eps)$ be the family of solutions to the scaled two-species VML \eqref{VML-F} constructed in Theorem \ref{Main-Thm-1}. Then, as $\eps \rightarrow 0$,
  \begin{equation}
    \begin{aligned}
      f_\eps \rightarrow \left\{( \rho + \tfrac{1}{2} n ) \tfrac{{q_1} + {q_2}}{2}  + ( \rho - \tfrac{1}{2} n ) \tfrac{{q_2} - {q_1}}{2}  + u \cdot v {q_2} \sqrt{M} + \theta ( \tfrac{|v|^2}{2} - \tfrac{3}{2} )\right\} \sqrt{M}
    \end{aligned}
  \end{equation}
  weakly-$\star$ in $t \geq 0$, strongly in $H^{N-1}_{x}L^2_v$ and weakly in $H^N_{x}L^2_v$, and
  \begin{equation}
    \begin{aligned}
      E_\eps \rightarrow E \quad \textrm{and } \quad B_\eps \rightarrow B
    \end{aligned}
  \end{equation}
  strongly in $C( \R^+; H^{N-1}_x )$, weakly-$\star$ in $t \geq 0$ and weakly in $H^N_x$. Here
  $$( u, \theta , n, E, B ) \in C(\R^+; H^{N-1}_x) \cap L^\infty (\R^+ ; H^N_x)$$
  is the solution to the incompressible NSFM \eqref{INSFM-Ohm} with Ohm's law, which has the initial data
  \begin{equation}
    \begin{aligned}
      u|_{t=0} = \mathcal{P} u(0,x) \,, \ \theta |_{t=0} = \tfrac{3}{5} \theta(0,x) - \tfrac{2}{5} \rho(0,x) \,, \ E|_{t=0} = E(0,x)\,, \ B|_{t=0} = B(0,x) \,,
    \end{aligned}
  \end{equation}
  where $\mathcal{P}$ is the Leray projection. Moreover, the convergence of the moments holds:
  \begin{equation}
    \begin{aligned}
      & \mathcal{P} \langle f_\eps , \tfrac{1}{2} {q_2} v \sqrt{M} \rangle_{L^2_v} \rightarrow u \,, \\
      & \langle f_\eps , \tfrac{1}{2} {q_2} ( \tfrac{|v|^2}{5} - 1 ) \sqrt{M} \rangle_{L^2_v} \rightarrow \theta \,,
    \end{aligned}
  \end{equation}
  strongly in $C(\R^+ ; H^{N-1}_x)$, weakly-$\star$ in $t \geq 0$ and weakly in $H^N_x$ as $\eps \rightarrow 0$.
\end{theorem}

\subsection{Outline of proofs}
The whole paper relies on the global-in-time energy estimate uniform in $0<\varepsilon\leq 1$
for small initial data, i.e. the estimate \eqref{Main-Thm-1-1}. This depends on the careful design of
the energy norm and energy dissipation functional, the control of the singular terms. The main difficulty, in deducing the uniform-in-$\varepsilon$ estimates, comes from the singularity terms, which can be stated as the following:
 \begin{itemize}
  \item [i)] A general coercivity estimates on $\mathscr{L}$ tell us that
\begin{eqnarray}\label{hard-1}
  &&\frac1{\varepsilon^2}\left(\mathscr{L}\partial^\alpha f,w^2_{l}(\alpha)\partial^\alpha f\right)\nonumber\\
  &\gtrsim&\frac1{\varepsilon^2}\|w_{l}(\alpha)\partial^\alpha f\|_\sigma^2-\frac1{\varepsilon^2}\|\partial^\alpha f\|_\sigma^2\nonumber\\
  &\gtrsim&\frac1{\varepsilon^2}\|w_{l}(\alpha)\partial^\alpha f\|_\sigma^2-\frac1{\varepsilon^2}\left\|\partial^\alpha {\bf P}f\right\|^2-\frac1{\varepsilon^2}\left\|\partial^\alpha \{{\bf I-P}\}f\right\|_\sigma^2,
\end{eqnarray}
where the singularity term $\frac1{\varepsilon^2}\left\|\partial^\alpha {\bf P}f\right\|^2$ can not be controlled any more.
\item [ii)]Since both the singular factor $\frac1\varepsilon$ and the weaker coercivity property for the Columb potential than the hard sphere model as in \cite{Jiang-Luo-2022-Ann.PDE}, how to deal with the singularity terms, i.e. $\frac1{\varepsilon}v\times B\cdot\nabla_v f$ and $\frac1{\varepsilon}v\cdot\nabla_x f$, especially for the corresponding energy estimates with weight?
\end{itemize}
To overcome the above difficulty induced by the singularity terms, the main strategies
and novelties can be summarized in the following parts.
      \begin{itemize}
        \item [1)]We design a time-velocity weight $w_\ell(\alpha)=e^{\frac {q\langle v\rangle^2}{(1+t)^\vartheta}}\langle v\rangle^{3(\ell-|\alpha|)}$ that only depends on the spatial-derivatives, which is different from \cite{DLYZ-KRM2013,Guo-2012-JAMS} and \cite{Duan-Yang-Zhao-MMMAS-2013,DLYZ-CMP2017}, in which the corresponding weight depends on  both spatial-derivatives and velocity-derivatives, or velocity-derivatives. In fact, with the help of the algebraic part $\langle v\rangle^{3(\ell-|\alpha|)}$ in $w(\alpha)$, one can deduce that
\begin{eqnarray*}
 &&{\frac1\varepsilon}\left|\left(\partial^\alpha[ v\times B\cdot\nabla_v {\bf \{I-P\}}f],w^2_\ell(\alpha)\partial^\alpha {\bf \{I-P\}}f\right)\right|\\
 &=&{\frac1\varepsilon}\left|\left([ v\times B\cdot\nabla_v \partial^\alpha{\bf \{I-P\}}f],w^2_\ell(\alpha)\partial^\alpha {\bf \{I-P\}}f\right)\right|\\
 &&+\sum_{1\leq|\alpha_1|\leq N}{\frac1\varepsilon}\left|\left([ v\times \partial^{\alpha_1}B\cdot\nabla_v \partial^{\alpha-\alpha_1}{\bf \{I-P\}}f],w^2_\ell(\alpha)\partial^\alpha {\bf \{I-P\}}f\right)\right|\\
 &=&\sum_{1\leq|\alpha_1|\leq N}{\frac1\varepsilon}\left|\left([ v\times \partial^{\alpha_1}B\cdot\nabla_v \partial^{\alpha-\alpha_1}{\bf \{I-P\}}f],w^2_\ell(\alpha)\partial^\alpha {\bf \{I-P\}}f\right)\right|\\
 &\lesssim&\sum_{1\leq|\alpha_1|\leq N}{\frac1\varepsilon}\int_{\mathbb{R}^3_x\times\mathbb{R}^3_v}|\partial^{\alpha_1}B|
 |w_\ell(\alpha-\alpha_1)\nabla_v \partial^{\alpha-\alpha_1}{\bf \{I-P\}}f\langle v\rangle^{\frac12-3|\alpha_1|}|\\
 &&\times |w_\ell(\alpha)\partial^\alpha {\bf \{I-P\}}f\langle v\rangle^{-\frac12}|dxdv\\
 &\lesssim&\mathcal{E}_N(t)\mathcal{D}_{N-1,\ell}(t)+\frac\eta{\varepsilon^2}\left\|w_\ell(\alpha)\partial^\alpha {\bf \{I-P\}}f\langle v\rangle^{-\frac12}\right\|^2.
\end{eqnarray*}

         \item [2)]As \cite{DLYZ-KRM2013,DLYZ-CMP2017,Duan-Yang-Zhao-MMMAS-2013}, an exponential part $e^{\frac {q\langle v\rangle^2}{(1+t)^\vartheta}}$ in $w_\ell(\alpha)$ can yield an extra dissipation term
\[\frac{q\vartheta}{(1+t)^{1+\vartheta}}\|w_\ell(\alpha)\partial^\alpha \{{\bf I-P}\}f\langle v\rangle\|^2,\]
which can help to deduce that, the nonlinear term with respective to the electric-field $E$ can be dominated by
\begin{eqnarray*}
 &&\left|\left(E\cdot v\partial^\alpha\{{\bf I-P}\}f,w^2_\ell(\alpha)\partial^\alpha \{{\bf I-P}\}f\right)\right|\\
&\lesssim&\|E\|_{L^\infty_x}\left\|w_\ell(\alpha)\partial^{\alpha}\{{\bf I-P}\}f\langle v\rangle^{1/2}\right\|
\left\|w_\ell(\alpha)\partial^\alpha \{{\bf I-P}\}f\langle v\rangle^{1/2}\right\|\\
&\lesssim&\left\{\frac{(1+t)^{1+\vartheta}}{q\vartheta}\|E\|_{L^\infty_x}\right\}\left\{\frac{q\vartheta}{(1+t)^{1+\vartheta}}\left\|w_\ell(\alpha)\partial^\alpha \{{\bf I-P}\}f\langle v\rangle^{1/2}\right\|^2\right\}
\end{eqnarray*}
as long as $\|E\|_{L^\infty}$ enjoys some sufficiently time-decay rates.
\item [3)]To this end, by applying the interpolation method with respective to negative Sobolev space as \cite{Guo-Wang-CPDE-2012,Lei-Zhao-JFA-2014} and sophisticated and careful estimation,  we obtain the temporal time-decay rates of electric-magnetic fields, which is independent of $\varepsilon$, whose detail proofs can be seen in Lemma \ref{lemma4.3}, Lemma \ref{Lemma4.4} and Proposition \ref{Lemma4.5}. However, to ensure that the time decay result holds, we need the boundedness of the weighted energy norms
    \[\max\left\{\mathcal{E}_{N}(t),\mathcal{\bar{E}}_{N_0,N}(t)\right\}\]
    where the energy functionals $\mathcal{E}_{N}(t)$ and $\mathcal{\bar{E}}_{N_0,N}(t)$ are defined in \eqref{Energy-E-D-1} and \eqref{def-E-N-bar}.
\item [4)]For this purpose, we have to take the energy estimates with the wight $w_\ell(\alpha)$, however, \eqref{hard-1} tells us that we have to deal with the singularity term $\frac1{\varepsilon^2}\left\|\partial^\alpha {\bf P}f\right\|^2$.
    \begin{itemize}
      \item To this end, we multiply $\varepsilon^2$ to the energy estimates with weight on the highest-order spatial derivatives, such that $\left\|\partial^\alpha {\bf P}f\right\|^2$ can be controlled by the macroscopic dissipation terms.
      \item However, if we multiply $\varepsilon^2$ to the energy estimates with weight on all spatial derivatives,
    it is impossible to obtain the boundedness of the weighted energy norms $\mathcal{\bar{E}}_{N_0,N}(t)$.
    \end{itemize}
\item [5)]Different from the previous traditional energy estimation steps as \cite{Lei-Zhao-JFA-2014}, except for  the energy estimation of the highest derivative, we must use the microscopic projection equation \eqref{I-P} to estimate them. A detail proof can be seen in Lemma \ref{lemma-micro}.
\item [6)] Finally, by combining the above strategies, we construct the a priori estimates \eqref{Def-a-priori} and close it. Based on the uniform in $\varepsilon$ estimates, we employ the moment method to rigorously justify the
hydrodynamic limit from the perturbed VML to the two fluid incompressible NSFM
system with Ohm's law as \cite{Jiang-Luo-2022-Ann.PDE}.
      \end{itemize}

\subsection{The structure}
In the next section, we will deduce the Lyapunov inequality for the energy functional $\mathcal{E}_{N}(t)$ without weight. In Section 3, Lyapunov inequality for the energy functional with weight will be derived. Section 4 and Section 5 are devoted to the temporal time decay rates and negative Sobolev estimates respectively. We will construct the a priori estimates and close it in Section 6. In Section 7, based on
the uniform global in time energy bound, we take the limit to derive the incompressible
NSFM system with Ohm's law. Some basic properties of the linear collision operator and
bilinear symmetric operator will be given  in Appendix.

\section{Lyapunov inequality for the energy functional $\mathcal{E}_{N}(t)$} For notational simplicity, we drop the lower index $\eps$, i.e. $f\equiv f_{\eps}$, $E\equiv E_{\eps}$ and $B\equiv B_{\eps}$, \eqref{VML-F} can be rewritten as the following system:
\begin{equation}\label{VML-drop-eps}
  \left\{
    \begin{array}{l}
      \partial_t f + \tfrac{1}{\eps} \big[ v \cdot  \nabla_x f + q_0 ( \eps E + v \times B ) \cdot \nabla_v f \big] + \tfrac{1}{\eps^2} \mathscr{L} f - \tfrac{1}{\eps}  (E \cdot v) \sqrt{{M}}  q_1 \\[2mm]
      \qquad \qquad = \tfrac{1}{2} q_1 (E \cdot v) f + \tfrac{1}{\eps} \mathscr{T} (f, f) \,,\\[2mm]
      \partial_t E - \nabla_x \times B = - \tfrac{1}{\eps} \int_{\R^3} f \cdot q_1 v \sqrt{{M}} \d v \,,\\[2mm]
      \partial_t B + \nabla_x \times E = 0 \,,\\[2mm]
      \div_x E = \int_{\R^3} f \cdot q_1 \sqrt{{M}} \d v \,, \ \div_x B = 0 \,.
    \end{array}
  \right.
\end{equation}

Without generality, we take $N=2N_0$ with $N_0\geq 3$ for brevity, since we do not attempt to obtain the optimal regularity index.
\begin{lemma}\label{lemma3.7}
For $|\alpha|\leq N$, if we take $l\geq N$, then
\begin{eqnarray}\label{lemma3.7-1}
&&\frac{d}{dt}\left\|\partial^\alpha [f,E,B]\right\|^2+\frac1{\varepsilon^2}\left\|\partial^\alpha\{{\bf I-P}\}f\right\|_\sigma^2\nonumber\\
&\lesssim &\| E\|^{2}_{L^{\infty}_x}\left\|\langle v\rangle^{\frac 32}\nabla^N_xf\right\|^2+\mathcal{E}_N(t)\left\{\mathcal{D}_{N}(t)+\mathcal{D}_{N-1,l}(t)\right\}
+\eta\mathcal{D}_{N}(t),
\end{eqnarray}
holds for all $0\leq t\leq T$.
\end{lemma}
\begin{proof}
First of all, it is straightforward to establish the energy identities
\begin{eqnarray}\label{spatial-without}
&&\frac12\frac{d}{dt}\left(\left\|\partial^\alpha f\right\|^2+ \left\|\partial^\alpha[E,B])\right\|^2\right)+\frac1{\varepsilon^2}\left(\mathscr{L}\partial^\alpha f,\partial^\alpha f\right)\nonumber\\
&=&\left(\partial^\alpha\left(\frac{q_0}{2}E\cdot vf\right),\partial^\alpha f\right)-\left(\partial^\alpha\left(q_0 E\cdot\nabla_vf\right),\partial^\alpha f\right)\\
&&-\frac1\varepsilon\left(\partial^\alpha\left(q_0(v\times B)\cdot\nabla_vf\right),\partial^\alpha f\right)+\frac1\varepsilon\left(\partial^\alpha\mathscr{T}(f,f),\partial^\alpha f\right).\nonumber
\end{eqnarray}
The coercivity property of the linear operator $\mathscr{L}$ tells us that
\[
\frac1{\varepsilon^2}\left(\mathscr{L}\partial^\alpha f,\partial^\alpha f\right)\gtrsim\frac1{\varepsilon^2}\|\partial^\alpha{\bf\{I-P\}}f\|_\sigma^2.
\]
For the four terms on the right-hand side of \eqref{spatial-without}, we estimate them one by one in the following.

{\bf Case 1: $\alpha=0$.}
By applying macro-micro decomposition, $L^2-L^3-L^6$ or $L^2-L^\infty-L^2$ Sobolev inequalities and Cauchy inequalities, one has
\begin{eqnarray}
 &&\left(\frac{q_0}{2}E\cdot vf, f\right)\nonumber\\
 &=& \left(\frac{q_0}{2}E\cdot v{\bf P}f, {\bf P}f\right)+\left(\frac{q_0}{2}E\cdot v{\bf P}f, {\bf \{I-P\}}f\right)\nonumber\\
 &&+\left(\frac{q_0}{2}E\cdot v{\bf \{I-P\}}f, {\bf P}f\right)+\left(\frac{q_0}{2}E\cdot v{\bf \{I-P\}}f, {\bf \{I-P\}}f\right)\\
 &\lesssim&\|f\|_{H^1_xL^2_v}\|E\|\|\nabla_xf\|_\sigma+\|E\|_{L^\infty_x}^2\|{\bf \{I-P\}}f\langle v\rangle^{\frac32}\|^2+\eta\|{\bf \{I-P\}}f\|^2_\sigma\nonumber\\
&\lesssim&\|f\|^2_{H^1_xL^2_v}\|E\|^2+\eta\|\nabla_xf\|^2_\sigma+\|E\|_{L^\infty_x}^2\|{\bf \{I-P\}}f\langle v\rangle^{\frac32}\|^2+\eta\|{\bf \{I-P\}}f\|^2_\sigma\nonumber\\
&\lesssim&\mathcal{E}_2(t)\|E\|^2+\eta\|\nabla_xf\|^2_\sigma+\|E\|_{L^\infty_x}^2\mathcal{D}_{1,1}(t)+\eta\|{\bf \{I-P\}}f\|^2_\sigma\nonumber\\
&\lesssim&\mathcal{E}_2(t)\|E\|^2+\|E\|_{L^\infty_x}^2\mathcal{D}_{1,1}(t)+\eta\mathcal{D}_{2}(t).\nonumber
\end{eqnarray}
Integrating in part with respect to $v$ yields that
\begin{eqnarray}
  -\left(\left(q_0 E\cdot\nabla_vf\right), f\right)-\frac1\varepsilon\left(\left(q_0(v\times B)\cdot\nabla_vf\right), f\right)=0.
\end{eqnarray}
By using Lemma \ref{Lemma2.1}, we can deduce that the last term can be dominated by
$$
 \mathcal{E}_2(t)\mathcal{D}_{2}(t)+\frac{\eta}{{\varepsilon^2}}\|{\{\bf I- P\}}f\|^2_\sigma.
$$
By collecting the above related estimates, we arrive at
\begin{eqnarray}\label{0-order}
&&\frac12\frac{d}{dt}\left(\left\|f\right\|^2+ \left\|[E,B])\right\|^2\right)+\frac1{\varepsilon^2}\|{\bf \{I-P\}}f\|^2_\sigma\nonumber\\
&\lesssim&\mathcal{E}_2(t)\mathcal{D}_{2}(t)+\mathcal{E}_2(t)\|E\|^2+\|E\|_{L^\infty_x}^2\mathcal{D}_{1,1}(t)+\eta\mathcal{D}_2(t).
\end{eqnarray}

{\bf Case 2: $1\leq |\alpha|\leq N$.} By using Sobolev inequalities, one can deduce that the first term can be bounded by
\begin{equation}\label{alpha-without-w}
\begin{aligned}
\lesssim&\|E\|_{L^\infty_x} \left\|\langle v\rangle^{\frac 32}\partial^{\alpha}f\right\|
\left\|\langle v\rangle^{-\frac {1}2}\partial^\alpha f\right\|\\[2mm]
&+\chi_{|\alpha|\geq3}\sum_{1\leq|\alpha_1|\leq |\alpha|-2}\left\|\partial^{\alpha_1}E\right\|_{L^\infty_x}
\left\|\langle v\rangle^{\frac 32}\partial^{\alpha-\alpha_1}f\right\|
\left\|\langle v\rangle^{-\frac {1}2}\partial^\alpha f\right\|\\[2mm]
&+\chi_{|\alpha|\geq2}\sum_{|\alpha_1|=|\alpha|-1}\int_{\mathbb{R}^3_x\times\mathbb{R}^3_v}|\partial^{\alpha_1}E|
|\langle v\rangle^{\frac 32}\partial^{\alpha-\alpha_1}f|
|\langle v\rangle^{-\frac {1}2}\partial^\alpha f|dvdx\\[2mm]
&+\chi_{|\alpha|\geq1}\int_{\mathbb{R}^3_x\times\mathbb{R}^3_v}|\partial^{\alpha}E|
|\langle v\rangle^{\frac32}f|
|\langle v\rangle^{-\frac {1}2}\partial^\alpha f|dvdx.
\end{aligned}
\end{equation}
By Cauchy inequality, one deduce that the first two terms on the right-hand side of \eqref{alpha-without-w} can be bounded  by
\begin{equation}
\begin{aligned}
\lesssim&\|E\|^2_{L^\infty_x} \left\|\langle v\rangle^{\frac 32}\partial^{\alpha}f\right\|^2+\chi_{|\alpha|\geq3}\sum_{1\leq|\alpha_1|\leq |\alpha|-2}\left\|\partial^{\alpha_1}E\right\|^2_{L^\infty_x}
\left\|\langle v\rangle^{\frac 32}\partial^{\alpha-\alpha_1}f\right\|^2
+\eta\left\|\langle v\rangle^{-\frac {1}2}\partial^\alpha f\right\|^2\\[2mm]
\lesssim&\| E\|^{2}_{L^{\infty}_x}\left\|\langle v\rangle^{\frac 32}\nabla^N_xf\right\|^2+\mathcal{E}_N(t)\left\{\mathcal{D}_{N}(t)+\mathcal{D}_{N-1,l}(t)\right\}+\eta\left\|\partial^\alpha f\right\|_\sigma^2,
\end{aligned}
\end{equation}
where we ask that $l\geq N+\frac23$.

By Cauchy inequality and macro-micro decomposition, one has that the last two terms can be controlled by
\begin{eqnarray}
&&\chi_{|\alpha|\geq2}\sum_{|\alpha_1|=|\alpha|-1}\int_{\mathbb{R}^3_x\times\mathbb{R}^3_v}|\partial^{\alpha_1}E|
|\langle v\rangle^{\frac 32}\partial^{\alpha-\alpha_1}f|
|\langle v\rangle^{-\frac {1}2}\partial^\alpha f|dvdx\nonumber\\[2mm]
&&+\chi_{|\alpha|\geq1}\int_{\mathbb{R}^3_x\times\mathbb{R}^3_v}|\partial^{\alpha}E|
|\langle v\rangle^{\frac32}f|
|\langle v\rangle^{-\frac {1}2}\partial^\alpha f|dvdx\nonumber\\
&\lesssim&\chi_{|\alpha|\geq2}\sum_{|\alpha_1|=|\alpha|-1}\int_{\mathbb{R}^3_x\times\mathbb{R}^3_v}|\partial^{\alpha_1}E|
|\langle v\rangle^{\frac 32}\partial^{\alpha-\alpha_1}{\bf P}f|
|\langle v\rangle^{-\frac {1}2}\partial^\alpha f|dvdx\nonumber\\[2mm]
&&+\chi_{|\alpha|\geq1}\int_{\mathbb{R}^3_x\times\mathbb{R}^3_v}|\partial^{\alpha}E|
|\langle v\rangle^{\frac32}{\bf P}f|
|\langle v\rangle^{-\frac {1}2}\partial^\alpha f|dvdx\nonumber\\
&&+\chi_{|\alpha|\geq2}\sum_{|\alpha_1|=|\alpha|-1}\int_{\mathbb{R}^3_x\times\mathbb{R}^3_v}|\partial^{\alpha_1}E|
|\langle v\rangle^{\frac 32}\partial^{\alpha-\alpha_1}{\bf\{I-P\}}f|
|\langle v\rangle^{-\frac {1}2}\partial^\alpha f|dvdx\nonumber\\[2mm]
&&+\chi_{|\alpha|\geq1}\int_{\mathbb{R}^3_x\times\mathbb{R}^3_v}|\partial^{\alpha}E|
|\langle v\rangle^{\frac32}{\bf \{I-P\}}f|
|\langle v\rangle^{-\frac {1}2}\partial^\alpha f|dvdx\nonumber\\
&\lesssim&\mathcal{E}_N(t)\left\{\mathcal{D}_{N}(t)+\mathcal{D}_{N-1,l}(t)\right\}+\eta\left\|\partial^\alpha f\right\|_\sigma^2.
\end{eqnarray}
Now by collecting we arrive at
\begin{eqnarray}
  &&\left(\partial^\alpha\left(\frac{q_0}{2}E\cdot vf\right),\partial^\alpha f\right)\nonumber\\
  &\lesssim&\| E\|^{2}_{L^{\infty}_x}\left\|\langle v\rangle^{\frac 32}\nabla^N_xf\right\|^2+\mathcal{E}_N(t)\left\{\mathcal{D}_{N}(t)+\mathcal{D}_{N-1,l}(t)\right\}+\eta\left\|\partial^\alpha f\right\|_\sigma^2,
\end{eqnarray}
where we ask that  $l\geq N-\frac13$.
By a similar way, note that $\left(E\cdot \partial^\alpha\nabla_vf,\partial^\alpha f\right)=0,$ we can deduce that
\begin{eqnarray}
  &&\left(\partial^\alpha\left(q_0 E\cdot\nabla_vf\right),\partial^\alpha f\right)
  \lesssim\mathcal{E}_N(t)\left\{\mathcal{D}_{N}(t)+\mathcal{D}_{N-1,l}(t)\right\}+\eta\left\|\partial^\alpha f\right\|_\sigma^2.
\end{eqnarray}

As for the third term on the right-hand side of \eqref{spatial-without}, by using macro-micro decomposition, we have
\begin{eqnarray}
&&\frac1\varepsilon\left(\partial^\alpha\left(q_0(v\times B)\cdot\nabla_vf\right),\partial^\alpha f\right)\nonumber\\
&=&\frac1\varepsilon\left(\partial^\alpha\left(q_0(v\times B)\cdot\nabla_v{\bf P}f\right),\partial^\alpha {\bf P}f\right)\nonumber\\
&&+\frac1\varepsilon\left(\partial^\alpha\left(q_0(v\times B)\cdot\nabla_v{\bf P}f\right),\partial^\alpha {\{\bf I- P\}}f\right)\nonumber\\
&&+\frac1\varepsilon\left(\partial^\alpha\left(q_0(v\times B)\cdot\nabla_v{\{\bf I- P\}}f\right),\partial^\alpha {\bf P}f\right)\nonumber\\
&&+\frac1\varepsilon\left(\partial^\alpha\left(q_0(v\times B)\cdot\nabla_v{\{\bf I- P\}}f\right),\partial^\alpha {\{\bf I- P\}}f\right)\nonumber\\
&=&\frac1\varepsilon\left(\partial^\alpha\left(q_0(v\times B)\cdot\nabla_v{\bf P}f\right),\partial^\alpha {\{\bf I- P\}}f\right)\nonumber\\
&&+\frac1\varepsilon\left(\partial^\alpha\left(q_0(v\times B)\cdot\nabla_v{\{\bf I- P\}}f\right),\partial^\alpha {\bf P}f\right)\nonumber\\
&&+\frac1\varepsilon\left(\partial^\alpha\left(q_0(v\times B)\cdot\nabla_v{\{\bf I- P\}}f\right),\partial^\alpha {\{\bf I- P\}}f\right).\label{B-without-w}
\end{eqnarray}
where we use the fact that
$$\frac1\varepsilon\left(\partial^\alpha\left(q_0(v\times B)\cdot\nabla_v{\bf P}f\right),\partial^\alpha {\bf P}f\right)=0,$$
due to the kernel structure of ${\bf P}$ and the integral of oddness function with respect to velocity $v$ over $\mathbb{R}^3_v$.

Using various Sobolev inequalities and Cauchy inequality, we can get that the first two terms on the right-hand side of \eqref{B-without-w} can be controlled by
\begin{eqnarray}
&&\frac1\varepsilon\left(\partial^\alpha\left(q_0(v\times B)\cdot\nabla_v{\bf P}f\right),\partial^\alpha {\{\bf I- P\}}f\right)\nonumber\\
&&+\frac1\varepsilon\left(\partial^\alpha\left(q_0(v\times B)\cdot\nabla_v{\{\bf I- P\}}f\right),\partial^\alpha {\bf P}f\right)\nonumber\\
&\lesssim&\|B\|^2_{H^N_x}(t)\mathcal{D}_{N}(t)+\frac{\eta}{{\varepsilon^2}}\sum_{|\alpha'|\leq \alpha}\|\partial^{\alpha'}{\{\bf I- P\}}f\|^2_\sigma,
\end{eqnarray}
as for the last term on the right-hand side of \eqref{B-without-w}, by using a similar way as the estimates on $\left(\partial^\alpha\left(q_0 E\cdot\nabla_vf\right),\partial^\alpha f\right)$, we can deduce that
\begin{eqnarray*}
&&\frac1\varepsilon\left(\partial^\alpha\left(q_0(v\times B)\cdot\nabla_v{\{\bf I- P\}}f\right),\partial^\alpha {\{\bf I- P\}}f\right)\\
&=&\frac1\varepsilon\left(\left(q_0(v\times B)\cdot\nabla_v\partial^\alpha{\{\bf I- P\}}f\right),\partial^\alpha {\{\bf I- P\}}f\right)\\
&&+\sum_{1\leq|\alpha_1|\leq |\alpha|}\frac1\varepsilon\left(\left(q_0(v\times \partial ^{\alpha_1}B)\cdot\nabla_v\partial^{\alpha-\alpha_1}{\{\bf I- P\}}f\right),\partial^\alpha {\{\bf I- P\}}f\right)\\
&=&-\sum_{1\leq|\alpha_1|\leq |\alpha|}\frac1\varepsilon\left(q_0(v\times \partial ^{\alpha_1}B)\cdot\partial^{\alpha-\alpha_1}{\{\bf I- P\}}f,\partial^\alpha \nabla_v{\{\bf I- P\}}f\right)\\
&\lesssim&\sum_{1\leq|\alpha_1|\leq |\alpha|}\||\partial ^{\alpha_1}B||\partial^{\alpha-\alpha_1}{\{\bf I- P\}}f|\langle v\rangle^{\frac52}\|^2+\frac{\eta}{{\varepsilon^2}}\|\partial^\alpha \nabla_v{\{\bf I- P\}}f\langle v\rangle^{-\frac32}\|^2\\
&\lesssim&\left\|\nabla^2 B\right\|^2_{H^{ N-2}_x}\mathcal{D}_{N-1,l}(t)
+\frac{\eta}{{\varepsilon^2}}\|\partial^\alpha{\{\bf I- P\}}f\|^2_\sigma,
\end{eqnarray*}
where we take $l\geq N$.

Now we arrive at
\begin{eqnarray}
  &&\frac1\varepsilon\left(\partial^\alpha\left(q_0(v\times B)\cdot\nabla_vf\right),\partial^\alpha f\right)\nonumber\\
&\lesssim&\|B\|^2_{H^N_x}(t)\mathcal{D}_{N}(t)+\|B\|^2_{H^N_x}(t)\left\{\mathcal{D}_{N}(t)+\mathcal{D}_{N-1,l}(t)\right\}
+\frac{\eta}{{\varepsilon^2}}\sum_{|\alpha'|\leq \alpha}\|\partial^{\alpha'}{\{\bf I- P\}}f\|^2_\sigma,
\end{eqnarray}
provided that $l\geq N$.

By using Lemma \ref{Lemma2.1} and Cauchy inequality, one can get that the last term on the right-hand side of \eqref{B-without-w} can be dominated by
$$
\lesssim \mathcal{E}_{N}(t)\mathcal{D}_N(t)+\frac{\eta}{{\varepsilon^2}}\sum_{\alpha'\leq \alpha}\|\partial^{\alpha'}{\{\bf I- P\}}f\|^2_\sigma.
$$
Now we arrive at by collecting the above estimates
\begin{eqnarray}\label{E1}
&&\frac{d}{dt}\sum_{1\leq|\alpha|\leq N}\left(\left\|\partial^\alpha  f\right\|^2+\left\|\partial^\alpha[E,B]\right\|^2\right)
+\sum_{1\leq|\alpha|\leq N}\frac1{\varepsilon^2}\left\|\partial^\alpha\{{\bf I-P}\}f\right\|^2_\sigma\nonumber\\
&\lesssim&\| E\|^{2}_{L^{\infty}_x}\left\|\langle v\rangle^{\frac 32}\nabla^N_xf\right\|^2+\mathcal{E}_N(t)\left\{\mathcal{D}_{N}(t)+\mathcal{D}_{N-1,l}(t)\right\}
+\eta\mathcal{D}_{N}(t).
\end{eqnarray}
where we ask that  $l\geq N$.

By plugging \eqref{0-order} and \eqref{E1} into \eqref{spatial-without}, one has \eqref{lemma3.7-1}.
Thus the proof of Lemma \ref{lemma3.7} is complete.
\end{proof}

For brevity, we use $\mathcal{D}_{N,mac}(t)$ to denote that
\begin{equation}\label{def-mac}
\mathcal{D}_{N,mac}(t)\equiv\left\|\nabla_x[\rho_\pm,u,\theta]\right\|^2_{H^{N-1}_x}+\|\rho_+-\rho_-\|^2+\|E\|^2_{H^{N-1}_x}+\|\nabla_xB\|^2_{H^{N-2}_x}.
\end{equation}

The next lemma is concerned with the macro dissipation $\mathcal{D}_{N,mac}(t)$ in the following.

\begin{lemma}\label{mac-dissipation}
There exists an interactive energy functional $\mathcal{E}^{int}_N(t)$ satisfying
$$
\mathcal{E}^{int}_N(t)\lesssim\sum_{|\alpha|\leq N}\left\|\partial^\alpha[f,E,B]\right\|^2$$
such that
\begin{equation}\label{mac-dissipation-1}
\frac{d}{dt}\mathcal{E}^{int}_N(t)+\mathcal{D}_{N,mac}(t)\lesssim\frac1{\varepsilon^2}\sum_{|\alpha|\leq N}\left\|\partial^\alpha\{{\bf I-P}\}f\right\|_\sigma^2+\mathcal{E}_N(t)\mathcal{D}_N(t)
\end{equation}
holds for any $t\in[0,T]$.
\end{lemma}
\begin{proof}
To this end, for any solution $f(t,x,v)$ of the VML system (\ref{VML-drop-eps}), by applying the macro-micro decomposition $(\ref{macro-micro})$ introduced in \cite{Guo-IUMJ-2004} and by defining moment functions $\mathcal{A}_{mj}(f)$ and $\mathcal{B}_j(f),~1\leq m,j\leq3,$ by
\begin{equation*}
\mathcal{A}_{mj}(f)=\int_{{\mathbb{R}}^3}\left( v_m v_j-1\right){M}^{1/2}fd v,\quad \mathcal{B}_j(f)=\frac{1}{10}\int_{{\mathbb{R}}^3}\left(| v|^2-5\right) v_j{M}^{1/2}fd v
\end{equation*}
as in \cite{DUAN-STAIN-ARMA-2011}, one can then derive from $(\ref{VML-drop-eps})$ a fluid-type system of equations
\begin{equation}\label{Macro-equation}
\begin{cases}
\partial_t\rho_\pm+\frac1\varepsilon\nabla_x\cdot u+\frac1\varepsilon\nabla_x\cdot
\left\langle v{M}^{1/2},\{{\bf I_\pm-P_\pm}\}f\right\rangle=\left\langle {M}^{1/2}, g_\pm\right\rangle,\\[2mm]
\partial_t\left(u_i+\left\langle v_i{M}^{1/2},\{{\bf I_\pm-P_\pm}\}f\right\rangle\right)+\frac1\varepsilon\partial_i\left(\rho_\pm+2\theta\right)\mp \frac1\varepsilon E_i\\[2mm]
\quad\quad+\frac1\varepsilon\nabla_x\cdot\left\langle vv_i{M}^{1/2}, \{{\bf I_\pm-P_\pm}\}f \right\rangle=\left\langle v_i{M}^{1/2}, g_\pm+\frac1{\varepsilon^2} L_\pm f\right\rangle,\\[2mm]
\partial_t\left(\theta+\frac16\left\langle (|v|^2-3){M}^{1/2},\{{\bf I_\pm-P_\pm}\}f\right\rangle\right)+\frac1{3\varepsilon}\nabla_x\cdot u\\[2mm]
\quad\quad+\frac1{6\varepsilon}\nabla_x\left\langle (|v|^2-3)v{M}^{1/2},\{{\bf I_\pm-P_\pm}\}f\right\rangle=\left\langle(|v|^2-3){M}^{1/2},g_\pm-\frac1{\varepsilon^2}L_\pm f \right\rangle,
\end{cases}
\end{equation}
and
\begin{equation}\label{Micro-equation}
\begin{cases}
\partial_t[\mathcal{A}_{ii}(\{{\bf I_\pm-P_\pm}\}f)+2\theta]+2\frac1\varepsilon\partial_iu_i
=\mathcal{A}_{ii}(r_\pm+g_\pm),\\[2mm]
\partial_t\mathcal{A}_{ij}(\{{\bf I_\pm-P_\pm}\}f)+\frac1\varepsilon\partial_ju_i+\frac1\varepsilon\partial_iu_j
+\frac1\varepsilon\nabla_x\cdot \langle v{M}^{1/2},\{{\bf I_\pm-P_\pm}\}f\rangle
=\mathcal{A}_{ij}(r_\pm+g_\pm),\\[2mm]
\partial_t \mathcal{B}_{j}(\{{\bf I_\pm-P_\pm}\}f)+\frac1\varepsilon\partial_j\theta=\mathcal{B}_j(r_\pm+g_\pm),
\end{cases}
\end{equation}
where
\begin{eqnarray}\label{r-G}
  r_\pm&=&- \frac1\varepsilon v\cdot\nabla_x\{{\bf I_\pm-P_\pm}\}f-\frac1{\varepsilon^2}{L_\pm}f,\nonumber\\
  \quad g_\pm&=&\frac12 v\cdot E f_\pm\mp (E+\frac1\varepsilon v\times B)\cdot\nabla_{ v}f_\pm+\frac1\varepsilon\mathscr{T}_\pm(f,f).
\end{eqnarray}
Setting
$$
G\equiv\left\langle v{M}^{1/2},\{{\bf I-P}\}f \cdot q_1 \right\rangle
$$
and noticing $\langle {M}^{1/2}, g_\pm\rangle=0$, we can get from (\ref{Macro-equation})-(\ref{Micro-equation}) that
\begin{equation}\label{Macro-equation1}
\begin{cases}
\partial_t\left(\frac{\rho_++\rho_-}2\right)+\frac1\varepsilon\nabla_x\cdot u=0\\[2mm]
\partial_tu_i+\frac1\varepsilon\partial_i\left(\frac{\rho_++\rho_-}2+2\theta\right)+\frac1{2\varepsilon}\sum\limits_{j=1}^3\partial_j\mathcal{A}_{ij}(\{{\bf I-P}\}f\cdot [1,1])=\frac{\rho_+-\rho_-}{2}E_i+\frac1\varepsilon[G\times B]_i,\\[2mm]
\partial_t\theta+\frac1{3\varepsilon}\nabla_x\cdot u+\frac5{6\varepsilon}\sum\limits_{i=1}^3\partial_i \mathcal{B}_i(\{{\bf I-P}\}f\cdot [1,1])=\frac16 G\cdot E,
\end{cases}
\end{equation}
and
\begin{equation}\label{Micro-equation1}
\begin{cases}
\partial_t[\frac12\mathcal{A}_{ij}(\{{\bf I-P}\}f\cdot [1,1])+2\theta\delta_{ij}]+\frac1\varepsilon\partial_ju_i+\frac1\varepsilon\partial_iu_j
=\frac12\mathcal{A}_{ij}(r_++r_-+g_++g_-),\\[2mm]
\frac12\partial_t \mathcal{B}_{j}(\{{\bf I-P}\}f\cdot [1,1])+\frac1\varepsilon\partial_j\theta=\frac12\mathcal{B}_{j}(r_++r_-+g_++g_-).
\end{cases}
\end{equation}
Moreover, by using the third equation of (\ref{Macro-equation1}) to replace $\partial_t\theta$ in the first equation of (\ref{Micro-equation1}), one has
\begin{equation}\label{Micro-equation2}
\begin{split}
&\frac12\partial_t\mathcal{A}_{ij}(\{{\bf I-P}\}f\cdot [1,1])+\frac1\varepsilon\partial_ju_i+\frac1\varepsilon\partial_iu_j-\frac2{3\varepsilon}\delta_{ij}\nabla_x\cdot u\\[2mm]
&-\frac5{3 \varepsilon}\delta_{ij}\nabla_x\cdot \mathcal{B}(\{{\bf I-P}\}f\cdot [1,1])=\frac12\mathcal{A}_{ij}(r_++r_-+g_++g_-)-\frac13\delta_{ij}G\cdot E.
\end{split}
\end{equation}
In order to further obtain the dissipation rate related to $\rho_\pm$ from the formula
\[
\rho_+^2+\rho_-^2=\frac{|\rho_++\rho_-|^2}{2}+\frac{|\rho_+-\rho_-|^2}{2},
\]
we need to consider the dissipation of $\rho_+-\rho_-$. For that purpose, one can get from $(\ref{Macro-equation1})_1$ and $(\ref{Macro-equation1})_2$ that
\begin{equation}\label{a_+-a_--original}
\begin{cases}
\partial_t(\rho_+-\rho_-)+\frac1\varepsilon\nabla_x\cdot G=0,\\[2mm]
\partial_tG+\frac1\varepsilon\nabla_x(\rho_+-\rho_-)-\frac2\varepsilon E+\frac1\varepsilon\nabla_x\cdot \mathcal{A}(\{{\bf I-P}\}f\cdot q_1)\\[2mm]
\qquad \qquad=E(\rho_++\rho_-)+\frac2\varepsilon u\times B+\left\langle [v,-v]{M}^{1/2},\frac1{\varepsilon^2}\mathscr{L}f+\frac1\varepsilon\mathscr{T}(f,f)\right\rangle.
\end{cases}
\end{equation}

Applying $\partial^\alpha$ to $(\ref{Micro-equation1})_2  $, and multiplying to the identity with $\varepsilon\partial^\alpha\partial_j \theta $, and integrating the identity result over $\mathbb{R}^3_x$, one has
\begin{eqnarray}
&&\left\|\partial^\alpha\nabla_x \theta\right\|^2=\sum_{j=1}^3\left(\frac1\varepsilon\partial^\alpha\partial_j\theta, \varepsilon\partial^\alpha\partial_j\theta\right)\nonumber\\
&\lesssim&-\sum_{j=1}^3\left(\frac12\partial_t \partial^\alpha \mathcal{B}_{j}(\{{\bf I-P}\}f\cdot [1,1]),\varepsilon\partial^\alpha\partial_j\theta\right)+\sum_{j=1}^3\left(\frac12\partial^\alpha \mathcal{B}_{j}(r_++r_-+g_++g_-),\varepsilon\partial^\alpha\partial_j\theta\right)\nonumber\\
&=&-\frac{d}{dt}\sum_{j=1}^3\left(\frac12 \partial^\alpha \mathcal{B}_{j}(\{{\bf I-P}\}f\cdot [1,1]),\varepsilon\partial^\alpha\partial_j\theta\right)+\sum_{j=1}^3\left(\frac12 \partial^\alpha\mathcal{ B}_{j}(\{{\bf I-P}\}f\cdot [1,1]),\varepsilon\partial^\alpha\partial_j\partial_t\theta\right)\nonumber\\
&&+\sum_{j=1}^3\left(\frac12\partial^\alpha \mathcal{B}_{j}(r_++r_-+g_++g_-),\varepsilon\partial^\alpha\partial_j\theta\right).\label{4.18}
\end{eqnarray}

For the second term on the right hand of the second equality in \eqref{4.18}, we have from $(\ref{Macro-equation1})_3$ that
\begin{equation*}
\begin{aligned}
&\sum_{j=1}^3\left(\frac12 \partial^\alpha \mathcal{B}_{j}(\{{\bf I-P}\}f\cdot [1,1]),\varepsilon\partial^\alpha\partial_j\partial_t\theta\right)\\
=&\bigg(\frac12 \partial^\alpha \mathcal{B}_{j}(\{{\bf I-P}\}f\cdot [1,1]),\varepsilon\partial^\alpha\partial_j\bigg\{-\frac1{3\varepsilon}\nabla_x\cdot u-\frac5{6\varepsilon}\sum\limits_{i=1}^3\partial_i \mathcal{B}_i(\{{\bf I-P}\}f\cdot [1,1])+\frac16 G\cdot E\bigg\}\bigg)\\
\lesssim&\eta\left\|\partial^\alpha\nabla_x u\right\|^2+\left\|\partial^{\alpha}\{{\bf I-P}\}f\right\|_{H^1_xL^2_\sigma}+\mathcal{E}_{N}(t)\mathcal{D}_{N}(t),
\end{aligned}
\end{equation*}
while for the last term on the right hand of (\ref{4.18}), we can deduce that
\begin{equation*}
\begin{aligned}
&\sum_{j=1}^3\left(\frac12\partial^\alpha \mathcal{B}_{j}(r_++r_-+g_++g_-),\varepsilon\partial^\alpha\partial_j\theta\right)\\[2mm]
\lesssim&\eta\left\|\partial^\alpha\nabla_x u\right\|^2+\left\|\partial^{\alpha}\{{\bf I-P}\}f\right\|_{H^1_xL^2_\sigma}+\mathcal{E}_{N}(t)\mathcal{D}_{N}(t).
\end{aligned}
\end{equation*}
Thus we can get that
\begin{equation}\label{G-c}
\begin{aligned}
\frac{d}{dt}G^{f}_\theta(t)+\left\|\partial^\alpha\nabla_x \theta\right\|^2\lesssim\eta\left\|\partial^\alpha\nabla_x u\right\|^2+\left\|\partial^{\alpha}\{{\bf I-P}\}f\right\|_{H^1_xL^2_\sigma}+\mathcal{E}_{N}(t)\mathcal{D}_{N}(t).
\end{aligned}
\end{equation}
Here
\[
G^{f}_\theta(t)\equiv\sum_{j=1}^3\left(\frac12 \partial^\alpha \mathcal{B}_{j}(\{{\bf I-P}\}f\cdot [1,1]),\varepsilon\partial^\alpha\partial_j\theta\right).
\]
On the other hand, by using \eqref{Micro-equation2}, one has
\begin{eqnarray}
  &&-\left\{\frac12\partial_i\partial_t\mathcal{A}_{ii}(\{{\bf I-P}\}f\cdot [1,1])+\sum_j\partial_j\partial_t\mathcal{A}_{ji}(\{{\bf I-P}\}f\cdot [1,1])\right\}\nonumber\\
  &&-\frac2\varepsilon\Delta u_i-\frac2\varepsilon\partial_i\partial_i u_i+\frac5{ \varepsilon}\partial_i\nabla_x\cdot \mathcal{B}(\{{\bf I-P}\}f\cdot [1,1])\nonumber\\
  &=&-\frac12\partial_i\mathcal{A}_{ii}(r_++r_-+g_++g_-)-\sum_j\partial_j\mathcal{A}_{ji}(r_++r_-+g_++g_-)+\partial_i\left[G\cdot E\right].
\end{eqnarray}
Applying $\partial^\alpha$ to the above equality, multiplying it with $\epsilon\partial^\alpha u_i$, and integrating the identity result over $\mathbb{R}^3_x$, then one has
\begin{eqnarray*}
&&2\|\partial^\alpha \nabla_xu\|^2+2\sum_i\|\partial^\alpha \partial_iu_i\|^2\nonumber\\
&=&\sum_i\left(-\frac2\varepsilon\partial^\alpha \Delta u_i, \varepsilon\partial^\alpha u_i\right)+\sum_i
\left(-\frac2\varepsilon\partial^\alpha\partial_i\partial_i u_i, \varepsilon\partial^\alpha u_i\right)\nonumber\\
&=&\sum_i\left(-\frac2\varepsilon\partial^\alpha \Delta u_i-\frac2\varepsilon\partial^\alpha\partial_i\partial_i u_i, \varepsilon\partial^\alpha u_i\right)\nonumber\\
&=&\sum_i\left(\frac12\partial^\alpha\partial_i\partial_t\mathcal{A}_{ii}(\{{\bf I-P}\}f\cdot [1,1])+\sum_j\partial^\alpha\partial_j\partial_t\mathcal{A}_{ji}(\{{\bf I-P}\}f\cdot [1,1]), \varepsilon\partial^\alpha u_i\right)\nonumber\\
&&+\sum_i\left(-\frac12\partial^\alpha\partial_i\mathcal{A}_{ii}(r_++r_-+g_++g_-)
-\sum_j\partial^\alpha\partial_j\mathcal{A}_{ji}(r_++r_-+g_++g_-), \varepsilon\partial^\alpha u_i\right)\nonumber\\
&&+\sum_i\left(\partial^\alpha\partial_i\left[G\cdot E\right]-\frac5{ \varepsilon}\partial^\alpha\partial_i\nabla_x\cdot \mathcal{B}(\{{\bf I-P}\}f\cdot [1,1]), \varepsilon\partial^\alpha u_i\right)\nonumber\\
&=&\frac{d}{dt}\sum_i\left(\frac12\partial^\alpha\partial_i\mathcal{A}_{ii}(\{{\bf I-P}\}f\cdot [1,1])+\sum_j\partial^\alpha\partial_j\mathcal{A}_{ji}(\{{\bf I-P}\}f\cdot [1,1]), \varepsilon\partial^\alpha u_i\right)\nonumber\\
&&-\sum_i\left(\frac12\partial^\alpha\partial_i\mathcal{A}_{ii}(\{{\bf I-P}\}f\cdot [1,1])+\sum_j\partial^\alpha\partial_j\mathcal{A}_{ji}(\{{\bf I-P}\}f\cdot [1,1]), \varepsilon\partial^\alpha \partial_tu_i\right)\nonumber\\
&&+\sum_i\left(-\frac12\partial^\alpha\partial_i\mathcal{A}_{ii}(r_++r_-+g_++g_-)
-\sum_j\partial^\alpha\partial_j\mathcal{A}_{ji}(r_++r_-+g_++g_-), \varepsilon\partial^\alpha u_i\right)\nonumber\\
&&+\sum_i\left(\partial^\alpha\partial_i\left[G\cdot E\right]-\frac5{ \varepsilon}\partial^\alpha\partial_i\nabla_x\cdot \mathcal{B}(\{{\bf I-P}\}f\cdot [1,1]), \varepsilon\partial^\alpha u_i\right)\nonumber\\
&=&\frac{d}{dt}\sum_i\left(\frac12\partial^\alpha\partial_i\mathcal{A}_{ii}(\{{\bf I-P}\}f\cdot [1,1])+\sum_j\partial^\alpha\partial_j\mathcal{A}_{ji}(\{{\bf I-P}\}f\cdot [1,1]), \varepsilon\partial^\alpha u_i\right)\nonumber\\
&&+\sum_i\left(\frac12\partial^\alpha\partial_i\mathcal{A}_{ii}(\{{\bf I-P}\}f\cdot [1,1])+\sum_j\partial^\alpha\partial_j\mathcal{A}_{ji}(\{{\bf I-P}\}f\cdot [1,1]), \partial^\alpha \left\{\partial_i\big(\frac{\rho_++\rho_-}2+2c\big)\right\}\right)\nonumber\\
&&-\sum_i\left(\frac12\partial^\alpha\partial_i\mathcal{A}_{ii}(\{{\bf I-P}\}f\cdot [1,1])+\sum_j\partial^\alpha\partial_j\mathcal{A}_{ji}(\{{\bf I-P}\}f\cdot [1,1]), \varepsilon\partial^\alpha \left\{\frac{\rho_+-\rho_-}{2}E_i+\frac1\varepsilon[G\times B]_i\right\}\right)\nonumber\\
&&+\sum_i\left(\frac12\partial^\alpha\partial_i\mathcal{A}_{ii}(\{{\bf I-P}\}f\cdot [1,1])+\sum_j\partial^\alpha\partial_j\mathcal{A}_{ji}(\{{\bf I-P}\}f\cdot [1,1]), \frac1{2}\partial^\alpha \partial_i\big(\sum\limits_{j=1}^3\partial_j\mathcal{A}_{ij}(\{{\bf I-P}\}f\cdot [1,1])\right)\nonumber\\
&&+\sum_i\left(-\frac12\partial^\alpha\partial_i\mathcal{A}_{ii}(r_++r_-+g_++g_-)
-\sum_j\partial^\alpha\partial_j\mathcal{A}_{ji}(r_++r_-+g_++g_-), \varepsilon\partial^\alpha u_i\right)\nonumber\\
&&+\sum_i\left(\partial^\alpha\partial_i\left[G\cdot E\right]-\frac5{ \varepsilon}\partial^\alpha\partial_i\nabla_x\cdot \mathcal{B}(\{{\bf I-P}\}f\cdot [1,1]), \varepsilon\partial^\alpha u_i\right)\nonumber
\end{eqnarray*}
The second term on the right-hand side of the above equality can be bounded by
\begin{eqnarray}
&&\sum_i\left(\frac12\partial^\alpha\partial_i\mathcal{A}_{ii}(\{{\bf I-P}\}f\cdot [1,1])+\sum_j\partial^\alpha\partial_j\mathcal{A}_{ji}(\{{\bf I-P}\}f\cdot [1,1]), \partial^\alpha \left\{\partial_i\big(\frac{\rho_++\rho_-}2+2c\big)\right\}\right)\nonumber\\
&\lesssim&\eta\|\partial^{\alpha}\nabla_x\{\rho_++\rho_-\}\|^2+\eta\|\partial^{\alpha}\nabla_xc\|^2
+\left\|\partial^{\alpha}\nabla_x\{{\bf I-P}\}f\right\|_\sigma^2.
\end{eqnarray}
The other terms can be bounded by
\begin{eqnarray}
  &\lesssim&\eta\|\partial^\alpha\nabla_x u\|^2+\left\|\partial^{\alpha}\nabla_x\{{\bf I-P}\}f\right\|_\sigma^2+\mathcal{E}_N(t)\mathcal{D}_N(t).
\end{eqnarray}
Consequently, one has
\begin{eqnarray}\label{G-b}
&&\frac{d}{dt}G^{f}_u(t)+\|\partial^\alpha \nabla_xu\|^2+\sum_i\|\partial^\alpha \partial_iu_i\|^2\nonumber\\
 &\lesssim&\eta\|\partial^{\alpha}\nabla_x\{\rho_++\rho_-\}\|^2+\eta\|\partial^{\alpha}\nabla_xc\|^2+\left\|\partial^{\alpha}\nabla_x\{{\bf I-P}\}f\right\|_\sigma^2+\mathcal{E}_N(t)\mathcal{D}_N(t),
\end{eqnarray}
where $G^{f}_u(t)$ denotes that
\[G^{f}_u(t)\equiv\sum_i\left(\frac12\partial^\alpha\partial_i\mathcal{A}_{ii}(\{{\bf I-P}\}f\cdot [1,1])+\sum_j\partial^\alpha\partial_j\mathcal{A}_{ji}(\{{\bf I-P}\}f\cdot [1,1]), \varepsilon\partial^\alpha u_i\right).\]
Next, we estimate $\rho_++\rho_-$. To this end, we have from $(\ref{Macro-equation1})_2$ and by employing the same argument to deduce $(\ref{G-c})$ that
\begin{equation}\label{G-a}
\frac{d}{dt}G^{f}_\rho(t)+\|\partial^\alpha\nabla_x(\rho_++\rho_-)\|^2\lesssim\|\partial^\alpha\nabla_xu\|^2
+\|\partial^\alpha\nabla_x\theta\|^2
+\left\|\partial^{\alpha}\nabla_x\{{\bf I-P}\}f\right\|_\sigma^2+\mathcal{E}_N(t)\mathcal{D}_N(t)
\end{equation}
with
$$
G^f_{\rho}(t)\equiv\sum_{i}\left(\partial^\alpha u_i, \epsilon\partial^\alpha\partial_i(\rho_++\rho_-)\right).
$$

Set
\begin{equation*}
G^f_{\rho_++\rho_-,u,\theta}(t)=G^{f}_u(t)+\kappa_1G^{f}_\theta(t)+\kappa_2G^{f}_\rho(t), , 0<\kappa_2\ll\kappa_1\ll1.
\end{equation*}
we can deduce from (\ref{G-c}), (\ref{G-b}), and (\ref{G-a}) that
\begin{equation}\label{g-f-}
\begin{aligned}
\frac{d}{dt}G^f_{\rho_++\rho_-,u,\theta}(t)+\left\|\partial^{\alpha}\nabla_x[\rho_++\rho_-,u,\theta]\right\|^2
\lesssim\left\|\{{\bf I-P}\}f\right\|^2_{H^N_xL^2_\sigma}+\mathcal{E}_{N}(t)\mathcal{D}_{N}(t).
\end{aligned}
\end{equation}
Finally, for the corresponding estimate on $\rho_+-\rho_-$, we have from $\eqref{a_+-a_--original}_2$ that
\begin{eqnarray}
&&\left\|\partial^{\alpha}\nabla_x(\rho_+-\rho_-)\right\|^2+2\left\|\partial^{\alpha}(\rho_+-\rho_-)\right\|^2
=\left(\partial^{\alpha}\nabla_x(\rho_+-\rho_-)-2\partial^\alpha E,\partial^{\alpha}\nabla_x(\rho_+-\rho_-)\right)\nonumber\\
&=&\left(\partial^{\alpha}\left\{\frac1\epsilon\nabla_x(\rho_+-\rho_-)-\frac2\epsilon E\right\},\varepsilon\partial^{\alpha}\nabla_x(\rho_+-\rho_-)\right)\nonumber\\
&=&\left(\partial^{\alpha}\left\{-\partial_tG-\frac1\varepsilon\nabla_x\cdot A(\{{\bf I-P}\}f\cdot q_1)\right\},\varepsilon\partial^{\alpha}\nabla_x(\rho_+-\rho_-)\right)\nonumber\\
&&+\left(\partial^{\alpha}\left\{E(\rho_++\rho_-)+\frac2\varepsilon u\times B+\left\langle [v,-v]{M}^{1/2},\frac1{\varepsilon^2}\mathscr{L}f+\frac1\varepsilon\mathscr{T}(f,f)\right\rangle\right\},\varepsilon\partial^{\alpha}\nabla_x(\rho_+-\rho_-)\right)\nonumber\\
&=&-\frac{d}{dt}\left(\partial^{\alpha}G,\varepsilon\partial^{\alpha}\nabla_x(\rho_+-\rho_-)\right)
+\left(\partial^{\alpha}G,\varepsilon\partial^{\alpha}\nabla_x\partial_t(\rho_+-\rho_-)\right)\nonumber\\
&&+\left(\partial^{\alpha}\left\{-\frac1\varepsilon\nabla_x\cdot \mathcal{A}(\{{\bf I-P}\}f\cdot q_1)\right\},\varepsilon\partial^{\alpha}\nabla_x(\rho_+-\rho_-)\right)\nonumber\\
&&+\left(\partial^{\alpha}\left\{E(\rho_++\rho_-)+\frac2\varepsilon u\times B+\left\langle [v,-v]{M}^{1/2},\frac1{\varepsilon^2}\mathscr{L}f+\frac1\varepsilon\mathscr{T}(f,f)\right\rangle\right\},\varepsilon\partial^{\alpha}\nabla_x(\rho_+-\rho_-)\right)\nonumber\\
&=&-\frac{d}{dt}\left(\partial^{\alpha}G,\varepsilon\partial^{\alpha}\nabla_x(\rho_+-\rho_-)\right)
-\left(\partial^{\alpha}G,\varepsilon\partial^{\alpha}\nabla_x
\frac1\varepsilon\nabla_x\cdot G\right)\nonumber\\
&&+\left(\partial^{\alpha}\left\{-\frac1\varepsilon\nabla_x\cdot \mathcal{A}(\{{\bf I-P}\}f\cdot q_1)\right\},\varepsilon\partial^{\alpha}\nabla_x(\rho_+-\rho_-)\right)\nonumber\\
&&+\left(\partial^{\alpha}\left\{E(\rho_++\rho_-)+\frac2\varepsilon u\times B+\left\langle [v,-v]{M}^{1/2},\frac1{\varepsilon^2}\mathscr{L}f+\frac1\varepsilon\mathscr{T}(f,f)\right\rangle\right\},\varepsilon\partial^{\alpha}\nabla_x(\rho_+-\rho_-)\right)\nonumber\\
&=&-\frac{d}{dt}\left(\partial^{\alpha}G,\varepsilon\partial^{\alpha}\nabla_x(\rho_+-\rho_-)\right)+
\left(\partial^{\alpha}\nabla_x\cdot G,\partial^{\alpha}
\nabla_x\cdot G\right)\nonumber\\
&&+\left(\partial^{\alpha}\left\{-\frac1\varepsilon\nabla_x\cdot \mathcal{A}(\{{\bf I-P}\}f\cdot q_1)\right\},\varepsilon\partial^{\alpha}\nabla_x(\rho_+-\rho_-)\right)\nonumber\\
&&+\left(\partial^{\alpha}\left\{E(\rho_++\rho_-)+\frac2\varepsilon u\times B+\left\langle [v,-v]{M}^{1/2},\frac1{\varepsilon^2}\mathscr{L}f+\frac1\varepsilon\mathscr{T}(f,f)\right\rangle\right\},\varepsilon\partial^{\alpha}\nabla_x(\rho_+-\rho_-)\right).\nonumber\\
\end{eqnarray}
Consequently
\begin{eqnarray}\label{g-a}
&&\frac{d}{dt}\left(\partial^{\alpha}G,\varepsilon\partial^{\alpha}\nabla_x(\rho_+-\rho_-)\right)
+\left\|\partial^{\alpha}\nabla_x(\rho_+-\rho_-)\right\|^2+2\left\|\partial^{\alpha}(\rho_+-\rho_-)\right\|^2\nonumber\\
&\lesssim&\left\|\partial^{\alpha}\nabla_x\{{\bf I-P}\}f\right\|_\sigma^2+\mathcal{E}_N(t)\mathcal{D}_N(t).
\end{eqnarray}

A suitable linear combination of (\ref{g-f-}) and (\ref{g-a}) gives
\begin{equation}\label{mac-alpha}
\begin{aligned}
&\frac{d}{dt}G^{f}_{\rho_+\pm\rho_-,u,\theta}(t)+
2\left\|\partial^{\alpha}(\rho_+-\rho_-)\right\|^2+\left\|\partial^{\alpha}\nabla_x[\rho_+\pm\rho_-,u,\theta]\right\|^2\\
\lesssim&\left\|\{{\bf I-P}\}f\right\|^2_{H^N_xL^2_\sigma}+\mathcal{E}_{N}(t)\mathcal{D}_{N}(t)
\end{aligned}
\end{equation}
where
\[
G^{f}_{\rho_+\pm\rho_-,u,\theta}(t)\equiv G^f_{\rho_++\rho_-,u,\theta}(t)+\left(\partial^{\alpha}G,\varepsilon\partial^{\alpha}\nabla_x(\rho_+-\rho_-)\right),
\]
where $|\alpha|\leq N-1$.

Since
\begin{eqnarray}
  \frac2\varepsilon E&=&\partial_tG+\frac1\varepsilon\nabla_x(\rho_+-\rho_-)+\frac1\varepsilon\nabla_x\cdot \mathcal{A}(\{{\bf I-P}\}f\cdot q_1)-E(\rho_++\rho_-)\nonumber\\
  &&-\frac2\varepsilon u\times B-\left\langle [v,-v]{M}^{1/2},\frac1{\varepsilon^2}\mathscr{L}f-\frac1\varepsilon\mathscr{T}(f,f)\right\rangle.
\end{eqnarray}
Applying $\partial^\alpha$ to the above equality and multiplying it with $\varepsilon \partial^\alpha E$, one has
\begin{eqnarray}
 && 2\|\partial^\alpha E\|^2=\left(\partial^\alpha\left\{ \frac2\varepsilon E\right\},\varepsilon\partial^\alpha E\right)\nonumber\\
 &=&\left(\partial^\alpha\left\{ \partial_tG+\frac1\varepsilon\nabla_x(\rho_+-\rho_-)+\frac1\varepsilon\nabla_x\cdot \mathcal{A}(\{{\bf I-P}\}f\cdot q_1)-E(\rho_++\rho_-)\right\},\varepsilon\partial^\alpha E\right)\nonumber\\
 &&+\left(\partial^\alpha\left\{ -\frac2\varepsilon u\times B-\left\langle [v,-v]{M}^{1/2},\frac1{\varepsilon^2}\mathscr{L}f-\frac1\varepsilon\mathscr{T}(f,f)\right\rangle\right\},\varepsilon\partial^\alpha E\right)\nonumber\\
 &=&\frac{d}{dt}\left(\partial^\alpha G,\varepsilon\partial^\alpha E\right)-\left(\partial^\alpha G,\varepsilon\partial^\alpha \left\{\nabla_x \times B  - \tfrac{1}{\eps}G\right\} \right)\nonumber\\
 &&+\left(\partial^\alpha\left\{\frac1\varepsilon\nabla_x(\rho_+-\rho_-)+\frac1\varepsilon\nabla_x\cdot \mathcal{A}(\{{\bf I-P}\}f\cdot q_1)-E(\rho_++\rho_-)\right\},\varepsilon\partial^\alpha E\right)\nonumber\\
 &&+\left(\partial^\alpha\left\{ -\frac2\varepsilon u\times B-\left\langle [v,-v]{M}^{1/2},\frac1{\varepsilon^2}\mathscr{L}f-\frac1\varepsilon\mathscr{T}(f,f)\right\rangle\right\},\varepsilon\partial^\alpha E\right)\nonumber\\
  &=&\frac{d}{dt}\left(\partial^\alpha G,\varepsilon\partial^\alpha E\right)+\|\partial^\alpha G\|^2-\left(\partial^\alpha G,\varepsilon\partial^\alpha \left\{\nabla_x \times B \right\} \right)\nonumber\\
 &&+\left(\partial^\alpha\left\{\frac1\varepsilon\nabla_x(\rho_+-\rho_-)+\frac1\varepsilon\nabla_x\cdot\mathcal{ A}(\{{\bf I-P}\}f\cdot q_1)-E(\rho_++\rho_-)\right\},\varepsilon\partial^\alpha E\right)\nonumber\\
 &&+\left(\partial^\alpha\left\{ -\frac2\varepsilon u\times B-\left\langle [v,-v]{M}^{1/2},\frac1{\varepsilon^2}\mathscr{L}f-\frac1\varepsilon\mathscr{T}(f,f)\right\rangle\right\},\varepsilon\partial^\alpha E\right)\nonumber\\
  &\lesssim&\frac{d}{dt}\left(\partial^\alpha G,\varepsilon\partial^\alpha E\right)-\|\partial^\alpha(\rho_+-\rho_-)\|^2+\|\nabla_xG\|_{H^{N}_x}^2+\eta\|\nabla_x B\|^2_{H^{N-2}_x}+\eta\|\partial^\alpha E\|^2\nonumber\\
 &&+\frac1{\varepsilon^2}\|\partial^\alpha\{{\bf I-P}\}f\|^2_\sigma+\|\partial^\alpha\nabla_x\{{\bf I-P}\}f\|^2_\sigma+\mathcal{E}_N(t)\mathcal{D}_N(t)
\end{eqnarray}
where $|\alpha|\leq N-1$.
Consequently
\begin{eqnarray}\label{E-1-s}
 &&-\frac{d}{dt}\left(\partial^\alpha G,\varepsilon\partial^\alpha E\right)+2\|\partial^\alpha E\|^2+\|\partial^\alpha(\rho_+-\rho_-)\|^2\nonumber\\
 &\lesssim&\|\nabla_xG\|_{H^{N}_x}^2+\eta\|\nabla_x B\|^2_{H^{N-2}_x}\nonumber\\
 &&+\frac1{\varepsilon^2}\|\partial^\alpha\{{\bf I-P}\}f\|^2_\sigma+\|\partial^\alpha\nabla_x\{{\bf I-P}\}f\|^2_\sigma+\mathcal{E}_N(t)\mathcal{D}_N(t).
\end{eqnarray}
For $B$, it follows that for $|\alpha|\leq N-2$
\begin{eqnarray}
&&\left\|\partial^{\alpha}\nabla_xB\right\|^2=\left(\partial^{\alpha}\nabla_xB,\partial^{\alpha}\nabla_xB\right)\nonumber\\
&=&\left(\partial^{\alpha}\nabla_x\times B,\partial^{\alpha}\nabla_x\times B\right)\nonumber\\
&=&\left(\partial^{\alpha}\left\{\partial_tE+\frac1\varepsilon G\right\},\partial^{\alpha}\nabla_x\times B\right)\nonumber\\
&=&\frac{d}{dt}\left(\partial^{\alpha}E,\partial^{\alpha}\nabla_x\times B\right)-\left(\partial^{\alpha}E,\partial^{\alpha}\nabla_x\times \partial_tB\right)+\left(\partial^{\alpha}\left\{\frac1\varepsilon G\right\},\partial^{\alpha}\nabla_x\times B\right)\nonumber\\
&=&\frac{d}{dt}\left(\partial^{\alpha}E,\partial^{\alpha}\nabla_x\times B\right)-\left(\partial^{\alpha}E,\partial^{\alpha}\nabla_x\times \left\{-\nabla_x\times E\right\}\right)+\left(\partial^{\alpha}\left\{\frac1\varepsilon G\right\},\partial^{\alpha}\nabla_x\times B\right)\nonumber\\
&=&\frac{d}{dt}\left(\partial^{\alpha}E,\partial^{\alpha}\nabla_x\times B\right)+\|\partial^{\alpha}\nabla_x\times E\|^2+\left(\partial^{\alpha}\left\{\frac1\varepsilon G\right\},\partial^{\alpha}\nabla_x\times B\right)\nonumber\\
&\lesssim&\frac{d}{dt}\left(\partial^{\alpha}E,\partial^{\alpha}\nabla_x\times B\right)+\|\partial^{\alpha}\nabla_x\times E\|^2+\eta\|\partial^{\alpha}\nabla_x\times B\|^2+\frac1{\varepsilon^2}\|\partial^{\alpha}G\|^2\nonumber
\end{eqnarray}
That is
\begin{eqnarray}\label{B-s}
&&-\frac{d}{dt}\left(\partial^{\alpha}E,\partial^{\alpha}\nabla_x\times B\right)+\left\|\partial^{\alpha}\nabla_xB\right\|^2\nonumber\\
&\lesssim&\|\partial^{\alpha}\nabla_x\times E\|^2+\eta\|\partial^{\alpha}\nabla_x\times B\|^2+\frac1{\varepsilon^2}\|\partial^{\alpha}G\|^2
\end{eqnarray}
for $|\alpha|\leq N-2$.

For sufficiently small $\kappa>0$, (\ref{E-1-s})$+\kappa$(\ref{B-s}) gives
\begin{equation}\label{EB-s-2}
\begin{aligned}
&\frac{d}{dt}G_{E,B}(t)+\left\|E\right\|_{H^{N-1}_x}^2+\left\|\nabla_xB\right\|_{H^{N-2}_x}^2+\left\|\rho_+-\rho_-\right\|^2_{H^{N-1}_x}\\[2mm]
\lesssim&\frac1{\varepsilon^2}\sum_{|\alpha|\leq N}\|\partial^\alpha\{{\bf I-P}\}f\|^2_{\sigma}+\mathcal{E}_N(t)\mathcal{D}_N(t).
\end{aligned}
\end{equation}
Here we set
\[
G_{E,B}(t)=-\left\{\sum_{|\alpha|\leq N-1}\left(\partial^\alpha G,\varepsilon\partial^\alpha E\right)+\kappa\sum_{|\alpha|\leq N-2}\left(\partial^{\alpha}E,\partial^{\alpha}\nabla_x\times B\right)\right\}.
\]
Recalling the definition of $\mathcal{D}_{N,mac}(t)$ in \eqref{def-mac}.
A proper combination of \eqref{mac-alpha} and \eqref{EB-s-2} gives \eqref{mac-dissipation-1}.
 This completes the proof of Lemma \ref{mac-dissipation}.
\end{proof}
{\bf Assumption 1:}
\begin{eqnarray}\label{Assump-1}
\sup_{0<t\leq T}\left\{(1+t)^{1+\epsilon_0}\left\|E\right\|_{L^\infty_x}^2+\mathcal{E}_{N}(t)\right\}\leq M_1
\end{eqnarray}
where $M_1$ is a sufficiently small positive constant.
\begin{proposition}
Under {\bf Assumption 1},
there exist an energy functional $\mathcal{E}_{N}(t)$ and the corresponding energy dissipation functional $\mathcal{D}_{N}(t)$ which satisfy (\ref{Energy-E-D-1}), (\ref{Energy-E-D-2}) respectively such that
\begin{eqnarray}\label{prof-1}
&&\frac{d}{dt}\mathcal{E}_{N}(t)+\mathcal{D}_{N}(t)\nonumber\\
&\lesssim &M_1(1+t)^{-(1+\epsilon_0)}\left\|\langle v\rangle^{\frac 32}\nabla^N_xf\right\|^2+\mathcal{E}_N(t)\mathcal{D}_{N-1,l}(t)
\end{eqnarray}
{\color{red}holds for all $0\leq t\leq T$.}
\end{proposition}
\begin{proof}
  A proper linear combination of \eqref{lemma3.7-1} and \eqref{mac-dissipation-1} gives \eqref{prof-1}.
\end{proof}
\section{Lyapunov inequality for the energy functional with weight}
To control the $$M_1(1+t)^{-(1+\epsilon_0)}\left\|\langle v\rangle^{\frac 32}\nabla^N_xf\right\|^2$$ on the right-hand side of \eqref{prof-1}, one need the energy estimate with the weight. However, the coercivity estimates on $\mathscr{L}$ tell us that
\begin{eqnarray*}
  &&\frac1{\varepsilon^2}\left(\mathscr{L}\partial^\alpha f,w^2_{l}(\alpha)\partial^\alpha f\right)\\
  &\gtrsim&\frac1{\varepsilon^2}\|w_{l}(\alpha)\partial^\alpha f\|_\sigma^2-\frac1{\varepsilon^2}\|\partial^\alpha f\|_\sigma^2\\
  &\gtrsim&\frac1{\varepsilon^2}\|w_{l}(\alpha)\partial^\alpha f\|_\sigma^2-\frac1{\varepsilon^2}\left\|\partial^\alpha {\bf P}f\right\|^2-\frac1{\varepsilon^2}\left\|\partial^\alpha \{{\bf I-P}\}f\right\|_\sigma^2,
\end{eqnarray*}
where the singularity term $\frac1{\varepsilon^2}\left\|\partial^\alpha {\bf P}f\right\|^2$ can not be controlled any more. To overcome this difficulty term, we have the following the result:

\begin{lemma}\label{lemma3.8} It holds that
\begin{eqnarray}\label{lemma3.8-1}
&&\frac{d}{dt}\sum_{|\alpha|=N}{\varepsilon^2}\left\|w_l(\alpha)\partial^\alpha f\right\|^2
+\frac{\vartheta q{\varepsilon^2}}{(1+t)^{1+\vartheta}}\sum_{|\alpha|=N}\left\|\langle v\rangle w_l(\alpha)\partial^\alpha f\right\|^2+\sum_{|\alpha|=N}\left\|w_l(\alpha)\partial^\alpha f\right\|^2_\sigma\nonumber\\
&\lesssim&\mathcal{D}_{N}(t)+\mathcal{E}_{N}(t)\mathcal{D}_{N-1,l}(t)+
\sum_{|\alpha|=N}{\varepsilon^2}\left\|\partial^\alpha E\right\|\left\|{M}^\delta\partial^\alpha f\right\|
\end{eqnarray}
{for all $0\leq t\leq T$.}
\end{lemma}
\begin{proof}
For this purpose, the standard energy estimate on $\partial^\alpha f$ with $|\alpha|= N$ weighted by the time-velocity dependent function $w_{l}(\alpha)(t,v)$ gives

\begin{equation}\label{N-f-w}
\begin{aligned}
&\frac{d}{dt}\left\|w_{l}(\alpha)\partial^\alpha f\right\|^2
+\frac{\vartheta q}{(1+t)^{1+\vartheta}}\left\|\langle v\rangle w_{l}(\alpha)\partial^\alpha f\right\|^2+\frac1{\varepsilon^2}\|w_{l}(\alpha)\partial^\alpha f\|_\sigma^2\\[2mm]
\lesssim&\frac1{\varepsilon^2}\left\|\partial^\alpha {\bf P}f\right\|^2+\frac1{\varepsilon^2}\left\|\partial^\alpha \{{\bf I-P}\}f\right\|_\sigma^2+
\left\|\partial^\alpha E\right\|\left\|{M}^\delta\partial^\alpha f\right\|\\
&+\left|\left(\partial^\alpha (E\cdot vf),w^2_{l}(\alpha)\partial^\alpha f\right)\right|+\left|\left(\partial^\alpha (E\cdot \nabla_vf),w^2_{l}(\alpha)\partial^\alpha f\right)\right|\\[2mm]
&+\frac1{\varepsilon}\left|\left(\partial^\alpha[ v\times B\cdot\nabla_v f],w^2_{l}(\alpha)\partial^\alpha f\right)\right|+\frac1{\varepsilon}\left|\left(\partial^\alpha \mathscr{T}(f,f),w^2_{l}(\alpha)\partial^\alpha f\right)\right|
\end{aligned}
\end{equation}
where we use the coercivity estimates on $\mathscr{L}$.

Due to the absence of the singularity factor $\frac1{\epsilon^2}$ in macroscopic dissipation term, to control $\frac1{\varepsilon^2}\left\|\partial^\alpha {\bf P}f\right\|^2$ on the right-hand side of \eqref{N-f-w}, we have to multiply ${\varepsilon^2}$ to \eqref{N-f-w} such that
\begin{equation}\label{alpha-f-1}
\begin{aligned}
&\frac{d}{dt}{\varepsilon^2}\left\|w_l(\alpha)\partial^\alpha f\right\|^2
+\frac{\vartheta q{\varepsilon^2}}{(1+t)^{1+\vartheta}}\left\|\langle v\rangle w_l(\alpha)\partial^\alpha f\right\|^2+\left\|w_l(\alpha)\partial^\alpha f\right\|^2_\sigma\\[2mm]
\lesssim&\left\|\partial^\alpha {\bf P}f\right\|^2+\left\|\partial^\alpha \{{\bf I-P}\}f\right\|_\sigma^2+
{\varepsilon^2}\left\|\partial^\alpha E\right\|\left\|{M}^\delta\partial^\alpha f\right\|\\
&+{\varepsilon^2}\left|\left(\partial^\alpha (E\cdot vf),w^2_l(\alpha)\partial^\alpha f\right)\right|+{\varepsilon^2}\left|\left(\partial^\alpha (E\cdot \nabla_vf),w^2_l(\alpha)\partial^\alpha f\right)\right|\\[2mm]
&+{\varepsilon}\left|\left(\partial^\alpha[ v\times B\cdot\nabla_v f],w^2_l(\alpha)\partial^\alpha f\right)\right|+{\varepsilon}\left|\left(\partial^\alpha \mathscr{T}(f,f),w^2_l(\alpha)\partial^\alpha f\right)\right|.
\end{aligned}
\end{equation}
By using Sobolev inequalities, one has
\begin{equation}\label{E-N-w}
\begin{aligned}
&{\varepsilon^2}\left|\left(\partial^\alpha (E\cdot vf),w^2_l(\alpha)\partial^\alpha f\right)\right|\\
\lesssim&{\varepsilon^2}\|E\|_{L^\infty_x}\left\|\langle v\rangle^{1/2}w_l(\alpha)\partial^{\alpha}f\right\|
\left\|\langle v\rangle^{1/2}w_l(\alpha)\partial^\alpha f\right\|\\[2mm]
&+\sum_{\substack{1\leq|\alpha_1|\leq N-2}}{\varepsilon^2}\left\|\partial^{\alpha_1}E\right\|_{L^\infty_x}
\left\|\langle v\rangle^{\frac12}w_l(\alpha)\partial^{\alpha-\alpha_1}f\right\|
\left\|\langle v\rangle^{-\frac12}w_l(\alpha)\partial^\alpha f\right\|\\[2mm]
&+\sum_{\substack{|\alpha_1|= N-1}}{\varepsilon^2}\left\|\partial^{\alpha_1}E\right\|_{L^3_x}
\left\|\langle v\rangle^{\frac12}w_l(\alpha)\partial^{\alpha-\alpha_1}f\right\|_{L^6_x}
\left\|\langle v\rangle^{-\frac12}w_l(\alpha)\partial^\alpha f\right\|\\[2mm]
&+{\varepsilon^2}\left\|\partial^{\alpha}E\right\|\left\|\langle v\rangle^{\frac12}w_l(\alpha)f\right\|_{L^\infty_x}
\left\|\langle v\rangle^{-\frac12}w_l(\alpha)\partial^\alpha f\right\|\\
\end{aligned}
\end{equation}
For the first term on the right-hand side of the above inequality, one has by using the interpolation method with respect to velocity weight and Young's inequality,
\begin{eqnarray}
&&{\varepsilon^2}\|E\|_{L^\infty_x}\left\|\langle v\rangle^{1/2}w_l(\alpha)\partial^{\alpha}f\right\|
\left\|\langle v\rangle^{1/2}w_l(\alpha)\partial^\alpha f\right\|\nonumber\\
&\lesssim&{\varepsilon^2}\|E\|_{L^\infty_x}\left\|\langle v\rangle^{-1/2}w_l(\alpha)\partial^{\alpha}f\right\|^{\frac23}\left\|\langle v\rangle w_l(\alpha)\partial^{\alpha}f\right\|^{\frac43}\nonumber\\
&\lesssim&{\varepsilon^6}\|E\|^3_{L^\infty_x}(1+t)^{2(1+\vartheta)}\left\|\langle v\rangle^{-1/2}w_l(\alpha)\partial^{\alpha}f\right\|^{2}\nonumber\\
&&+\eta(1+t)^{-1-\vartheta}\left\|\langle v\rangle w_l(\alpha)\partial^{\alpha}f\right\|^{2}.
\end{eqnarray}
By recalling the definition of the weight $w_l(\alpha)$, one has
\[\langle v\rangle^{\frac12}w_l(\alpha)\lesssim w_l(\alpha-\alpha_1)\langle v\rangle^{\frac12-3|\alpha_1|},\]
which combing with macro-micro decomposition, the last three terms on the right-hand side of \eqref{E-N-w} can be bounded by
\begin{eqnarray}
&&\|E\|_{H^N_x}\left\{\|\nabla_x{\bf P}f\|_{H^{N-2}_x}^2+\sum_{|\alpha|\leq N-1}\|w_{l}(\alpha)\partial^\alpha{\bf \{I-P\}}f\|^2_\sigma\right\}+\eta\|w_{l}(\alpha)\partial^\alpha f\|_\sigma^2\nonumber\\
&\lesssim&\mathcal{E}_N(t)\left\{\mathcal{D}_{N}(t)+\mathcal{D}_{N-1,l}(t)\right\}+\eta\|w_{l}(\alpha)\partial^\alpha f\|_\sigma^2,
\end{eqnarray}
consequently,
\begin{eqnarray}
  &&{\varepsilon^2}\left|\left(\partial^\alpha (E\cdot vf),w^2_l(\alpha)\partial^\alpha f\right)\right|\nonumber\\
 &\lesssim&{\varepsilon^6}\|E\|^3_{L^\infty_x}(1+t)^{2(1+\vartheta)}\left\|\langle v\rangle^{-1/2}w_l(\alpha)\partial^{\alpha}f\right\|^{2}+\mathcal{E}_N(t)\left\{\mathcal{D}_{N}(t)+\mathcal{D}_{N-1,l}(t)\right\}\nonumber\\
&&+\eta(1+t)^{-1-\vartheta}\left\|\langle v\rangle w_l(\alpha)\partial^{\alpha}f\right\|^{2}+\eta\|w_{l}(\alpha)\partial^\alpha f\|_\sigma^2,
\end{eqnarray}
Similarly, one also has
\begin{eqnarray}
&&{\varepsilon^2}\left|\left(\partial^\alpha (E\cdot \nabla_vf),w^2_l(\alpha)\partial^\alpha f\right)\right|\\
&\lesssim&{\varepsilon^6}\|E\|^3_{L^\infty_x}(1+t)^{2(1+\vartheta)}\left\|\langle v\rangle^{-1/2}w_l(\alpha)\partial^{\alpha}f\right\|^{2}+\mathcal{E}_N(t)\left\{\mathcal{D}_{N}(t)+\mathcal{D}_{N-1,l}(t)\right\}\nonumber\\
&&+\eta(1+t)^{-1-\vartheta}\left\|\langle v\rangle w_l(\alpha)\partial^{\alpha}f\right\|^{2}+\eta\|w_{l}(\alpha)\partial^\alpha f\|_\sigma^2,\nonumber
\end{eqnarray}
As for ${\varepsilon}\left|\left(\partial^\alpha[ v\times B\cdot\nabla_v f],w^2_l(\alpha)\partial^\alpha f\right)\right|$, one need to estimate it very carefully because of lack of $\varepsilon-$order,
\begin{eqnarray*}
&&{\varepsilon}\left|\left(\partial^\alpha[ v\times B\cdot\nabla_v f],w^2_l(\alpha)\partial^\alpha f\right)\right|\\
&\leq&{\varepsilon}\left|\left(\partial^\alpha[ v\times B\cdot\nabla_v {\bf P}f],w^2_l(\alpha)\partial^\alpha {\bf P}f\right)\right|\\
&&+{\varepsilon}\left|\left(\partial^\alpha[ v\times B\cdot\nabla_v {\bf P}f],w^2_l(\alpha)\partial^\alpha {\bf \{I-P\}}f\right)\right|\\
&&+{\varepsilon}\left|\left(\partial^\alpha[ v\times B\cdot\nabla_v {\bf \{I-P\}}f],w^2_l(\alpha)\partial^\alpha {\bf P}f\right)\right|\\
&&+{\varepsilon}\left|\left(\partial^\alpha[ v\times B\cdot\nabla_v {\bf \{I-P\}}f],w^2_l(\alpha)\partial^\alpha {\bf \{I-P\}}f\right)\right|\\
&\leq&{\varepsilon}\left|\left(\partial^\alpha[ v\times B\cdot\nabla_v {\bf P}f],w^2_l(\alpha)\partial^\alpha {\bf \{I-P\}}f\right)\right|\\
&&+{\varepsilon}\left|\left(\partial^\alpha[ v\times B\cdot\nabla_v {\bf \{I-P\}}f],w^2_l(\alpha)\partial^\alpha {\bf P}f\right)\right|\\
&&+{\varepsilon}\left|\left(\partial^\alpha[ v\times B\cdot\nabla_v {\bf \{I-P\}}f],w^2_l(\alpha)\partial^\alpha {\bf \{I-P\}}f\right)\right|\\
&\lesssim&\mathcal{E}_N(t)\mathcal{D}_N(t)+\eta\|\partial^\alpha f\|_\sigma^2+{\varepsilon}\left|\left(\partial^\alpha[ v\times B\cdot\nabla_v {\bf \{I-P\}}f],w^2_l(\alpha)\partial^\alpha {\bf \{I-P\}}f\right)\right|.
\end{eqnarray*}
where the last term can be bounded by
\begin{eqnarray*}
 &&{\varepsilon}\left|\left(\partial^\alpha[ v\times B\cdot\nabla_v {\bf \{I-P\}}f],w^2_l(\alpha)\partial^\alpha {\bf \{I-P\}}f\right)\right|\\
 &=&{\varepsilon}\left|\left([ v\times B\cdot\nabla_v \partial^\alpha{\bf \{I-P\}}f],w^2_l(\alpha)\partial^\alpha {\bf \{I-P\}}f\right)\right|\\
 &&+\sum_{1\leq|\alpha_1|\leq N}{\varepsilon}\left|\left([ v\times \partial^{\alpha_1}B\cdot\nabla_v \partial^{\alpha-\alpha_1}{\bf \{I-P\}}f],w^2_l(\alpha)\partial^\alpha {\bf \{I-P\}}f\right)\right|\\
 &=&\sum_{1\leq|\alpha_1|\leq N}{\varepsilon}\left|\left([ v\times \partial^{\alpha_1}B\cdot\nabla_v \partial^{\alpha-\alpha_1}{\bf \{I-P\}}f],w^2_l(\alpha)\partial^\alpha {\bf \{I-P\}}f\right)\right|\\
 &\lesssim&\sum_{1\leq|\alpha_1|\leq N}{\varepsilon}\int_{\mathbb{R}^3_x\times\mathbb{R}^3_v}|\partial^{\alpha_1}B|
 |w_l(\alpha-\alpha_1)\nabla_v \partial^{\alpha-\alpha_1}{\bf \{I-P\}}f\langle v\rangle^{\frac12-3|\alpha_1|}|\\
 &&\times |w_l(\alpha)\partial^\alpha {\bf \{I-P\}}f\langle v\rangle^{-\frac12}|dxdv\\
 &\lesssim&\mathcal{E}_N(t)\mathcal{D}_{N-1,l}(t)+\eta\left\|w_l(\alpha)\partial^\alpha {\bf \{I-P\}}f\langle v\rangle^{-\frac12}\right\|^2.
\end{eqnarray*}
The last term on the right-hand side of \eqref{alpha-f-1} can be bounded by
\begin{equation}
\begin{aligned}
&{\varepsilon}\left|\left(\partial^\alpha \mathscr{T}(f,f),w^2_l(\alpha)\partial^\alpha f\right)\right|\\
\lesssim&\left\|{M}^\delta f\right\|_{L^\infty_x}\left\|w_l(\alpha)\partial^\alpha f\right\|^2_\sigma
 +\left\|{M}^\delta \partial^\alpha f\right\|\left\|w_l(\alpha) f\right\|_{L^\infty_xL^2_\sigma}
 \left\|w_l(\alpha)\partial^\alpha f\right\|_\sigma\\[2mm]
 &+\sum_{1\leq|\alpha_1|\leq|\alpha|-1}\left\|{M}^{\delta}\partial^{\alpha_1}f\right\|_{L^3_x}
 \left\|w_l(\alpha)\partial^{\alpha-\alpha_1}f\right\|_{L^6_xL^2_\sigma}\left\|w_l(\alpha)\partial^\alpha f\right\|_\sigma\\[2mm]
\lesssim& \mathcal{E}_N(t)\left\{\mathcal{D}_{N-1}(t)+\mathcal{D}_{N-1,l}(t)\right\}+\mathcal{E}_N(t)\left\|w_l(\alpha)\partial^\alpha f\right\|^2_\sigma+\eta\left\|w_l(\alpha)\partial^\alpha f\right\|_\sigma^2.
\end{aligned}
\end{equation}
By collecting the related inequalities into \eqref{alpha-f-1}, one has
\begin{eqnarray}\label{alpha-f-1-1}
&&\frac{d}{dt}{\varepsilon^2}\left\|w_l(\alpha)\partial^\alpha f\right\|^2
+\frac{\vartheta q{\varepsilon^2}}{(1+t)^{1+\vartheta}}\left\|\langle v\rangle w_l(\alpha)\partial^\alpha f\right\|^2+\left\|w_l(\alpha)\partial^\alpha f\right\|^2_\sigma\nonumber\\[2mm]
&\lesssim&\left\|\partial^\alpha {\bf P}f\right\|^2+\left\|\partial^\alpha \{{\bf I-P}\}f\right\|_\sigma^2+
{\varepsilon^2}\left\|\partial^\alpha E\right\|\left\|{M}^\delta\partial^\alpha f\right\|\nonumber\\
&&+\mathcal{E}_N(t)\left\{\mathcal{D}_{N-1}(t)+\mathcal{D}_{N-1,l}(t)\right\},
\end{eqnarray}
where we use the assumption.
Thus we complete the proof of Proposition \ref{lemma3.8}.
\end{proof}
In addition to the above highest-order energy estimates with the wight $w(\alpha)$ for $|\alpha|=N$, to avoid the macroscopic part with the singularity factor $\frac 1{\epsilon^2}$, for the other cases with weight, we applying the micro projection equality. To this end, by applying the microscopic projection $\{{\bf I-P}\}$ to the first equation of (\ref{VML-drop-eps}), we can get that
\begin{eqnarray}\label{I-P}
&&\partial_t\{{\bf I-P}\}f+\frac1\varepsilon v\cdot \nabla_x\{{\bf I-P}\}f-\frac1 \varepsilon E\cdot v{M}^{1/2}q_1+\frac1 {\varepsilon^2} \mathscr{L}f\nonumber\\
&=&\{{\bf I-P}\}I_{tri}+\frac1 \varepsilon{\bf P}(v\cdot\nabla_x f)-\frac1 \varepsilon v\cdot\nabla_x{\bf P}f.
\end{eqnarray}
where
\[I_{tri}=\frac12 v\cdot E f+q_1\left(E+\frac1\varepsilon v\times B\right)\cdot\nabla_{ v}f+\frac1\varepsilon\mathscr{T}(f,f)\]

\begin{lemma}\label{lemma-micro}
There exist an energy functional $\mathcal{E}_{N-1,l}(t)$ and the corresponding energy dissipation functional $\mathcal{D}_{N-1,l}(t)$ which satisfy \eqref{E-N-1-w} and \eqref{D-N-1-w} respectively such that
  \begin{eqnarray}\label{lemma3.8-2}
&&\frac{d}{dt}\mathcal{E}_{N-1,l}(t)+\mathcal{D}_{N-1,l}(t)
\lesssim\mathcal{D}_{N}(t)+\mathcal{E}_{N}(t)\mathcal{D}_{N-1,l}(t)
\end{eqnarray}
\end{lemma}
\begin{proof}
For the weighted energy estimate on $\{{\bf I-P}\}\partial^\alpha f$ with $|\alpha|\leq N-1$, by applying $\partial^\alpha$ into \eqref{I-P}, multiplying $w^2_{l}(\alpha)\partial^\alpha\{{\bf I-P}\}f$ and integrating the result identity over $\mathbb{R}^3_x\times\mathbb{R}^3_v$, then we have
\begin{eqnarray}\label{I-P-w}
&&\frac12\frac{d}{dt}\|w_l(\alpha)\partial^\alpha\{{\bf I-P}\}f\|^2+\frac{q\vartheta}{(1+t)^{1+\vartheta}}\left\|\langle v\rangle w_{l}(\alpha)\partial^\alpha\{{\bf I-P}\}f\right\|^2\nonumber\\
&&+\frac1 {\varepsilon^2} \left(\partial^\alpha\mathscr{L}f,w^2_l(\alpha)\partial^\alpha\{{\bf I-P}\}f\right)\nonumber\\
&=&-\frac1\varepsilon\left(\partial^\alpha v\cdot \nabla_x\{{\bf I-P}\}f,w^2_l(\alpha)\partial^\alpha\{{\bf I-P}\}f\right)\nonumber\\
&&+\frac1 \varepsilon \left(\partial^\alpha E\cdot v{M}^{1/2}q_1,w^2_l(\alpha)\partial^\alpha\{{\bf I-P}\}f\right)\nonumber\\
&&+\frac1 \varepsilon\left(\partial^\alpha\left\{{\bf P}(v\cdot\nabla_x f)-\frac1 \varepsilon v\cdot\nabla_x{\bf P}f\right\},w^2_l(\alpha)\partial^\alpha\{{\bf I-P}\}f\right)\nonumber\\
&&+\left(\partial^\alpha\{{\bf I-P}\}I_{tri},w^2_l(\alpha)\partial^\alpha\{{\bf I-P}\}f\right).
\end{eqnarray}
The coercivity estimates on the linear operator $\mathscr{L}$ yields that
 \begin{eqnarray}
   &&\frac1 {\varepsilon^2} \left(\partial^\alpha\mathscr{L}f,w^2_l(\alpha)\partial^\alpha\{{\bf I-P}\}f\right)\nonumber\\
   &\gtrsim&\frac1 {\varepsilon^2} \|w_l(\alpha)\partial^\alpha\{{\bf I-P}\}f\|_\sigma^2-\frac1 {\varepsilon^2}\|\partial^\alpha\{{\bf I-P}\}f\|_\sigma^2.
 \end{eqnarray}
 Integrating by part over $\mathbb{R}^3_x$, one has
 \[\frac1\varepsilon\left(\partial^\alpha v\cdot \nabla_x\{{\bf I-P}\}f,w^2_l(\alpha)\partial^\alpha\{{\bf I-P}\}f\right)=0.\]
 By using the Cauchy inequality, one has
\begin{eqnarray}
  &&\frac1 \varepsilon \left(\partial^\alpha E\cdot v{M}^{1/2}q_1,w^2_l(\alpha)\partial^\alpha\{{\bf I-P}\}f\right)\nonumber\\
  &\lesssim&\|\partial^\alpha E\|^2+\eta\frac1 {\varepsilon^2}\|\partial^\alpha\{{\bf I-P}\}f\|_\sigma^2.
\end{eqnarray}
and
\begin{eqnarray}
  &&\frac1 \varepsilon\left(\partial^\alpha\left\{{\bf P}(v\cdot\nabla_x f)-\frac1 \varepsilon v\cdot\nabla_x{\bf P}f\right\},w^2_l(\alpha)\partial^\alpha\{{\bf I-P}\}f\right)\nonumber\\
  &\lesssim&\|\nabla^{|\alpha|+1}f\|_\sigma^2+\eta\frac1 {\varepsilon^2}\|\partial^\alpha\{{\bf I-P}\}f\|_\sigma^2.
\end{eqnarray}
As for the last term on the right-hand side of \eqref{I-P-w}, by using the similar argument as the estimate on \eqref{E-N-w}, one has
\begin{eqnarray}
  &&\left(\partial^\alpha\{{\bf I-P}\}I_{tri},w^2_l(\alpha)\partial^\alpha\{{\bf I-P}\}f\right)\nonumber\\
&\lesssim&{\varepsilon^6}\|E\|^3_{L^\infty_x}(1+t)^{2(1+\vartheta)}\left\|\langle v\rangle^{-1/2}w_l(\alpha)\partial^{\alpha}\{{\bf I-P}\}f\right\|^{2}+\mathcal{E}_N(t)\left\{\mathcal{D}_{N}(t)+\mathcal{D}_{N-1,l}(t)\right\}\nonumber\\
&&+\eta(1+t)^{-1-\vartheta}\left\|\langle v\rangle w_l(\alpha)\partial^{\alpha}\{{\bf I-P}\}f\right\|^{2}+\eta\|w_{l}(\alpha)\partial^\alpha \{{\bf I-P}\}f\|_\sigma^2
\end{eqnarray}
 Now by collecting the above related estimates into \eqref{I-P-w}, we arrive at
 \begin{equation}\label{3.68}
\begin{aligned}
&\frac{d}{dt}\left\|w_{l}(\alpha)\partial^\alpha\{{\bf I-P}\}f\right\|^2+\frac1{\varepsilon^2}\left\|w_{l}(\alpha)\partial^\alpha\{{\bf I-P}\}f\right\|^2_\sigma\\[2mm]
&+\frac{q\vartheta}{(1+t)^{1+\vartheta}}\left\|\langle v\rangle w_{l}(\alpha)\partial^\alpha\{{\bf I-P}\}f\right\|^2\\[2mm]
\lesssim&\frac1{\varepsilon^2}\left\|\partial^\alpha\{{\bf I-P}\}f\right\|^2_\sigma+\|\partial^\alpha E\|^2+\left\|\nabla_x^{|\alpha|+1} f\right\|^2\\
&+\mathcal{E}_N(t)\left\{\mathcal{D}_{N-1}(t)+\mathcal{D}_{N-1,l}(t)\right\}.
\end{aligned}
\end{equation}

A proper linear combination of (\ref{lemma3.7-1}) and (\ref{3.68}) implies \eqref{lemma3.8-1}.
\end{proof}
To control
\[\sum_{|\alpha|=N}{\varepsilon^2}\left\|\partial^\alpha E\right\|\left\|{M}^\delta\partial^\alpha f\right\|\]
on the right-hand side of \eqref{lemma3.8-1}, by multiplying $(1+t)^{-\epsilon_0}$ into \eqref{lemma3.8-1}, one has
\begin{proposition}\label{prop1}
Under {\bf Assumption 1},
take $l\geq N+\frac12$, we can deduce that
\begin{eqnarray}\label{prop1-1}
&&\frac{d}{dt}\left\{(1+t)^{-\frac{1+\epsilon_0}2}\sum_{|\alpha|=N}{\varepsilon^2}\left\|w_l(\alpha)\partial^\alpha f\right\|^2+\mathcal{E}_{N-1,l}(t)+\mathcal{E}_{N}(t)\right\}\nonumber\\
&&+\mathcal{D}_{N}(t)+\mathcal{D}_{N-1,l}(t)
\lesssim0
\end{eqnarray}
{ holds for all $0\leq t\leq T$.}
\end{proposition}
\begin{proof}
  By multiplying $(1+t)^{-\epsilon_0}$ into \eqref{prof-1}, one has
  \begin{eqnarray}\label{prof-1-1}
&&\frac{d}{dt}\left\{(1+t)^{-\epsilon_0}\mathcal{E}_{N}(t)\right\}+\epsilon_0(1+t)^{-1-\epsilon_0}\mathcal{E}_{N}(t)+(1+t)^{-\epsilon_0}\mathcal{D}_{N}(t)\nonumber\\
&\lesssim &M_1(1+t)^{-(1+2\epsilon_0)}\left\|\langle v\rangle^{\frac 32}\nabla^N_xf\right\|^2+\mathcal{E}_N(t)\mathcal{D}_{N-1,l}(t)
\end{eqnarray}
By multiplying $(1+t)^{-\frac{1+\epsilon_0}2}$ into \eqref{lemma3.8-1}, one has
\begin{eqnarray}\label{lemma3.8-1-1}
&&\frac{d}{dt}\left\{(1+t)^{-\frac{1+\epsilon_0}2}\sum_{|\alpha|=N}{\varepsilon^2}\left\|w_l(\alpha)\partial^\alpha f\right\|^2\right\}\nonumber\\
&&+(1+t)^{-\frac{1+\epsilon_0}2}\frac{\vartheta q{\varepsilon^2}}{(1+t)^{1+\vartheta}}\sum_{|\alpha|=N}\left\|\langle v\rangle w_l(\alpha)\partial^\alpha f\right\|^2\nonumber\\
&&+(1+t)^{-\frac{1+\epsilon_0}2}\sum_{|\alpha|=N}\left\|w_l(\alpha)\partial^\alpha f\right\|^2_\sigma\nonumber\\
&\lesssim&(1+t)^{-\frac{1+\epsilon_0}2}\mathcal{D}_{N}(t)+(1+t)^{-\frac{1+\epsilon_0}2}\mathcal{E}_{N}(t)\mathcal{D}_{N-1,l}(t)\nonumber\\
&&+(1+t)^{-\frac{1+\epsilon_0}2}
\sum_{|\alpha|=N}{\varepsilon^2}\left\|\partial^\alpha E\right\|\left\|{M}^\delta\partial^\alpha f\right\|\nonumber\\
&\lesssim&\mathcal{D}_{N}(t)+\mathcal{E}_{N}(t)\mathcal{D}_{N-1,l}(t)+\eta(1+t)^{-1-\epsilon_0}\|\nabla^N_xE\|^2.
\end{eqnarray}
By applying a proper linear combination of \eqref{prof-1-1} and \eqref{lemma3.8-1-1}, one has
\begin{eqnarray}\label{prop1-1-1}
&&\frac{d}{dt}\left\{(1+t)^{-\epsilon_0}\mathcal{E}_{N}(t)+(1+t)^{-\frac{1+\epsilon_0}2}\sum_{|\alpha|=N}{\varepsilon^2}\left\|w_l(\alpha)\partial^\alpha f\right\|^2\right\}\nonumber\\
&&+\epsilon_0(1+t)^{-1-\epsilon_0}\mathcal{E}_{N}(t)+(1+t)^{-\epsilon_0}\mathcal{D}_{N}(t)\nonumber\\
&&+(1+t)^{-\frac{1+\epsilon_0}2}\frac{\vartheta q{\varepsilon^2}}{(1+t)^{1+\vartheta}}\sum_{|\alpha|=N}\left\|\langle v\rangle w_l(\alpha)\partial^\alpha f\right\|^2\nonumber\\
&&+(1+t)^{-\frac{1+\epsilon_0}2}\sum_{|\alpha|=N}\left\|w_l(\alpha)\partial^\alpha f\right\|^2_\sigma\nonumber\\
&\lesssim&\mathcal{D}_{N}(t)+\mathcal{E}_{N}(t)\mathcal{D}_{N-1,l}(t).\nonumber
\end{eqnarray}
Thus \eqref{prop1-1} follows from \eqref{prof-1},  \eqref{lemma3.8-2}, and \eqref{prop1-1-1}, which complete the proof of this proposition.
\end{proof}
\section{The temporal time decay estimate on $\mathcal{E}^k_{N_0}(t)$}
To ensure {\bf Assumption} 1 in \eqref{Assump-1}, this section is devoted into the temporal time decay rates for $[f,E,B]$ to
the Cauchy problem \eqref{VML-drop-eps}.

In fact, from Proposition \ref{prop1}, we only need to guarantee the bound on $(1+t)^{1+\epsilon_0}\left\|E\right\|_{L^\infty_x}^2$. For this purpose, since
\[H^2_x\hookrightarrow L^\infty_x\]
for $x\in \mathbb{R}^3_x$, we need to deduce the temporal time decay estimate on $\mathcal{E}^k_{N_0}(t)$ satisfying \eqref{E-k} for $N_0\geq 3$.

To state for brevity, we define
\begin{equation}\label{def-E-N-bar}
\mathcal{\bar{E}}_{N_0,\ell}(t)=\mathcal{{E}}_{N_0}(t)+\mathcal{{E}}_{N_0,\ell}(t)+\left\|\Lambda^{-\frac12}[f,E,B]\right\|^2
\end{equation}

\subsection{Temporal time decay estimate on $\mathcal{E}^k_{N_0}(t)$}
\begin{lemma}\label{lemma4.3} Under {\bf Assumption} 1, let $N_0\geq 3$, $N=2N_0$, one has the following estimates:
\begin{itemize}
\item[(i).] For $k=0,1,\cdots,N_0-2$, it holds that
\begin{equation}\label{Lemma4.3-1}
\begin{aligned}
&\frac{d}{dt}\left(\left\|\nabla^kf\right\|^2+\left\|\nabla^k[E,B]\right\|^2\right)+\frac1{\varepsilon^2}\left\|\nabla^k\{{\bf I-P}\}f\right\|^2_{\sigma}\\[2mm]
\lesssim&\mathcal{\bar{E}}_{N_0,N}(t)\left(\left\|\nabla^{k+1}[E,B]\right\|^2+\left\|\nabla^{k+1}{\bf P}f\right\|^2\right)+\eta\left\|\nabla^{k+1}f\right\|_\sigma^2.
\end{aligned}
\end{equation}
\item [(ii).] If $k=N_0-1$, it holds that
\begin{equation}\label{Lemma4.3-2}
\begin{aligned}
&\frac{d}{dt}\left(\left\|\nabla^{N_0-1}f\right\|^2+\left\|\nabla^{N_0-1}[E,B]\right\|^2\right)+\frac1{\varepsilon^2}\left\|\nabla^{N_0-1}\{{\bf I-P}\}f\right\|_\sigma^2\\[2mm]
\lesssim&\mathcal{\bar{E}}_{N_0,N}(t)\left(\left\|\nabla^{N_0-1}[E,B]\right\|^2+\left\|\nabla^{N_0-1}{\bf P}f\right\|^2_\sigma\right)+\eta\left\|\nabla^{N_0-1}f\right\|_\sigma^2.
\end{aligned}
\end{equation}
\item[(iii).] For $k=N_0$, it follows that
\begin{equation}\label{Lemma4.3-3}
\begin{aligned}
&\frac{d}{dt}\left(\left\|\nabla^{N_0}f\right\|^2+\left\|\nabla^{N_0}[E,B]\right\|^2\right)
+\frac1{\varepsilon^2}\left\|\nabla^{N_0}\{{\bf I-P}\}f\right\|^2_{\sigma}\\[2mm]
\lesssim&\max\left\{\mathcal{E}_{N}(t),\mathcal{\bar{E}}_{N_0,N}(t)\right\}\left(\left\|\nabla^{N_0-1}[E,B]\right\|^2
+\left\|\nabla^{N_0-1}f\right\|_\sigma^2+\left\|\nabla^{N_0}f\right\|_\sigma^2\right)\\
&+\eta\left\|\nabla^{N_0}f\right\|_\sigma^2.
\end{aligned}
\end{equation}
\end{itemize}
\end{lemma}
\begin{proof} For the case $k=0$, we have by multiplying $(\ref{VML-drop-eps})$ by $f$ and integrating the resulting identity with respect to $x$ and $v$ over $\mathbb{R}_x^3\times\mathbb{R}_v^3$ that
\begin{equation}\label{k-0}
\begin{aligned}
\frac{d}{dt}\left(\|f\|^2+\|[E,B]\|^2\right)+\frac1{\varepsilon^2}\|\{{\bf I-P}\}f\|^2_{\sigma}\lesssim|(v\cdot E f,f)|+\frac1{\varepsilon}|({\Gamma}(f,f),\{{\bf I-P}\}f)|
\end{aligned}
\end{equation}
For the first term on the right hand of (\ref{k-0}), applying the macro-micro decomposition yields that
\begin{eqnarray}
&&|(v\cdot E f,f)|\nonumber\\
&\lesssim&|(v\cdot E {\bf P}f,{\bf P}f)|+|(v\cdot E {\bf P}f,\{{\bf I-P}\}f)|\nonumber\\
&&+|(v\cdot E \{{\bf I-P}\}f,{\bf P}f)|+|(v\cdot E \{{\bf I-P}\}f,\{{\bf I-P}\}f)|\nonumber
\end{eqnarray}
where the first term on the above inequality can be bounded by
\begin{eqnarray}
 &&|(v\cdot E {\bf P}f,{\bf P}f)|\nonumber\\
 &\lesssim&\|E\|_{L^2_x}\left\|{\bf P}f\right\|_{L^3_xL^2_v}\left\|{\bf P}f\right\|_{L^6_xL^2_v}\nonumber\\
  &\lesssim&\left\|\Lambda^{-\frac12}E\right\|^{\frac 23}\left\|\nabla_xE\right\|^{\frac 13}
\left\|\Lambda^{-\frac12}{\bf P}f\right\|^{\frac 13}\left\|\nabla_x{\bf P}f\right\|^{\frac 23} \left\|\nabla_x{\bf P}f\right\|\nonumber\\
&\lesssim&\left\|\Lambda^{-\frac12}[{\bf P}f,E]\right\|^{2}\left\|\nabla_x[{\bf P}f,E]\right\|^{2}+
 \eta\left\|\nabla_x{\bf P}f\right\|,
\end{eqnarray}
the other three terms can be dominated by
\begin{eqnarray}
  &\lesssim&
\|E\|_{L^6_x}\left\|\langle v\rangle^{\frac{3}2}f\right\|_{L^3_xL^2_v}
\left\|\langle v\rangle^{-\frac12}\{{\bf I-P}\}f\right\|_{L^2_xL^2_v}\nonumber\\
&\lesssim&\|E\|_{L^2_vL^6_x}\left\|\langle v\rangle^{\frac{3}2}f\right\|_{L^2_vL^3_x}
\left\|\{{\bf I-P}\}f\right\|_\sigma\nonumber\\
&\lesssim&\|\nabla_xE\|\left\|\langle v\rangle^{\frac{3}2}f\right\|_{H^1_x}
\|\{{\bf I-P}\}f\|_\sigma\nonumber\\
&\lesssim&\mathcal{\bar{E}}_{2,2}(t)\left\|\nabla_xE\right\|^2
+\eta\left\|\{{\bf I-P}\}f\right\|^2_{\sigma}.
\end{eqnarray}
The last term on the right-hand side of \eqref{k-0} can be estimated as follows
\begin{equation*}
\begin{aligned}
\frac1{\varepsilon}|(\mathscr{T}(f,f),\{{\bf I-P}\}f)|\lesssim&\frac{\eta}{{\varepsilon^2}}\|\{{\bf I-P}\}f\|^2_{\sigma}
+C_\eta\left\|\left|{M}^\delta f\right|_{L^2_v}|f|_{L^2_\sigma}\right\|^2\\[2mm]
\lesssim&\frac{\eta}{{\varepsilon^2}}\|\{{\bf I-P}\}f\|^2_{\sigma}+C_\eta\left\|{M}^\delta f\right\|^2_{L^3_x}\|f\|^2_{L^6_\sigma}\\[2mm]
\lesssim&\frac{\eta}{{\varepsilon^2}}\|\{{\bf I-P}\}f\|^2_{\sigma}+C_\eta\left\|{M}^\delta f\right\|^2_{H^1_x}\|\nabla_xf\|^2_\sigma.
\end{aligned}
\end{equation*}
Collecting the above estimates gives
\begin{equation}\label{0}
\begin{aligned}
&\frac{d}{dt}\left(\|f\|^2+\|[E,B]\|^2\right)+\frac{1}{{\varepsilon^2}}\|\{{\bf I-P}\}f\|^2_{\sigma}\\
\lesssim&\mathcal{\bar{E}}_{2,2}(t)\left(\left\|\nabla_xE\right\|^2
+\left\|\nabla_x{\bf P}f\right\|^2\right)+\eta\|\nabla_x{\bf P}f\|^2,
\end{aligned}
\end{equation}
which gives (\ref{Lemma4.3-1}) with $k=0$.

For $k=1,2,\cdots,N_0-2$, we have by applying $\nabla^k$ to $(\ref{VML-drop-eps})$,
multiplying the resulting identity by $\nabla^kf$, and then integrating the final result with respect to $x$ and $v$ over $\mathbb{R}_x^3\times\mathbb{R}_v^3$ that
\begin{equation}\label{k}
\begin{aligned}
\frac{d}{dt}&\left(\left\|\nabla^kf\right\|^2+\left\|\nabla^k[E,B]\right\|^2\right)
+\frac{1}{{\varepsilon^2}}\left\|\nabla^k\{{\bf I-P}\}f\right\|^2_{\sigma}\\
\lesssim&{\left|\left(\nabla^k(v\cdot Ef),\nabla^kf\right)\right|}
+{\left|\left(\nabla^k(E\cdot\nabla_vf),\nabla^kf\right)\right|}\\
&+\frac1\varepsilon{\left|\left(\nabla^k\left((v\times B)\cdot\nabla_vf\right),\nabla^kf\right)\right|}
+\frac1\varepsilon{\left|\left(\nabla^k{\bf\Gamma}(f,f),\nabla^kf\right)\right|}.
\end{aligned}
\end{equation}
By applying the macro-micro decomposition \eqref{macro-micro}, one can deduce that the first term on the right-hand side of the above inequality can be bounded by
\begin{eqnarray}\label{k-E}
&\lesssim&\left|\left(\nabla^k(v\cdot E{\bf P}f),\nabla^k{\bf P}f\right)\right|
+\left|\left(\nabla^k(v\cdot E{\bf P}f),\nabla^k{\{\bf I-P\}}f\right)\right|\\
&&+\left|\left(\nabla^k(v\cdot E{\{\bf I-P\}}f),\nabla^k {\bf P}f\right)\right|+\left|\left(\nabla^k(v\cdot E\{{\bf I-P}\}f),\nabla^k{\{\bf I-P\}}f\right)\right|,\nonumber\\
\end{eqnarray}
where the first three terms on the right-hand side can be controlled by
\begin{eqnarray}
&\lesssim&\sum_{j\leq k-1}\left\|\left(\nabla^jE \nabla^{k-1-j}\left({M}^{\delta}f\right)\right)\right\|\left\|\nabla^{k+1}\left({M}^{\delta}f\right)\right\|\nonumber\\
&\lesssim&\sum_{j\leq k-1}\left\|\nabla^jE\right\|_{L^6_x}
\left\|\nabla^{k-1-j}\left({M}^{\delta}f\right)\right\|_{L^3_xL^2_v}\left\|\nabla^{k+1}\left({M}^{\delta}f\right)\right\|\nonumber\\
&\lesssim&\sum_{j\leq k-1}\left\|\Lambda^{-\frac 12}E\right\|^{\frac{2k-2j}{2k+3}} \left\|\nabla^{k+1}E\right\|^{\frac{2j+3}{2k+3}}
\left\|\Lambda^{-\frac 12}\left({M}^{\delta}f\right)\right\|^{\frac{2j+3}{2k+3}} \left\|\nabla^{k+1}\left({M}^{\delta}f\right)\right\|^{\frac{2k-2j}{2k+3}}
\left\|\nabla^{k+1}\left({M}^{\delta}f\right)\right\|\nonumber\\
&\lesssim&\left\|\Lambda^{-\frac12}[f,E]\right\|^{2}\left(\left\|\nabla^{k+1}E\right\|^2+\left\|\nabla^{k+1}f\right\|_\sigma^2\right)
+\eta\left\|\nabla^{k+1}f\right\|_\sigma^2.\nonumber
\end{eqnarray}
As for the last term, it holds from the Cauchy inequality and the Holder inequality that
\begin{eqnarray*}
&&\sum_{j\leq k}\left\|\nabla^jE  \nabla^{k-j}\{{\bf I-P}\}f\langle v\rangle^{\frac 32}\right\|^2\nonumber\\
&\lesssim&\sum_{j\leq k}\left\|\nabla^jE  \nabla^{k-j}\{{\bf I-P}\}f\langle v\rangle^{\frac 32}\right\|
\left\|\nabla^k\{{\bf I-P}\}f\langle v\rangle^{-\frac 1{2}}\right\|\\[2mm]
&\lesssim&\sum_{j\leq k}\left\|\nabla^jE  \nabla^{k-j}\{{\bf I-P}\}f\langle v\rangle^{\frac 32}\right\|^2+\eta\left\|\nabla^k\{{\bf I-P}\}f\right\|^2_\sigma.
\end{eqnarray*}
For the first term on the above equality, one has by employing interpolation method with respect to both spatial derivative and velocity weight $v$,
\begin{eqnarray}
  &&\sum_{j\leq k}\left\|\nabla^jE  \nabla^{k-j}\{{\bf I-P}\}f\langle v\rangle^{\frac 32}\right\|^2\nonumber\\
 &\lesssim&\left\|\nabla^kE\right\|_{L^6_x}^2\left\| \{{\bf I-P}\}f\langle v\rangle^{\frac 32}\right\|_{L^3_x}^2\nonumber\\
 &&+\sum_{0\leq j\leq k-1}\left\|\nabla^jE\right\|_{L^\infty_x}\left\|  \nabla^{k-j}\{{\bf I-P}\}f\langle v\rangle^{\frac 32}\right\|^2\nonumber\\
 &\lesssim&\left\| \{{\bf I-P}\}f\langle v\rangle^{\frac 32}\right\|_{H^1_x}^2\left\|\nabla^{k+1}E\right\|^2\nonumber\\
  &&+\sum_{0\leq j\leq k-1}\left(\left\|\Lambda^{-\frac12}E\right\|^{1-\theta_1}\left\|\nabla^{k+1}E\right\|^{\theta_1}
 \left\|\nabla^{k-j}\{{\bf I-P}\}f\langle v\rangle^{-\frac {1}2}\right\|^{1-\theta_1}
 \left\|\nabla^{k-j}\{{\bf I-P}\}f\langle v\rangle^{\hat{l}_1}\right\|^{\theta_1}\right)^2\nonumber\\
 &\lesssim&\left(\left\|\Lambda^{-\frac12}E\right\|^2+ \left\|\{{\bf I-P}\}f\langle v\rangle^{\hat{l}_1}\right\|_{H^k_x}^2\right)\left(\left\|\nabla^{k+1}E\right\|^2
+ \left\|\nabla^{k-j}\{{\bf I-P}\}f\langle v\rangle^{-\frac {1}2}\right\|^2\right)\nonumber\\
 &\lesssim&\mathcal{\bar{E}}_{N_0,N_0+\frac{\hat{l}_1}3}(t)\left(\left\|\nabla^{k+1}E\right\|^2+
 \left\|\{{\bf I-P}\}f\right\|_{H^k_xL^2_\sigma}^2\right).
\end{eqnarray}

Here $\theta_1=\frac{2j+4}{2k+3}$ and $\hat{l}_1$ satisfies that
\begin{eqnarray}
  \frac 32=-\frac12(1-\theta_1)+\hat{l}_1\theta_1
\end{eqnarray}
which yields that
\[\hat{l}_1=\frac{2k+3}{j+2}-\frac12,\]
without generality, we take $\hat{l}_1=N_0.$

Consequently, the first term on the right-hand side of \eqref{k} can be bounded by
\begin{eqnarray}
&&\left|\left(\nabla^k(v\cdot Ef),\nabla^kf\right)\right|\nonumber\\
&\lesssim&\mathcal{\bar{E}}_{N_0,N_0+\frac{\hat{l}_1}3}(t)\left(\left\|\nabla^{k+1}E\right\|^2+\left\|\nabla^{k+1}f\right\|_\sigma^2
+\left\|\{{\bf I-P}\}f\right\|_{H^k_xL^2_\sigma}^2\right)\nonumber\\
&&+\eta\left(\left\|\nabla^{k+1}f\right\|^2_\sigma+\left\|\nabla^k\{{\bf I-P}\}f\right\|^2_\sigma\right).\label{E-k-decay}
\end{eqnarray}
Similarly, integrating by part with respect to velocity variable $v$, one has
\begin{eqnarray*}
&&\left|\left(\nabla^k(E\cdot\nabla_vf),\nabla^kf\right)\right|=\left|\left(\nabla^k(E\cdot f),\nabla^k \nabla_vf\right)\right|\\
&\lesssim&\mathcal{\bar{E}}_{N_0,N_0+\frac{\hat{l}_1}3}(t)\left(\left\|\nabla^{k+1}E\right\|^2+\left\|\nabla^{k+1}f\right\|_\sigma^2
+\left\|\{{\bf I-P}\}f\right\|_{H^k_xL^2_\sigma}^2\right)\\
&&+\eta\left(\left\|\nabla^{k+1}f\right\|^2_\sigma+\left\|\nabla^k\{{\bf I-P}\}f\right\|^2_\sigma\right).
\end{eqnarray*}
By applying the macro-micro decomposition \eqref{macro-micro} and noting that
\[\left(\nabla^k\left(q_1(v\times B)\cdot\nabla_v{\bf P}f\right),\nabla^k{\bf P}f\right)=0,\]
one can deduce that by a little modification as the estimates on \eqref{E-k-decay}
\begin{eqnarray*}
  &&\frac1\varepsilon{\left|\left(\nabla^k\left({q_0}(v\times B)\cdot\nabla_vf\right),\nabla^kf\right)\right|}\\
  &\lesssim&\frac1\varepsilon{\left|\left(\nabla^k\left({q_0}(v\times B)\cdot\nabla_v{\bf P}f\right),\nabla^k{\bf P}f\right)\right|}+\frac1\varepsilon{\left|\left(\nabla^k\left({q_0}(v\times B)\cdot\nabla_v{\bf P}f\right),\nabla^k{\bf \{I-P\}}f\right)\right|}\\
  &&+\frac1\varepsilon{\left|\left(\nabla^k\left({q_0}(v\times B)\cdot\nabla_v{\bf\{I- P\}}f\right),\nabla^k{\bf P}f\right)\right|}+\frac1\varepsilon{\left|\left(\nabla^k\left({q_0}(v\times B)\cdot\nabla_v{\bf \{I-P\}}f\right),\nabla^k{\bf \{I-P\}}f\right)\right|}\\
  &\lesssim&\frac1\varepsilon{\left|\left(\nabla^k\left({q_0}(v\times B)\cdot\nabla_v{\bf P}f\right),\nabla^k{\bf \{I-P\}}f\right)\right|}+\frac1\varepsilon{\left|\left(\nabla^k\left({q_0}(v\times B)\cdot\nabla_v{\bf\{I- P\}}f\right),\nabla^k{\bf P}f\right)\right|}\\
  &&+\frac1\varepsilon{\left|\left(\nabla^k\left({q_0}(v\times B)\cdot\nabla_v{\bf \{I-P\}}f\right),\nabla^k{\bf \{I-P\}}f\right)\right|}\\
  &=&\frac1\varepsilon{\left|\left(\nabla^k\left({q_0}(v\times B)\cdot\nabla_v{\bf P}f\right),\nabla^k{\bf \{I-P\}}f\right)\right|}+\frac1\varepsilon{\left|\left(\nabla^k\left({q_0}(v\times B)\cdot\nabla_v{\bf\{I- P\}}f\right),\nabla^k{\bf P}f\right)\right|}\\
  &&+\frac1\varepsilon{\left|\left(\nabla^k\left({q_0}(v\times B)\cdot{\bf \{I-P\}}f\right),\nabla^k\nabla_v{\bf \{I-P\}}f\right)\right|}\\
  &\lesssim&\mathcal{\bar{E}}_{N_0,N_0+\frac{\hat{l}_1}3}(t)\left(\left\|\nabla^{k+1}B\right\|^2+\left\|\nabla^{k+1}f\right\|_\sigma^2
+\left\|\{{\bf I-P}\}f\right\|_{H^k_xL^2_\sigma}^2\right)\\
&&+\eta\left(\left\|\nabla^{k+1}{\bf P}f\right\|^2+\frac1{\varepsilon^2}\left\|\nabla^k\{{\bf I-P}\}f\right\|^2_\sigma\right).
\end{eqnarray*}
For the last term on the right-hand side of (\ref{k}), it follows from Lemma \ref{lemma2.2} that
\begin{equation}\label{Gamma-k}
\begin{aligned}
&\frac1\varepsilon{\left|\left(\nabla^k\mathscr{T}(f,f),\nabla^kf\right)\right|}
=\frac1\varepsilon{\left|\left(\nabla^k\mathscr{T}(f,f),\nabla^k\{{\bf I-P}\}f\right)\right|}\\
\lesssim& \frac1\varepsilon\sum_{j\leq k}\left\|\left|\nabla^j{M}^\delta f\right|_{L^2_v}\left|\nabla^{k-j}f\right|_{L^2_\sigma}\right\|\left\|\nabla^k\{{\bf I-P}\}f\right\|_\sigma\\[2mm]
\lesssim&\frac1\varepsilon\sum_{j\leq k}\left\|{M}^\delta\nabla^j f\right\|_{L^3_xL^2_v}\left\|\nabla^{k-j}f\right\|_{L^6_xL^2_\sigma} \left\|\nabla^k\{{\bf I-P}\}f\right\|_\sigma\\[2mm]
\lesssim&\frac1\varepsilon\sum_{j\leq k}\left\|{M}^\delta\nabla^j f\right\|_{L^2_vL^3_x}\left\|\nabla^{k-j}f\right\|_{L^2_\sigma L^6_x}\left\|\nabla^k\{{\bf I-P}\}f\right\|_\sigma\\[2mm]
\lesssim&\frac1\varepsilon\sum_{j\leq k}\left\|{M}^\delta \nabla^mf\right\|^{\frac{j+1}{k+1}}\left\|{M}^\delta\nabla^{k+1} f\right\|^{1-\frac{j+1}{k+1}}\| f\|_\sigma^{1-\frac{j+1}{k+1}}\left\| \nabla^{k+1}f\right\|_\sigma^{\frac{j+1}{k+1}}\left\|\nabla^k\{{\bf I-P}\}f\right\|_\sigma\\[2mm]
\lesssim&\max\left\{\mathcal{E}_{N_0-2}(t), \mathcal{E}_{1,1}(t)\right\} \left\|\nabla^{k+1} f\right\|_\sigma^2 +\frac\eta{{\varepsilon^2}}\left\|\nabla^k\{{\bf I-P}\}f\right\|_\sigma^2.
\end{aligned}
\end{equation}
Here $m=\frac{k+1}{2(j+1)}\leq\frac{k+1}{2}\leq \frac N2.$

By collecting the above related estimates, one can deduce (\ref{Lemma4.3-1}) immediately.
\eqref{Lemma4.3-2} and \eqref{Lemma4.3-2} can be obtained by a similar way as (\ref{Lemma4.3-1}).
\end{proof}

\begin{lemma}\label{Lemma4.4}

\begin{itemize} For the macro dissipation estimates on $f(t,x,v)$, we have the following results:
\item[(i).] For $k=0,1,2\cdots,N_0-2$, there exist interactive energy functionals $G^k_f(t)$ satisfying
\[
G^k_f(t)\lesssim \left\|\nabla^k[f,E,B]\right\|^2+\left\|\nabla^{k+1}[f,E,B]\right\|^2+\left\|\nabla^{k+2}E\right\|^2
\]
such that
\begin{eqnarray}\label{Lemma3.4-1}
&&\frac{d}{dt}G^k_f(t)+\left\|\nabla^k[E,\rho_+-\rho_-]\right\|_{H^1}^2+\left\|\nabla^{k+1}[{\bf P}f,B])\right\|^2\nonumber\\
&\lesssim&\left\{\left\|\Lambda^{-\frac12}[f,E,B]\right\|^2+\mathcal{E}_{N}(t)\right\}\left(\left\|\nabla^{k+1}[E,B]\right\|^2+\left\|\nabla^{k+1}f\right\|^2_\sigma\right)
+\frac1{{\varepsilon^2}}\left\|\nabla^k\{{\bf I-P}\}f\right\|^2_\sigma\nonumber\\
&&
+\frac1{{\varepsilon^2}}\left\|\nabla^{k+1}\{{\bf I-P}\}f\right\|^2_\sigma
+\frac1{{\varepsilon^2}}\left\|\nabla^{k+2}\{{\bf I-P}\}f\right\|^2_\sigma;
\end{eqnarray}
\item[(ii).] For $k=N_0-1$, there exists an interactive energy functional $G^{N_0-1}_f(t)$ satisfying
$$
G^{N_0-1}_f(t)\lesssim\left\|\nabla^{N_0-2}[f,E,B]\right\|^2+\left\|\nabla^{N_0-1}[f,E,B]\right\|^2
+\left\|\nabla^{N_0}[f,E]\right\|^2
$$
such that
\begin{eqnarray}\label{Lemma3.4-2}
&&\frac{d}{dt}G^{N_0-1}_f(t) +\left\|\nabla^{N_0-2}[E,\rho_+-\rho_-]\right\|_{H^1}^2+\left\|\nabla^{N_0-1}B
\right\|^2+\left\|\nabla^{N_0}{\bf P}f\right\|^2\nonumber\\
&\lesssim&\left\{\left\|\Lambda^{-\frac12}[f,E,B]\right\|^2+\mathcal{E}_{N}(t)\right\}\left(\left\|\nabla^{N_0-1}[E,B]\right\|^2
+\left\|\nabla^{N_0-1}f\right\|^2_\sigma\right)+\frac1{{\varepsilon^2}}\left\|\nabla^{N_0-2}\{{\bf I-P}\}f\right\|^2_\sigma\nonumber\\
&&+\frac1{{\varepsilon^2}}\left\|\nabla^{N_0-1}\{{\bf I-P}\}f\right\|^2_\sigma
+\frac1{{\varepsilon^2}}\left\|\nabla^{N_0}\{{\bf I-P}\}f\right\|^2_\sigma.
\end{eqnarray}
\end{itemize}
\end{lemma}
\begin{proof}
Applying the similar way as Lemma \ref{lemma4.3} and Lemma \ref{mac-dissipation}, we can prove this lemma, we omit the proof for brevity.
\end{proof}
{\bf Assumption 2:}
\[\mathcal{\bar{E}}_{N_0,l}(t)\leq M_2,\]
where $M_2$ is a sufficiently small positive constant.
\begin{proposition}\label{Lemma4.5}
Under {\bf Assumption 1} and {\bf Assumption 2}, assume $N_0\geq 3$, $N=2N_0$, $\hat{l}_1\geq N_0$, $l\geq N+\hat{l}_1$,
there exist an energy functional $\mathcal{E}^k_{N_0}(t)$ and the corresponding energy dissipation rate functional $\mathcal{D}^k_{N_0}(t)$ satisfying \eqref{E-k} and \eqref{D-k} respectively such that
\begin{equation}\label{Lemma4.5-1}
\frac{d}{dt}\mathcal{E}^k_{N_0}(t)+\mathcal{D}^k_{N_0}(t)\leq 0
\end{equation}
{\color{red}holds for $k=0,1,2,\cdots, N_0-2$ and all $0\leq t\leq T.$}

Furthermore, we can get that
\begin{equation}\label{Lemma4.5-2}
\mathcal{E}^k_{N_0}(t)\lesssim\sup_{0\leq \tau\leq t}\left\{\mathcal{\bar{E}}_{N_0,\frac{k+\frac12}{2}}(\tau),\mathcal{E}_{N_0+k+\frac12}(\tau)\right\}(1+t)^{-(k+\frac12)},\quad 0\leq t\leq T.
\end{equation}
\end{proposition}
\begin{proof} From Lemma \ref{lemma4.3} and \ref{Lemma4.4}, we can obtain by a suitable linear combination of the corresponding estimates obtained there that there exist an energy functional $\mathcal{E}^k_{N_0}(t)$ and the corresponding energy dissipation rate functional $\mathcal{D}^k_{N_0}(t)$ satisfying \eqref{E-k} and \eqref{D-k} respectively such that
\begin{equation*}
\frac{d}{dt}\mathcal{E}^k_{N_0}(t)+\mathcal{D}^k_{N_0}(t)\leq 0
\end{equation*}
{\color{red}holds for all $0\leq t\leq T$.}

For this purpose, we first deduce from Lemma \ref{lemma2.2} and Corollary \ref{corrollary} that
\begin{equation*}
\begin{aligned}
\left\|\nabla^k[{\bf P}f,E,B]\right\|\leq \left\|\nabla^{k+1}[{\bf P}f,E,B]\right\|^{\frac{k+\frac12}{k+\frac32}}
\left\|\Lambda^{-\frac12}[{\bf P}f,E,B]\right\|^{\frac{1}{k+\frac32}}.
\end{aligned}
\end{equation*}
The above inequality together with the facts that
\begin{eqnarray}
\left\|\nabla^m{\bf\{I-P\}}f\right\|&\leq& \left\|\langle v\rangle^{-\frac12}\nabla^m{\bf\{I-P\}}f\right\|^{\frac{k+\frac12}{k+\frac32}}
\left\|\langle v\rangle^{-\frac{(\gamma+2){k+\frac12}}{2}}\nabla^m{\bf\{I-P\}}f\right\|^{\frac{1}{k+\frac32}},\nonumber\\
\left\|\nabla^{N_0}[E,B]\right\|&\lesssim&\left\|\nabla^{N_0-1}[E,B]\right\|^\frac{k+\frac12}{k+\frac32}
\left\|\nabla^{N_0+k+\frac12}[E,B]\right\|^\frac{1}{k+\frac32}
\end{eqnarray}
imply
\begin{equation*}
\begin{aligned}
\mathcal{E}^k_{N_0}(t)\leq \left(\mathcal{D}^k_{N_0}(t)\right)^\frac{k+\frac12}{k+\frac32}\left\{\sup_{0\leq \tau\leq t}\left\{\mathcal{\bar{E}}_{N_0,\frac{k+\frac12}{2}}(\tau),\mathcal{E}_{N_0+k+\frac12}(\tau)\right\}\right\}^\frac{1}{k+\frac32}.
\end{aligned}
\end{equation*}
Hence, we deduce that
\begin{equation*}
\frac{d}{dt}\mathcal{E}^k_{N_0}(t)+\left\{\sup_{0\leq \tau\leq t}\left\{\mathcal{\bar{E}}_{N_0,\frac{k+\frac12}{2}}(\tau),\mathcal{E}_{N_0+k+\frac12}(\tau)\right\}\right\}^{-\frac{1}{k+\frac12}}
\left\{\mathcal{E}^k_{N_0}(t)\right\}^{1+\frac{1}{k+\frac12}}\leq 0
\end{equation*}
and we can get by solving the above inequality directly that
\begin{equation*}
\mathcal{E}^k_{N_0}(t)\lesssim\sup_{0\leq \tau\leq t}\left\{\mathcal{\bar{E}}_{N_0,\frac{k+\frac12}{2}}(\tau),\mathcal{E}_{N_0+k+\frac12}(\tau)\right\}(1+t)^{-{k+\frac12}}.
\end{equation*}
This completes the proof of Lemma \ref{Lemma4.5}.
\end{proof}

\section{The estimates on the negative Sobolev space}
To ensure {\bf Assumption} 2, this section is devoted into
bound on $\mathcal{\bar{E}}_{N_0,l}(t)$, especially the bound on $\left\|\Lambda^{-\frac12}[f,E,B](t)\right\|$.

The first one is on the estimate on $\|[f,E,B](t)\|_{\dot{H}^{-s}}$.
\begin{lemma}\label{Lemma4.1} Under the assumptions stated above, we have for $s\in[\frac12,\frac32)$ that
\begin{eqnarray}\label{Lemma4.1-1}
&&\frac{d}{dt}\left(\left\|\Lambda^{-\frac12}f\right\|^2+\left\|\Lambda^{-\frac12}[E,B]\right\|^2\right)+\frac1{\varepsilon^2}\left\|\Lambda^{-\frac12}\{{\bf I-P}\}f\right\|_{\sigma}^2\nonumber\\
&\lesssim&\left\|\Lambda^{-\frac12}f\right\|\left(\left\|\Lambda^{\frac12}E\right\|^2+\left\|\Lambda^{\frac12}f\right\|^2_{\sigma}\right)\nonumber\\
&&+\bar{\mathcal{E}}_{N_0,N_0}(t)\left(\left\|\Lambda^{\frac12}E\right\|^2+\left\|\nabla_xf\right\|^2_{\sigma}
+\left\|\Lambda^{\frac12}B\right\|^2+\frac1{\varepsilon^2}\left\|\Lambda^{\frac12}{\bf\{I-P\}}f\right\|^2_{\sigma}\right).
\end{eqnarray}
\end{lemma}
\begin{proof}
By taking Fourier transform of the first equation of \eqref{VML-drop-eps} with respect to $x$,  multiplying the resulting identity by $|\xi|^{-2{s}}\bar{\hat{f}}_\pm$ with $\bar{\hat{f}}_\pm$ being the complex conjugate of $\hat{f}_\pm$, and integrating the final result with respect to $\xi$ and $v$ over $\mathbb{R}^3_\xi\times\mathbb{R}^3_v$ that
we can get by using the coercivity property $\mathscr{L}$
\begin{eqnarray}\label{4.6}
&&\frac{d}{dt}\left(\left\|\Lambda^{-\frac12}f\right\|^2+\left\|\Lambda^{-\frac12}[E,B]\right\|^2\right)+\frac1{\varepsilon^2}\left\|\Lambda^{-\frac12}\{{\bf I-P}\}f\right\|_{\sigma}^2\nonumber\\
&\lesssim&\left|\left(\mathcal{F}[{q_0} E\cdot\nabla_{ v  }f]\mid|\xi|^{-1}\hat{f}\right)\right|+\left|\left(v  \cdot\mathcal{F}[{q_0} E f]\mid|\xi|^{-1}\hat{f}\right)\right|\nonumber\\
&&+\frac1{\varepsilon}\left|\left(\mathcal{F}[{q_0} v\times B\cdot\nabla_{ v  }f]\mid|\xi|^{-1}\hat{f}\right)\right|+\frac1{\varepsilon}\left|\left(\mathcal{F}[{ \mathscr{T}}(f,f)]\mid|\xi|^{-1}\hat{f}\right)\right|.
\end{eqnarray}
 One has by macro-micro decomposition
\begin{eqnarray}
&&\left(\mathcal{F}[E\cdot\nabla_{ v  }f]\mid|\xi|^{-1}\hat{f}\right)\nonumber\\
&=&\left(\mathcal{F}[E\cdot\nabla_{ v  }{\bf P}f]\mid|\xi|^{-1}\mathcal{F}[{\bf P}f]\right)
+\left(\mathcal{F}[E\cdot\nabla_{ v  }{\bf P}f]\mid|\xi|^{-1}\mathcal{F}[{\bf \{I-P\}}f]\right)\nonumber\\
&&+\left(\mathcal{F}[E\cdot\nabla_{ v  }{\bf \{I-P}f]\mid|\xi|^{-1}\mathcal{F}[{\bf P}f]\right)+\left(\mathcal{F}[E\cdot\nabla_{ v  }{\bf \{I-P\}}f]\mid|\xi|^{-1}\mathcal{F}[{\bf \{I-P\}}f]\right).
\end{eqnarray}
From Lemma \ref{lemma2.2}, Lemma \ref{lemma2.3}, Corollary \ref{corrollary}, and Lemma \ref{lemma2.4}, one can deduce that
the first three terms on the right-hand side of the above equation can be bounded by
\begin{eqnarray}\label{E-s}
&\lesssim&\left\|\Lambda^{-\frac12}\left(E\cdot {M}^{\delta}f\right)\right\|\left\|\Lambda^{-\frac12}\left({M}^{\delta}f\right)\right\|
\lesssim\left\|E\cdot {M}^{\delta}f\right\|_{L_x^{\frac32}}\left\|\Lambda^{-\frac12}\left({M}^{\delta}f\right)\right\|
\nonumber\\
&\lesssim&\|E\|_{L_x^{3}}\left\|{M}^{\delta}f\right\|_{L_x^{3}}
\left\|\Lambda^{-\frac12}\left({M}^{\delta}f\right)\right\|\lesssim\left\|\Lambda^{\frac12}E\right\|\left\|\Lambda^{\frac12}\left({M}^{\delta}f\right)\right\|\left\|\Lambda^{-\frac12}\left({M}^{\delta}f\right)\right\|
\nonumber\\
&\lesssim&\left\|\Lambda^{-\frac12}f\right\|\left(\left\|\Lambda^{\frac12}E\right\|^2+\left\|\Lambda^{\frac12}f\right\|^2_{\sigma}\right).
\end{eqnarray}
As for the last term, one has
\begin{eqnarray}\label{E-mic-s}
  &&\left(\mathcal{F}[E\cdot\nabla_{ v  }{\bf \{I-P\}}f]\mid|\xi|^{-1}\mathcal{F}[{\bf \{I-P\}}f]\right)\nonumber\\
  &\lesssim&\left\|\Lambda^{-\frac12}\left(E {\bf \{I-P\}}f \langle v\rangle^{\frac32}\right)\right\| \left\|\Lambda^{-\frac12}\left({\bf\{I-P\}}f\right)\right\|_{\sigma}\nonumber\\
  &\lesssim&\left\|E {\bf \{I-P\}}f \langle v\rangle^{\frac32}\right\|_{L_x^{\frac32}}\left\|\Lambda^{-\frac12}{\bf\{I-P\}}f\right\|_{\sigma}\nonumber\\
  &\lesssim&\left(\|E\|_{L_x^{6}}
\left\|{\bf \{I-P\}}f \langle v\rangle^{\frac32}\right\|\right)^2+\eta\left\|\Lambda^{-\frac12}{\bf\{I-P\}}f\right\|^2_{\sigma}\nonumber\\
&\lesssim&\left\|{\bf \{I-P\}}f \langle v\rangle^{\frac32}\right\|^2\left\|\Lambda^{\frac12}E\right\|^2
+\eta\left\|\Lambda^{-\frac12}{\bf\{I-P\}}f\right\|^2_{\sigma}.
\end{eqnarray}
Consequently, one has
\begin{eqnarray}
&&\left|\left(\mathcal{F}[E\cdot\nabla_{ v  }f]\mid|\xi|^{-1}\hat{f}\right)\right|\nonumber\\
&\lesssim&\left\|\Lambda^{-\frac12}f\right\|\left(\left\|\Lambda^{\frac12}E\right\|^2+\left\|\Lambda^{\frac12}f\right\|^2_{\sigma}\right)\nonumber\\
&&+\left\|{\bf \{I-P\}}f \langle v\rangle^{\frac32}\right\|^2\left\|\Lambda^{\frac12}E\right\|^2
+\eta\left\|\Lambda^{-\frac12}{\bf\{I-P\}}f\right\|^2_{\sigma}.
\end{eqnarray}
Similarly,
\begin{eqnarray}
  &&\left|\left(v  \cdot\mathcal{F}[E f]\mid|\xi|^{-1}\hat{f}\right)\right|\nonumber\\
&\lesssim&\left\|\Lambda^{-\frac12}f\right\|\left(\left\|\Lambda^{\frac12}E\right\|^2+\left\|\Lambda^{\frac12}f\right\|^2_{\sigma}\right)\nonumber\\
&&+\left\|{\bf \{I-P\}}f \langle v\rangle^{\frac32}\right\|^2\left\|\Lambda^{\frac12}E\right\|^2
+\eta\left\|\Lambda^{-\frac12}{\bf\{I-P\}}f\right\|^2_{\sigma}.
\end{eqnarray}
For the third term on the right-hand side of \eqref{4.6}, we have by repeating the argument used in deducing the estimate on the first two terms that
\begin{eqnarray}
&&\frac1{\varepsilon}\left|\left(\mathcal{F}[{q_0} v\times B\cdot\nabla_{ v  }f]\mid|\xi|^{-1}\hat{f}\right)\right|\nonumber\\
&\lesssim&\frac1{\varepsilon}\left|\left(\mathcal{F}[{q_0} v\times B\cdot\nabla_{ v  }{\bf P}f]\mid|\xi|^{-1}{\bf P}\hat{f}\right)\right|\nonumber\\
&&+\frac1{\varepsilon}\left|\left(\mathcal{F}[{q_0} v\times B\cdot\nabla_{ v  }{\bf P}f]\mid|\xi|^{-1}{\bf \{I-P\}}\hat{f}\right)\right|\nonumber\\
&&+\frac1{\varepsilon}\left|\left(\mathcal{F}[{q_0} v\times B\cdot\nabla_{ v  }{\bf \{I-P\}}f]\mid|\xi|^{-1}{\bf P}\hat{f}\right)\right|\nonumber\\
&&+\frac1{\varepsilon}\left|\left(\mathcal{F}[{q_0} v\times B\cdot\nabla_{ v  }{\bf \{I-P\}}f]\mid|\xi|^{-1}{\bf \{I-P\}}\hat{f}\right)\right|\nonumber\\
&=&\frac1{\varepsilon}\left|\left(\mathcal{F}[{q_0} v\times B\cdot\nabla_{ v  }{\bf P}f]\mid|\xi|^{-1}{\bf \{I-P\}}\hat{f}\right)\right|\nonumber\\
&&+\frac1{\varepsilon}\left|\left(\mathcal{F}[{q_0} v\times B\cdot\nabla_{ v  }{\bf \{I-P\}}f]\mid|\xi|^{-1}{\bf P}\hat{f}\right)\right|\nonumber\\
&&+\frac1{\varepsilon}\left|\left(\mathcal{F}[{q_0} v\times B\cdot\nabla_{ v  }{\bf \{I-P\}}f]\mid|\xi|^{-1}{\bf \{I-P\}}\hat{f}\right)\right|
\end{eqnarray}
where we used the fact that
\[\frac1{\varepsilon}\left|\left(\mathcal{F}[{q_0} v\times B\cdot\nabla_{ v  }{\bf P}f]\mid|\xi|^{-1}{\bf P}\hat{f}\right)\right|=0.\]
Similar with the argument used in \eqref{E-s}, one has
\begin{eqnarray}
&&\frac1{\varepsilon}\left|\left(\mathcal{F}[{q_0} v\times B\cdot\nabla_{ v  }{\bf P}f]\mid|\xi|^{-1}{\bf \{I-P\}}\hat{f}\right)\right|\nonumber\\
&&+\frac1{\varepsilon}\left|\left(\mathcal{F}[{q_0} v\times B\cdot\nabla_{ v  }{\bf \{I-P\}}f]\mid|\xi|^{-1}{\bf P}\hat{f}\right)\right|\nonumber\\
&\lesssim&\left\|\Lambda^{-\frac12}f\right\|\left(\left\|\Lambda^{\frac12}B\right\|^2
+\frac1{\varepsilon^2}\left\|\Lambda^{\frac12}{\bf\{I-P\}}f\right\|^2_{\sigma}\right)
+\left\|f\right\|^2\left\|\Lambda^{\frac12}B\right\|^2
\nonumber\\
&&+\frac\eta{\varepsilon^2}\left\|\Lambda^{-\frac12}{\bf\{I-P\}}f\right\|^2_{\sigma}.
\end{eqnarray}
By using the similar way as \eqref{E-mic-s}, one has
\begin{eqnarray}
 && \frac1{\varepsilon}\left|\left(\mathcal{F}[{q_0} v\times B\cdot\nabla_{ v  }{\bf \{I-P\}}f]\mid|\xi|^{-1}{\bf \{I-P\}}\hat{f}\right)\right|\nonumber\\
 &\lesssim&\left\|{\bf \{I-P\}}f \langle v\rangle^{\frac52}\right\|^2\left\|\Lambda^{\frac12}B\right\|^2
+ \frac\eta{\varepsilon^2}\left\|\Lambda^{-\frac12}{\bf\{I-P\}}f\right\|^2_{\sigma},
\end{eqnarray}
consequently, one  has
\begin{eqnarray}\label{B-s-neg}
  &&\frac1{\varepsilon}\left|\left(\mathcal{F}[{q_0} v\times B\cdot\nabla_{ v  }f]\mid|\xi|^{-1}\hat{f}\right)\right|\nonumber\\
&\lesssim&\left\|\Lambda^{-\frac12}f\right\|\left(\left\|\Lambda^{\frac12}B\right\|^2+\frac1{\varepsilon^2}\left\|\Lambda^{\frac12}{\bf\{I-P\}}f\right\|^2_{\sigma}\right)\nonumber\\
&&+\left\|f\langle v\rangle^{\frac52}\right\|^2\left\|\Lambda^{\frac12}B\right\|^2
+\frac\eta{\varepsilon^2}\left\|\Lambda^{-\frac12}{\bf\{I-P\}}f\right\|^2_{\sigma}\nonumber\\
&\lesssim&\bar{\mathcal{E}}_{N_0,N_0}(t)\left(\left\|\Lambda^{\frac12}B\right\|^2+\frac1{\varepsilon^2}\left\|\Lambda^{\frac12}{\bf\{I-P\}}f\right\|^2_{\sigma}\right)
+\frac\eta{\varepsilon^2}\left\|\Lambda^{-\frac12}{\bf\{I-P\}}f\right\|^2_{\sigma}.
\end{eqnarray}
As to the last term on the right-hand side of \eqref{4.6}, applying the similar argument as Proposition 1 in \cite{Strain-Zhu-2012-ARMA}, one has
\begin{eqnarray*}
&&\frac1{\varepsilon}\left(\mathcal{F}[{\mathscr{T}}(f,f)]\mid|\xi|^{-1}\mathcal{F}\left[{{\bf\{I-P\}}f}\right]\right)\\
&=&\frac1{\varepsilon}\left(\mathcal{F}\left[\partial_i(\{\Phi^{ij}\ast[{M}^{1/2}f]\}\partial_j f)-\{\Phi^{ij}\ast[v_i{M}^{1/2}f]\}\partial_j f\right.\right.\\
&&\left.\left.-\partial_i(\{\Phi^{ij}\ast[{M}^{1/2}\partial_jf]\} f)+\{\Phi^{ij}\ast[v_i{M}^{1/2}\partial_jf]\} f\right],|\xi|^{-1}\mathcal{F}\left[{{\bf\{I-P\}}f}\right]\right)\\
&=&-\frac1{\varepsilon}\left(\mathcal{F}\left[\{\Phi^{ij}\ast[{M}^{1/2}f]\}\partial_j f\right]\mid|\xi|^{-1}\mathcal{F}\left[{{\bf\{I-P\}}f}\right]\right)\\
&&
+\frac1{\varepsilon}\left(\mathcal{F}\left[\{\Phi^{ij}\ast[{M}^{1/2}\partial_jf]\} f\right]\mid|\xi|^{-1}\mathcal{F}\left[{{\bf\{I-P\}}f}\right]\right)\\
&&-\frac1{\varepsilon}\left(\mathcal{F}\left[\{\Phi^{ij}\ast[v_i{M}^{1/2}f]\}\partial_j f\right]\mid|\xi|^{-1}\mathcal{F}\left[{{\bf\{I-P\}}f}\right]\right)\\
&&
+\frac1{\varepsilon}\left(\mathcal{F}\left[\{\Phi^{ij}\ast[v_i{M}^{1/2}\partial_jf]\} f\right]\mid|\xi|^{-1}\mathcal{F}\left[{{\bf\{I-P\}}f}\right]\right)\nonumber\\
&\lesssim&\left\|\left|{M}^\delta f\right|^2_2\left|f\right|_\sigma^2\right\|_{L_x^{\frac32}}
+\frac\eta{\varepsilon^2}\left\|\Lambda^{-\frac12}{\bf\{I-P\}}f\right\|^2_\sigma\nonumber\\
&\lesssim&\left\|{M}^\delta f\right\|^2\left\|f\right\|^2_{{L_x^{6}}L^2_\sigma}+\frac\eta{\varepsilon^2}\left\|\Lambda^{-\frac12}{\bf\{I-P\}}f\right\|^2_\sigma\nonumber\\
&\lesssim&\left\|{M}^\delta f\right\|^2\left\|\nabla_xf\right\|^2_{\sigma}+\frac\eta{\varepsilon^2}\left\|\Lambda^{-\frac12}{\bf\{I-P\}}f\right\|^2_\sigma,
\end{eqnarray*}

Substituting the estimates  into (\ref{4.6}) yields \eqref{Lemma4.1-1}, which complete the proof of Lemma $\ref{Lemma4.1}$.
\end{proof}
\begin{lemma}\label{Lemma4.2}
There exists an interactive functional $G^{-\frac12}_{f}(t)$ satisfying
\begin{equation*}
G^{-\frac12}_{f}(t)\lesssim \left\|\Lambda^{\frac12}[f,E,B]\right\|^2+\left\|\Lambda^{-\frac12}[f,E,B]\right\|^2+\|\Lambda^{\frac32}E\|^2
\end{equation*}
such that
\begin{eqnarray}\label{Lemma4.2-1}
&&\frac{d}{dt}G^{-\frac12}_{f}(t)+\left\|\Lambda^{\frac12}{\bf P}f\right\|^2+\left\|\Lambda^{\frac12}[E,B]\right\|^2+\left\|\Lambda^{-\frac12}E\right\|_{H^1_x}^2+\left\|\Lambda^{-\frac12}(\rho_+-\rho_-)\right\|^2_{H^1_x}\nonumber\\
&\lesssim&\frac1{\varepsilon^2}\left\|\Lambda^{-\frac12}\{{\bf I-P}\}f\right\|^2_{H^2_xL^2_\sigma}+\mathcal{E}_{2}(t)\mathcal{D}_{2}(t)
\end{eqnarray}
holds for any $0\leq t\leq T$.
\end{lemma}
\begin{proof}
This lemma can be proved by a similar way as Lemma \ref{mac-dissipation}, we omit its proof for brevity.
\end{proof}
Based on the above two lemmas, one has
\begin{proposition}\label{prop-neg}
Under Under {\bf Assumption 1} and {\bf Assumption 2}, one has
\begin{eqnarray}\label{negative-estimate}
&&\frac{d}{dt}\left(\left\|\Lambda^{-\frac12}f\right\|^2+\left\|\Lambda^{-\frac12}[E,B]\right\|^2+\kappa_1G^{-\frac12}_{f}(t)\right)+\frac1{\varepsilon^2}\left\|\Lambda^{-\frac12}\{{\bf I-P}\}f\right\|_{\sigma}^2\nonumber\\
&&+\kappa_1\left\|\Lambda^{\frac12}{\bf P}f\right\|^2+\kappa_1\left\|\Lambda^{\frac12}[E,B]\right\|^2+\kappa_1\left\|\Lambda^{-\frac12}E\right\|_{H^1_x}^2+\kappa_1\left\|\Lambda^{-\frac12}(\rho_+-\rho_-)\right\|^2_{H^1}\nonumber\\
&\lesssim&\bar{\mathcal{E}}_{N_0,N_0}(t)\left\|\nabla_xf\right\|^2_{\sigma}.
\end{eqnarray}
\end{proposition}
\begin{proof}
  By the interpolation with respect to the spatial derivatives, one has
\begin{eqnarray}
  &&\frac1{\varepsilon^2}\left\|\Lambda^{\frac12}{\bf\{I-P\}}f\right\|^2_{\sigma}\nonumber\\
  &\lesssim&\frac1{\varepsilon^2}\left\|\nabla_x{\bf\{I-P\}}f\right\|^2_{\sigma}+
  \frac\eta{\varepsilon^2}\left\|\Lambda^{-\frac12}{\bf\{I-P\}}f\right\|^2_{\sigma}
\end{eqnarray}
In addition,
\begin{eqnarray}
  &&\left\|\Lambda^{-\frac12}f\right\|\left\|\Lambda^{\frac12}{\bf \{I-P\}}f\right\|^2_{\sigma}\nonumber\\
  &\lesssim&\left\|\Lambda^{-\frac12}f\right\|^2\left\|\Lambda^{\frac12}{\bf \{I-P\}}f\right\|^2_{\sigma}+
  \eta\left\|\Lambda^{\frac12}{\bf \{I-P\}}f\right\|^2_{\sigma}\nonumber\\
  &\lesssim&\left\|\Lambda^{-\frac12}f\right\|^2\left\{\left\|\nabla_x{\bf\{I-P\}}f\right\|^2_{\sigma}+
  \eta\left\|\Lambda^{-\frac12}{\bf\{I-P\}}f\right\|^2_{\sigma}\right\}+\eta\left\|\Lambda^{\frac12}{\bf \{I-P\}}f\right\|^2_{\sigma}.
\end{eqnarray}
By using the above estimates, one can deduce  \eqref{negative-estimate} from a proper linear combination of \eqref{Lemma4.1-1} and \eqref{Lemma4.2-1}.
\end{proof}

\section{The a priori estimates}
Based on Proposition \ref{prop1} and \ref{prop-neg}, we are ready to construct the a priori estimates
\begin{equation}\label{Def-a-priori}
  \sup_{0<t\leq T}\left\{\left\|\Lambda^{-\frac12}[f,E,B]\right\|^2+(1+t)^{-\frac{1+\epsilon_0}2}\sum_{|\alpha|=N}{\varepsilon^2}\left\|w_l(\alpha)\partial^\alpha f\right\|^2+\mathcal{E}_{N-1,l}(t)+\mathcal{E}_{N}(t)\right\}\leq \overline{M}
\end{equation}
where $ \overline{M}$ is a sufficiently small positive constant.

One has
\begin{eqnarray}
&&\frac{d}{dt}\left\{\left\|\Lambda^{-\frac12}[f,E,B]\right\|^2+(1+t)^{-\frac{1+\epsilon_0}2}\sum_{|\alpha|=N}{\varepsilon^2}\left\|w_l(\alpha)\partial^\alpha f\right\|^2+\mathcal{E}_{N-1,l}(t)+\mathcal{E}_{N}(t)\right\}\nonumber\\
&&
+\mathcal{D}_{N}(t)+\mathcal{D}_{N-1,l}(t)\lesssim 0,
\end{eqnarray}
which gives that
\begin{eqnarray}
  &&\left\|\Lambda^{-\frac12}[f,E,B](t)\right\|^2+(1+t)^{-\frac{1+\epsilon_0}2}\sum_{|\alpha|=N}{\varepsilon^2}\left\|w_l(\alpha)\partial^\alpha f(t)\right\|^2+\mathcal{E}_{N-1,l}(t)+\mathcal{E}_{N}(t)\nonumber\\
  &\lesssim&\left\|\Lambda^{-\frac12}[f,E,B](0)\right\|^2+\sum_{|\alpha|=N}{\varepsilon^2}\left\|w_l(\alpha)\partial^\alpha f(0)\right\|^2+\mathcal{E}_{N-1,l}(0)+\mathcal{E}_{N}(0)\nonumber\\
  &\lesssim&Y^2_{f_\varepsilon,E_\varepsilon,B_\varepsilon}(0),
\end{eqnarray}
where $Y_{f_\varepsilon,E^\varepsilon,B^\varepsilon}(0)$ is defined in \eqref{Def-Y_0},
thus we close the a priori estimates if the initial data $Y_0$ is taken as small sufficiently.
Furthermore, we can get that
\begin{equation}
\mathcal{E}^k_{N_0}(t)\lesssim Y^2_{f_\varepsilon,E_\varepsilon,B_\varepsilon}(0)(1+t)^{-(k+\frac12)},\quad 0\leq t\leq T.
\end{equation}

\section{Limit to two fluid incompressible Navier-Stokes-Fourier-Maxwell equations with Ohm's law}
\label{Sec:Limits}
In this section, we will derive the two fluid incompressible NSFM equations \eqref{INSFM-Ohm} with Ohm's law from the perturbed two-species VML as $\eps \rightarrow 0$.

As \cite{Jiang-Luo-2022-Ann.PDE}, We first introduce the following fluid variables
\begin{equation}\label{Fluid-Quanities}
  \begin{aligned}
    \rho_\eps = \tfrac{1}{2} \langle {f}_\eps , {q_2} \sqrt{M} \rangle_{L^2_v} \,, \ u_\eps = \tfrac{1}{2} \langle f_\eps , {q_2} v \sqrt{M} \rangle_{L^2_v} \,, \ \theta_\eps = \tfrac{1}{2} \langle f_\eps , {q_2} ( \tfrac{|v|^2}{3} - 1 ) \sqrt{M} \rangle_{L^2_v} \,, \\
    n_\eps = \langle f_\eps , {q_1} \sqrt{M} \rangle_{L^2_v} \,,\ j_\eps = \tfrac{1}{\eps} \langle f_\eps , {q_1} v \sqrt{M} \rangle_{L^2_v} \,,\ w_\eps = \tfrac{1}{\eps} \langle f_\eps , {q_1} ( \tfrac{|v|^2}{3} - 1 )\sqrt{M} \rangle_{L^2_v} \,.
  \end{aligned}
\end{equation}
We use the similar argument as \eqref{Macro-equation1}, we can deduce the following local conservation laws
	\begin{equation}\label{Local-Consvtn-Law}
	  \left\{
	    \begin{array}{l}
	      \partial_t \rho_\eps + \tfrac{1}{\eps} \div_x \, u_\eps = 0 \,, \\
	      \partial_t u_\eps + \tfrac{1}{\eps} \nabla_x ( \rho_\eps + \theta_\eps ) + \div_x \, \big\langle \widehat{A} (v) \sqrt{M}\cdot {q_2} , \tfrac{1}{\eps} \mathscr{L} ( \tfrac{f_\eps }{2} ) \big\rangle_{L^2_v} = \tfrac{1}{2} ( n_\eps E_\eps + j_\eps \times B_\eps ) \,, \\
	      \partial_t \theta_\eps + \tfrac{2}{3} \tfrac{1}{\eps} \div_x \, u_\eps + \tfrac{2}{3} \div_x \, \big\langle \widehat{B} (v) \sqrt{M}\cdot {q_2} , \tfrac{1}{\eps} \mathscr{L} ( \tfrac{f_\eps}{2} ) \big\rangle_{L^2_v} = \tfrac{\eps}{3} j_\eps \cdot E_\eps \,, \\
	      \partial_t n_\eps + \div_x \, j_\eps = 0 \,, \\
	      \partial_t E_\eps - \nabla_x \times B_\eps = - j_\eps \,, \\
	      \partial_t B_\eps + \nabla_x \times E_\eps = 0 \,, \\
	      \div_x \, E_\eps = n_\eps \,, \quad \div_x B_\eps = 0 \,.
	    \end{array}
	  \right.
	\end{equation}
where $\mathscr{L}[\widehat{A} (v) \sqrt{M}\cdot {q_2}]=\left(v\otimes v-\frac{|v|^2}{3}I_3\right)\sqrt{M}\cdot {q_2}\in\ker^{\bot}({\mathscr{L}})$ with $\widehat{A} (v) \sqrt{M}\cdot {q_2}\in\ker^{\bot}({\mathscr{L}})\cdot {q_2}$
and $\mathscr{L}[\widehat{B} (v) \sqrt{M}]=\left(v\left(\frac{|v|^2}2-\frac52\right)\right)\sqrt{M}\cdot {q_2}\in\ker^{\bot}({\mathscr{L}})$ with $\widehat{B} (v) \sqrt{M}\cdot {q_2}\in\ker^{\bot}({\mathscr{L}})$.

Based on Theorem \ref{Main-Thm-1}, the Cauchy problem to \eqref{VML-F} admits a global solution
$(f_\epsilon,E_\epsilon,B_\epsilon)$ belonging to $L^\infty(\mathbb{R}_+;H^N_xL^2_v)$, one also has
\begin{eqnarray}\label{limits-1}
&&\frac{d}{dt}\left\{(1+t)^{-\frac{1+\epsilon_0}2}\sum_{|\alpha|=N}{\varepsilon^2}\left\|w_l(\alpha)\partial^\alpha f_\varepsilon\right\|^2+\mathcal{E}_{N-1,l}(t)+\mathcal{E}_{N}(t)\right\}\nonumber\\
&&+\mathcal{D}_{N}(t)+\mathcal{D}_{N-1,l}(t)
\lesssim0,
\end{eqnarray}
then
\begin{eqnarray}\label{limits-2}
  &&\sup_{t\geq0}\left\{(1+t)^{-\frac{1+\epsilon_0}2}\sum_{|\alpha|=N}{\varepsilon^2}\left\|w_l(\alpha)\partial^\alpha f_\varepsilon\right\|^2+\mathcal{E}_{N-1,l}(t)+\mathcal{E}_{N}(t)\right\}\nonumber\\
  &\leq&\sum_{|\alpha|=N}{\varepsilon^2}\left\|w_l(\alpha)\partial^\alpha f_\varepsilon(0)\right\|^2+\mathcal{E}_{N-1,l}(0)+\mathcal{E}_{N}(0)\leq C
\end{eqnarray}
and
\begin{eqnarray}\label{limits-3}
  \int^\infty_0\left\{\mathcal{D}_{N}(t)+\mathcal{D}_{N-1,l}(t)\right\}dt\leq C
\end{eqnarray}
where $C$ is independent of $\varepsilon$. In fact, \eqref{limits-3} tells us that
\begin{eqnarray}\label{limits-4}
\sum_{|\alpha|\leq N}\int^\infty_0\left\|\partial^\alpha\{{\bf I-P}\}f_\varepsilon\right\|^2_\sigma dt\lesssim C\varepsilon^2.
\end{eqnarray}
which yields that
\begin{equation}
\partial^\alpha\{{\bf I-P}\}f_\varepsilon\rightarrow 0, \textrm{\ strongly in $L^2(\mathbb{R}_+;H^N_xL^2_v)$ as $\varepsilon\rightarrow0$.}
\end{equation}
By standard convergent method, there exist $f,E,B,\rho,u,\theta,n,w,j$ such that
\begin{eqnarray}\label{limit-weak}
  f_\epsilon&\rightarrow& f,\textrm{weakly$-*$ for $t>0$, weakly in $H^N_x$},\nonumber\\
  E_\epsilon&\rightarrow& E,\textrm{weakly$-*$ for $t>0$, weakly in $H^N_x$},\nonumber\\
  B_\epsilon&\rightarrow& B,\textrm{weakly$-*$ for $t>0$, weakly in $H^N_x$},\nonumber\\
  \rho_\epsilon&\rightarrow& \rho,\textrm{weakly$-*$ for $t>0$, weakly in $H^N_x$},\nonumber\\
  u_\epsilon&\rightarrow& u,\textrm{weakly$-*$ for $t>0$, weakly in $H^N_x$},\nonumber\\
  \theta_\epsilon&\rightarrow& \theta,\textrm{weakly$-*$ for $t>0$, weakly in $H^N_x$},\nonumber\\
  n_\epsilon&\rightarrow& n,\textrm{weakly$-*$ for $t>0$, weakly in $H^N_x$},\nonumber\\
  w_\epsilon&\rightarrow& w,\textrm{weakly in $L^2(\mathbb{R}^+;H^N_xL^2_v)$},\nonumber\\
  j_\epsilon&\rightarrow& j,\textrm{weakly in $L^2(\mathbb{R}^+;H^N_xL^2_v)$}.
\end{eqnarray}
In the sense of distributions, utilizing the uniform estimates \eqref{limits-1}, \eqref{limits-2}, \eqref{limits-3} and \eqref{limits-4}, applying Aubin-Lions-Simon Theorem and the similar argument as \cite{Jiang-Luo-2022-Ann.PDE}, we can deduce that
$$(u, \theta, n, E, B) \in C(\R^+; H^{N-1}_x ) \cap L^\infty (\R^+; H^N_x) $$
satisfy the following two fluid incompressible NSFM equations with Ohm's law
\begin{equation*}
  \left\{
    \begin{array}{l}
      \partial_t u + u \cdot \nabla_x u - \mu \Delta_x u + \nabla_x p = \tfrac{1}{2} ( n E + j \times B ) \,, \qquad \div_x \, u = 0 \,, \\ [2mm]
      \partial_t \theta + u \cdot \nabla_x \theta - \kappa \Delta_x \theta = 0 \,, \qquad\qquad\qquad\qquad\qquad\quad\ \, \rho + \theta = 0 \,, \\ [2mm]
      \partial_t E - \nabla_x \times B = - j \,, \qquad\qquad\qquad\qquad\qquad\qquad\ \ \ \, \div_x \, E = n \,, \\ [2mm]
      \partial_t B + \nabla_x \times E = 0 \,, \qquad\qquad\qquad\qquad\qquad\qquad\qquad \div_x \, B = 0 \,, \\ [2mm]
      \qquad \qquad j - nu = \sigma \big( - \tfrac{1}{2} \nabla_x n + E + u \times B \big) \,, \qquad\quad\,\ w = \tfrac{3}{2} n \theta \,,
    \end{array}
  \right.
\end{equation*}
with initial data
\begin{equation*}
  \begin{aligned}
    u (0,x) = \mathcal{P} u^{in} (x) \,, \ \theta (0,x) = \tfrac{3}{5} \theta^{in} (x) - \tfrac{2}{5} \rho^{in} (x) \,, \ E(0,x) = E^{in} (x) \,, \ B(0,x) = B^{in} (x) \,.
  \end{aligned}
\end{equation*}
We omit its detail proof for brevity.
Moreover, from the uniform bound \eqref{Main-Thm-1-1} in Theorem \ref{Main-Thm-1} and the convergence \eqref{limit-weak}, we have
\begin{eqnarray}
   &&\sup_{t \geq 0} \big( \|f \|^2_{H^N_{x}L^2_v} + \| E \|^2_{H^N_x} + \| B \|^2_{H^N_x} \big) (t)\nonumber\\
    & \leq& \sup_{t \geq 0} \big( \| f_\eps \|^2_{H^N_{x}L^2_v} + \| E_\eps \|^2_{H^N_x} + \| B_\eps \|^2_{H^N_x} \big) (t) \nonumber\\
    & \lesssim&Y^2_{f_\varepsilon,E_\varepsilon,B_\varepsilon}(0)\rightarrow Y^2_{f,E,B}(0)
\end{eqnarray}
as $\eps \rightarrow 0$. Hence
\begin{equation*}
  \begin{aligned}
    &&\sup_{t \geq 0} \big( \|f \|^2_{H^N_{x}L^2_v} + \| E \|^2_{H^N_x} + \| B \|^2_{H^N_x} \big) (t) \lesssim Y^2_{f,E,B}(0) \lesssim \ell_0 \,.
  \end{aligned}
\end{equation*}
There are positive generic constants $C_h$ and $C_l$ such that
\begin{equation*}
  \begin{aligned}
    C_l \big( \| u \|^2_{H^N_x} + \| \theta \|^2_{H^N_x} + \| n \|^2_{H^N_x} \big) \leq \| f \|^2_{H^N_{x}L^2_v} \leq C_h \big( \| u \|^2_{H^N_x} + \| \theta \|^2_{H^N_x} + \| n \|^2_{H^N_x} \big) \,.
  \end{aligned}
\end{equation*}
Consequently, the solution $(u,\theta,n,E,B)$ to the two fluid incompressible NSFM equations \eqref{INSFM-Ohm} with Ohm's law constructed above admits the energy bound
\begin{equation*}
  \begin{aligned}
    \sup_{t \geq 0} \big( & \| u \|^2_{H^N_x} + \| \theta \|^2_{H^N_x} + \| n \|^2_{H^N_x} + \| E \|^2_{H^N_x} + \| B \|^2_{H^N_x} \big) (t) \lesssim Y^2_{f,E,B}(0) \lesssim \ell_0.
  \end{aligned}
\end{equation*}
Then the proof of Theorem \ref{Main-Thm-2} is complete.

\appendix

\section{Appendix}
\begin{lemma}\label{Lemma2.1} It holds that
\begin{itemize}
\item[(i).] (\cite{Wang-12}) Let $w_{l}$ be given in $(\ref{weight})$, there exist positive constants $\eta>0$ and $C_\eta>0$ such that
\begin{equation}\label{L_0}
  \left\langle w_{l}^2(t, v){ L}f,f\right\rangle\geq \eta|f|_{\sigma,w_{l}}^2-C_{\eta}\left|\chi_{\{| v|\leq2C_{\eta}\}}f\right|^2.
\end{equation}
Here $\chi_{\{| v|\leq2C_{\eta}\}}$ denotes the characteristic function of the set $\{v=( v_1,  v_2,  v_3) \in {\mathbb{R}}^3: | v|\leq2C_{\eta}\}$.
\item[(ii).](\cite{Wang-12}) Let $w_{l}$ be given in $(\ref{weight})$, then we have
\begin{equation}\label{Gamma-w}
  \left\langle w_{l}^2(t, v)\partial^{\alpha}_{\beta}\Gamma(g_1,g_2),\partial^{\alpha}_{\beta}g_3\right\rangle
  \lesssim \sum_{\alpha_1\leq \alpha\atop \bar{\beta}\leq\beta_1\leq\beta}
 \left|\mu^{\delta}\partial^{\alpha_1}_{\bar{\beta}}g_1\right|_2
  \left|\partial^{\alpha-\alpha_1}_{\beta-\beta_1}g_2\right|_{\sigma,w_{l}}
  \left(\left|\partial^{\alpha}_{\beta}g_3\right|_{\sigma,w_{l}}
  +l\left|\partial^{\alpha}_{\beta}g_3\right|_{2,\frac{w_{l}}{\langle v\rangle^{-\gamma/2}}}\right).
\end{equation}
Here $\delta>0$ is a sufficiently small universal constant. In particular, we have
\begin{equation}\label{Gamma-0}
  \left\langle\Gamma(g_1,g_2),g_3\right\rangle
  \lesssim \left|\mu^{\delta}g_1\right|_2\left|g_2\right|_{\sigma}\left|g_3\right|_{\sigma}.
\end{equation}
\end{itemize}
\end{lemma}
In what follows, we will collect some inequalities which will be used throughout the rest of this manuscript. The first one is the Sobolev interpolation inequalities.
\begin{lemma}\label{lemma2.2}(\cite{Wang-12})
Let $2\leq p<\infty$ and $k,\ell, m\in\mathbb{R}$, we have
\begin{equation}
\left\|\nabla^k f\right\|_{L^p}\lesssim\left\|\nabla^\ell f\right\|^{\theta}\left\|\nabla^m f\right\|^{1-\theta},
\end{equation}
where $0\leq \theta\leq1$ and $\ell$ satisfy
\begin{equation}
\frac{1}{p}-\frac k3=\left(\frac12-\frac\ell3\right)\theta+\left(\frac12-\frac m3\right)(1-\theta).
\end{equation}
For the case $p=+\infty$, we have
\begin{equation}
\left\|\nabla^k f\right\|_{L^\infty}\lesssim\left\|\nabla^\ell f\right\|^{\theta}\left\|\nabla^m f\right\|^{1-\theta},
\end{equation}
where $0\leq \theta\leq1$ and $\ell$ satisfy
\begin{equation}
-\frac k3=\left(\frac12-\frac\ell3\right)\theta+\left(\frac12-\frac m3\right)(1-\theta).
\end{equation}
Here we need to assume that $\ell\leq k+1$, $m\geq k+2$. Throughout the rest of this manuscript, for each positive integer $k$,
$\left\|\nabla^k f\right\|\sim\sum\limits_{|\alpha|=k}\left\|\partial^\alpha f\right\|.$
\end{lemma}
For the $L^q-L^q$ type estimate on $\Lambda^{-s}f$, we have
\begin{lemma}\label{lemma2.3}(\cite{Stein-1970})
Let $0<s<3$, $1<p<q<\infty$, $\frac1q+\frac s3=\frac1p$, then we have
\begin{equation}
\|\Lambda^{-s}f\|_{L^q}\lesssim\|f\|_{L^p}.
\end{equation}
\end{lemma}

Based on Lemma \ref{lemma2.2} and Lemma \ref{lemma2.3}, we have the following corollary which will be used frequently later
\begin{corollary}\label{corrollary}
Let $0<s<\frac 32$, then for any positive integer $k$ and any nonnegative integer $j$ satisfying $0\leq j\leq k$, we have
$$
\|f\|_{L^{3}_x}\lesssim \left\|\Lambda^{\frac34-\frac s2}f\right\|,\quad \|f\|_{L^{\frac 3s}_x}\leq \left\|\Lambda^{\frac12}f\right\|,
$$
$$
\|f\|_{L_x^\infty}\lesssim\left\|\Lambda^{-\frac12}f\right\|^{\frac{2k-1}{2(k+1+s)}}
\left\|\nabla^{k+1}f\right\|^{\frac{3+2s}{2(k+1+s)}},
$$
$$
\left\|\nabla^j f\right\|_{L^6_x}\lesssim \left\|\Lambda^{-\frac12}f\right\|^{\frac{k-j}{k+1+s}} \left\|\nabla^{k+1}f\right\|^{\frac{j+s+1}{k+1+s}},
$$
and
$$
\left\|\nabla^jf\right\|_{L^3_x}\lesssim \left\|\Lambda^{-\frac12}f\right\|^{\frac{2k-2j+1}{2k+2+2s}} \left\|\nabla^{k+1}f\right\|^{\frac{2j+2s+1}{2k+2+2s}}.
$$
\end{corollary}
In many places, we will use Minkowski's integral inequality to interchange the orders of integration over $x$ and $v$.
\begin{lemma}\label{lemma2.4}(\cite{Guo-Wang-CPDE-2012})
For $1\leq p\leq q\leq \infty$, we have
\begin{equation}
\|f\|_{L^q_xL^p_v}\leq \|f\|_{L^p_vL^q_x}.
\end{equation}
\end{lemma}


\section*{Acknowledgment}

\bigskip

\bibliography{reference}

\end{document}